\documentclass[11pt]{amsart}
\usepackage{fullpage}
\usepackage{amsfonts,amssymb,amsmath,amsthm}
\usepackage{mathtools}
\usepackage{hyperref}
\usepackage{bm}
\usepackage{mdwlist}

\DeclarePairedDelimiter\bra{\langle}{\rvert}
\DeclarePairedDelimiter\ket{\lvert}{\rangle}
\DeclarePairedDelimiterX\braket[2]{\langle}{\rangle}{#1\,\delimsize\vert\,\mathopen{}#2}

\theoremstyle{plain}
\newtheorem{theorem}{Theorem}[section]

\newtheorem{lemma}[theorem]{Lemma}
\newtheorem{proposition}[theorem]{Proposition}
\theoremstyle{definition}
\newtheorem{definition}[theorem]{Definition}

\theoremstyle{remark}
\newtheorem{remark}[theorem]{Remark}
\numberwithin{equation}{section}

\usepackage{mathrsfs}

\def \bP {\mathbb P}

\def \bS {\mathbb S}

\def \bU {\mathbb U}

\def \cA {\mathcal A}
\def \cB {\mathcal B}
\def \cC {\mathcal C}
\def \cD {\mathcal D}
\def \cE {\mathcal E}
\def \cF {\mathcal F}
\def \cG {\mathcal G}
\def \cH {\mathcal H}
\def \cI {\mathcal I}

\def \cK {\mathcal K}
\def \cL {\mathcal L}
\def \cM {\mathcal M}

\def \cO {\mathcal O}

\def \cR {\mathcal R}
\def \cS {\mathcal S}
\def \cR {\mathcal R}
\def \cR {\mathcal R}

\def \fg {\mathfrak g}
\def \fh {\mathfrak h}

\def \fr {\mathfrak r}

\def \fM {\mathfrak M}

\def \fU {\mathfrak U}

\def\R{{\mathbb R}}
\def\C{{\mathbb C}}
\def\N{{\mathbb N}}
\def\Eng{{\mathbb E}}
\def\Heis{{\mathbb H}}
\mathchardef\mhyphen="2D 

\def \fghat{\widehat{G}}
\def \Hhat{\widehat{\Heis}}
\def \Enghat{\widehat{\Eng}}

\def \Op  {{\rm Op}}

\def \vect{\rm Vect}
\def \Exp{{\rm Exp}}
\def \Hpi{\cH_{\pi}}
\def \eps{\epsilon}

\def\restriction#1#2{\mathchoice
              {\setbox1\hbox{${\displaystyle #1}_{\scriptstyle #2}$}
              \restrictionaux{#1}{#2}}
              {\setbox1\hbox{${\textstyle #1}_{\scriptstyle #2}$}
              \restrictionaux{#1}{#2}}
              {\setbox1\hbox{${\scriptstyle #1}_{\scriptscriptstyle #2}$}
              \restrictionaux{#1}{#2}}
              {\setbox1\hbox{${\scriptscriptstyle #1}_{\scriptscriptstyle #2}$}
              \restrictionaux{#1}{#2}}}
\def\restrictionaux#1#2{{#1\,\smash{\vrule height .8\ht1 depth .85\dp1}}_{\,#2}}

\def\Tend#1#2{\mathop{\longrightarrow}\limits_{#1\rightarrow#2}}

\numberwithin{equation}{section}

\begin{document}
\title[]{Quantum-classical correspondence and obstruction to dispersion on the Engel group}
\author[L. Benedetto]{Lino Benedetto}
\address[L. Benedetto]{DMA, École normale supérieure, Université PSL, CNRS, 75005 Paris, France \& Univ Angers, CNRS, LAREMA, SFR MATHSTIC, F-49000 Angers, France} 
\email{lbenedetto@dma.ens.fr}

\begin{abstract} 
 In this paper, we develop the semiclassical analysis of the lowest dimensional simply connected nilpotent Lie group of step 3, called the Engel group and denoted by $\Eng$, in the continuity of works for groups of step 2 realized in \cite{FF2} and \cite{FFF}, while 
 exhibiting new features specific to groups of higher step. 
 We are interested in the propagation of the semiclassical measures associated to solutions of the Schrödinger equation for the canonical subLaplacian $\Delta_\Eng$ at different time-scales $\tau\in\R_{>0}^+$.
 In particular, for $\tau=1$ we recover a quantum-classical correspondence where we observe the fundamental role of abnormal extremal lifts of the Engel group and are able to discuss the 
 speed of the propagation of the singularities.
 Furthermore, in order to understand the dispersive nature of the subLaplacian, we are led to develop a second-microlocal analysis on particular cones in our phase space. This is done by relating the 
 quasi-contact structure of the group $\Eng$ to its semiclassical analysis via harmonic analysis.
 As a consequence, we are able to prove obstruction to local smoothness-type estimates and Strichartz estimates for the Engel subLaplacian.
\end{abstract}

\keywords{Abstract harmonic analysis, 
Semiclassical analysis on nilpotent Lie groups, Second-microlocal semiclassical measures}

\maketitle

\makeatletter
\renewcommand\l@subsection{\@tocline{2}{0pt}{3pc}{5pc}{}}
\makeatother

\tableofcontents

\section{Introduction}

\subsection{The Engel group and its subLaplacian} 
We consider the lowest dimensional connected and simply connected nilpotent Lie group of step 3, called the Engel group and denoted by $\Eng$. 
Its Lie algebra $\fg_{\Eng}$ is the four-dimensional Lie algebra linearly spanned by the vectors $X_1,X_2,X_3$ and $X_4$
satisfying the following bracket relations:
\begin{equation}
    \label{eq:bracketrel}
    [X_1,X_2] = X_3,\ [X_1,X_3] = X_4.
\end{equation}
The Lie algebra $\fg_\Eng$ is stratitified when one considers the following decomposition:
\begin{equation*}
    \fg_{\Eng} = \fg_{1} \oplus \fg_{2} \oplus \fg_{3},
\end{equation*}
where $\fg_{1} = \vect(X_1,X_2)$, $\fg_{2}$ is spanned by $X_3$ and $\fg_{3}$ is the center of the Lie algebra, spanned by $X_4$. 
We also introduce the notation $$\fr_2 = \vect(X_2).$$

The Engel group is, up to isomorphism, the only connected and simply connected Lie group admitting $\fg_{\Eng}$ as its Lie algebra. 
In what follows we endow the vector space $\fg_\Eng$ with the scalar product for which the basis $(X_1,X_2,X_3,X_4)$ is orthonormal.
We thus give the Engel group $\Eng$ the structure of a Carnot group, see \cite{Mon3}.

As a connected and simply connected nilpotent Lie group, the exponential map denoted by
\begin{equation*}
    \Exp_{\Eng}: \fg_{\Eng} \longrightarrow \Eng,
\end{equation*}
defines a diffeomorphism between the Lie algebra and $\Eng$. The pushforward of the Lebesgue measure on $\fg_\Eng$ via $\Exp_\Eng$ defines a bi-invariant Haar measure on $\Eng$ that we 
will simply denote by $dx$ when a variable $x\in\Eng$ is considered.

\subsubsection*{Homogeneous structure} The stratification of $\fg_{\Eng}$ leads us to consider the family of dilations defined by $r \cdot X = r^{i} X$ for every $X\in\fg_i$, with $i\in\{1,2,3\}$ and $r>0$ (see \cite{FR},\cite{FS}).
The integers $i\in\{1,2,3\}$ are called the weights and will sometimes be denoted by $\upsilon_i$.

The associated group dilations are defined by
\begin{equation*}
    r\cdot x := \Exp_\Eng(r \cdot X),\quad x=\Exp_\Eng(X)\in\Eng,\ r>0.
\end{equation*}
These families of dilations on $\fg_{\Eng}$ and $\Eng$ form one-parameter groups of automorphisms of the Lie algebra $\fg_{\Eng}$ and of the group $\Eng$.
The Jacobian of the dilation associated to $r>0$ is $r^Q$ where
\begin{equation*}
    Q:={\rm dim}\, \fg_{1} +2{\rm dim}\, \fg_{2} +3{\rm dim}\, \fg_{3} = 7
\end{equation*}
is called the {\it homogeneous dimension} of $\Eng$.

\subsubsection*{SubLaplacian and Schrödinger equation}
We are interested in the study of the following differential operator 
\begin{equation*}
    \Delta_{\Eng} = X_{1}^2 + X_{2}^2,
\end{equation*}
called the \textit{Engel subLaplacian}, and more precisely in the associated semiclassical Schrödinger equations
\begin{equation}
    \label{eq:SchEngel}
    \begin{cases}
        & i\hbar^{\tau}\partial_{t}\psi^{\hbar} = -\hbar^2 \Delta_{\Eng}\psi^{\hbar},\\
        & \psi_{|t=0}^{\hbar} = \psi_{0}^{\hbar}\in L^2(\Eng),
    \end{cases}
\end{equation}
where we have introduced a small semiclassical parameter $\hbar>0$ and a time-scale $\tau\in\R_{>0}$. As $\Delta_{\Eng}$ is essentially self-adjoint on $C_{0}^{\infty}(\Eng)$ (see \cite{FR}), 
using Stone's theorem we have existence and unicity of the solution to \eqref{eq:SchEngel} for any initial datum in $L^2(\Eng)$. We will denote by $e^{it\Delta_{\Eng}}$, $t\in\R$, the 
associated propagator.

The use of a semiclassical parameter $\hbar$ can be understood as a way of measuring the scale of the oscillations of our initial data: indeed, we will often assume that the family of initial data $(\psi_{0}^\hbar)_{\hbar>0}$
is \textit{$\hbar$-oscillating}, i.e satisfies the following property:
\begin{equation}
    \label{eq:assumptionoscillation}
        \limsup_{\hbar\rightarrow 0}\|{\bm 1}_{-\hbar^2\Delta_\Eng >R}\,\psi_{0}^\hbar\|_{L^2(\Eng)}\Tend{R}{+\infty} 0,
\end{equation}
meaning the oscillations are of the scale $1/\hbar$ when measured with respect to the subLaplacian $\Delta_\Eng$. We highlight the fact that measuring oscillations of a family $(\psi_{0}^\hbar)_{\hbar>0}$
with respect to the subLaplacian is compatible with the homogeneous structure of $\Eng$ and is thus anisotropic: we allow for higher scales of oscillations 
in directions transverse to the first stratum.

\subsubsection*{Time-scales}
As said before, the parameter $\tau$ corresponds to a choice of time-scale. The first natural one is $\tau = 1$, called the {\it semiclassical time-scale}, through which we recover 
a quantum-classical correspondence. More precisely, we wish to find the relevant dynamic underlying the quantum evolution as $\hbar$ goes to 0. For example, in the Riemannian setting and considering the Laplace-Beltrami operator,
this dynamic is the one of the geodesic flow associated to the metric. On the Engel group, we expect to observe flows related to its subRiemannian structure, about which more will be discussed below.
 
The other natural choice is $\tau=2$, where Equation \eqref{eq:SchEngel} is then equivalent to the standard Schrödinger equation.
We expect this time-scale, along with assumption \ref{eq:assumptionoscillation} on the family of initial data, to give us insights on the dispersive nature of the subLaplacian.

We will study the Schrödinger equation for both of these time-scales by the introduction of an adapted phase space different from the cotangent bundle $T^* G$ and of operator-valued measures capturing the concentration loci of our solutions in the asymptotics
$\hbar\rightarrow 0$.

\subsection{Phase space analysis on nilpotent Lie groups}
Our approach relies on the introduction of a particular phase space on the Engel group based on its harmonic analysis
and on the use of an associated pseudodifferential calculus.
Such a strategy has already been implemented on several nilpotent Lie groups: one can refer to \cite{FF2},\cite{FF3} for the study of the quantum evolution of the subLaplacian on H-type groups, to \cite{FL} for the case of Heisenberg nilmanifolds with a view towards observability
and to \cite{FFF} for the study of quantum limits on general 2-step nilmanifolds.
For a systematic development of this approach for general graded nilpotent Lie groups, we refer to the monograph \cite{FR} and to the articles \cite{FF},\cite{BFF}. However, this article is 
the first one implementing this approach for a group of step 3.

\subsubsection*{Unitary dual set of a nilpotent Lie group}
We will consider the unitary dual of a nilpotent Lie group $G$, denoted by $\widehat{G}$, as the natural replacement for impulsion space.

First, recall that a \emph{unitary representation} $(\pi,\mathcal H_\pi)$ of a group $G$ is a pair consisting in a Hilbert space~$\mathcal H_\pi$  and a group morphism~$\pi$ from~$G$ to the set of unitary operators on $\mathcal H_\pi$.
In this paper, the representations will always be assumed strongly continuous, and their associated Hilbert spaces separable. 
A representation is said to be {\it irreducible} if the only closed subspaces of $\mathcal H_\pi$ that are stable under~$\pi$ are $\{0\}$ and $\mathcal H_\pi$ itself. 
Two representations $\pi_1$ and $\pi_2$ are equivalent  if there exists a unitary transform $\cI$ called an {\it intertwining map} that sends $\mathcal H_{\pi_1}$ on $\mathcal H_{\pi_2}$ with 
$$\pi_1=\cI^{-1}\circ  \pi_2 \circ \cI.$$ 
The {\it dual set} $\widehat G$ is obtained by taking the quotient of the set of irreducible representations by this equivalence relation.

\begin{definition}[\cite{CG}]
    The set of all irreducible unitary representations, up to equivalence, is called the unitary dual and denoted by $\widehat{G}$:
    \[
        \widehat{G} := \{ \left[\pi\right]\,:\, \pi: G \rightarrow \cL(\Hpi) ~\text{irreducible s.c. unitary representation of}~ G \},
    \]
    where $\left[\pi\right]$ denotes the equivalence class of the representation $\pi$. This set is naturally endowed with a topology, called the Fell topology, and a standard Borel space structure.
    Furthermore, there exists a unique (up to a scalar) canonical measure $\mu_{\widehat{G}}$ on $\widehat{G}$, called the Plancherel measure.
\end{definition}
We can then define the Fourier transform $\cF_G$ on the group $G$, sending integrable functions on $G$ to fields of operators on the dual $\widehat{G}$ acting on the representations' spaces $\cH_\pi$, see Section \ref{sect:preliEngel}.

\subsubsection*{Semiclassical quantization and semiclassical measures}
We will work with the measured space $G\times\widehat{G}$ as our natural phase space. The observables associated to this phase space are the measurable fields of operators 
\begin{equation*}
    \sigma = \{\sigma(x,\pi)\in\cL(\cH_\pi)\,:\, (x,\pi)\in G \times\widehat{G}\}.
\end{equation*}
For a certain class of observables regular enough that we will denote by $\cA_0$, we define a quantization procedure based on the Fourier transform on the group $G$ turning these observables into operators on $L^2(G)$:
\begin{equation*}
    \Op_{\hbar,G}(\sigma): L^2(G)\rightarrow L^2(G),
\end{equation*} 
(see Section \ref{sect:semiclassEng}). 

For a bounded family $(\psi_{0}^\hbar)_{\hbar>0}$ of functions in $L^2(G)$, we can consider the following quadratic quantities
\begin{equation*}
    \sigma\in\cA_0 \mapsto \left(\Op_{\hbar,G}(\sigma)\psi_{0}^\hbar,\psi_{0}^\hbar\right)_{L^2(G)}\in \C,
\end{equation*}
and more precisely, the limit of such quantities as $\hbar \rightarrow 0$. Such a limit allows us to understand the concentration loci of 
the family $(\psi_{0}^\hbar)_{\hbar>0}$ in phase space, for the scale of oscillations $\hbar^{-1}$. 
In a similar way as in the Euclidean setting, such limits can be described by
positive measures on the phase space $G\times\widehat{G}$, or more precisely, in order to take into account the operator nature of our phase space, by operator-valued measures $\Gamma d\gamma$ (see Section \ref{sect:semiclassEng}):
such an object is the data of a positive finite measure $\gamma$ on $G\times\widehat{G}$ and of a field of trace-class operators $(x,\pi)\in G\times\widehat{G}\mapsto \Gamma(x,\pi)\in \cL^1(\cH_\pi)$ integrable with respect to $d\gamma$.
Indeed, up to an extraction $\hbar_k \Tend{k}{+\infty} 0$, we can write
\begin{equation*}
    \left(\Op_{\hbar_k,G}(\sigma)\psi_{0}^{\hbar_k},\psi_{0}^{\hbar_k}\right)_{L^2(G)} \Tend{\hbar_k}{0} \int_{G\times\widehat{G}}{\rm Tr}_{\cH_\pi}\left(\sigma(x,\pi)\Gamma(x,\pi)\right)d\gamma(x,\pi).
\end{equation*}

\subsection{Generic representations of the Engel group and Montgomery operators}
The unitary dual set $\Enghat$ of the Engel group can be explicity computed, similarly as in any other nilpotent Lie group. We refer to Section \ref{sect:preliEngel} for a complete description.
We describe here only the support of the Plancherel measure, called the set of generic representations and denoted by $\Enghat_{\rm gen}$. It can be parametrized by two parameters
$(\delta,\beta)$ in $\fg_{3}^*\setminus\{0\}\times \fr_{2}^*$, and is homeomorphic to $\R\setminus\{0\}\times\R$. In the equivalence class given by $(\delta,\beta)$, we can exhibit a particular representation that we denote by $\pi^{\delta,\beta}$ (see Proposition \ref{prop:dualset}).

As the subLaplacian $\Delta_\Eng$ is invariant by translations of the group $\Eng$, it can be seen as a Fourier multiplier for the Fourier transform of the group $\Eng$.
For a generic representation $\pi^{\delta,\beta}\in\Enghat_{\rm gen}$, we introduce the following notation for the Fourier transform of $-\Delta_{\Eng}$.
\begin{equation*}
    H(\pi^{\delta,\beta}) = \cF_{\Eng}(-\Delta_\Eng)(\pi^{\delta,\beta})= -\partial_{\xi}^2 + \left(\beta + \frac{\delta}{2} \xi^2 \right)^2.
\end{equation*}

Following \cite{HL}, the resolvents of these differential operators are compact and as such their spectrums are discrete:
we will denote by $\mu_{n}(\delta,\beta)$, $n\in\N_{>0}$, their eigenvalues in increasing order. We also introduce the following notation: 
for $n\in\N_{>0}$, we write $\Pi_n$ the field of eigenprojectors on $\Enghat_{\rm gen}$ satisfying
\begin{equation*}
    \forall \pi^{\delta,\beta}\in\Enghat_{\rm gen},\ H(\pi^{\delta,\beta})\Pi_{n}(\pi^{\delta,\beta}) = \mu_{n}(\delta,\beta)\Pi_{n}(\pi^{\delta,\beta}).
\end{equation*}

This family of operators can actually be rescaled to a one-parameter family of differential operators: for $\nu\in\R$, we introduce the operator
\begin{equation*}
    \Tilde{H}(\nu) := -\partial_{\Tilde{\xi}}^2 + \left(\nu + \frac{1}{2} \Tilde{\xi}^2 \right)^2.
\end{equation*}
This family of differential operators, called the \textit{Montgomery operators}, has been studied by numerous authors (see \cite{HL}, \cite{HP},\cite{Mon}) and
is related to the Fourier transform of $-\Delta_\Eng$ in the following way: for $(\delta,\beta)\in\fg_{3}^{*}\setminus\{0\}\times\fr_{2}^*$, through the rescaling $\xi = \delta^{-1/3}\Tilde{\xi}$, $H(\pi^{\delta,\beta})$ is unitarily equivalent to the operator
\begin{equation*}
    \delta^{2/3}\Tilde{H}(\beta\delta^{-1/3}) = \delta^{2/3}\left(\partial_{\Tilde{\xi}}^2 + \left(\beta\delta^{-1/3} + \frac{1}{2} \Tilde{\xi}^2 \right)^2\right).
\end{equation*}
Similarly as above, we can consider the ordered discrete spectrum $\{\Tilde{\mu}_{n}(\nu)\,:\,n\in\N_{>0}\}$ of $\Tilde{H}(\nu)$, and we deduce from the previous rescaling 
the following relation between the eigenvalues of these operators: for all $n\in\N_{>0}$,
\begin{equation}
    \label{eq:releigenvalues}
    \mu_n(\delta,\beta) = \delta^{2/3}\Tilde{\mu}_{n}(\beta\delta^{-1/3}).
\end{equation}

We observe that the family $(\Tilde{H}(\nu))_{\nu\in\R}$ is an analytic family of type B (see \cite{Kato}). We deduce that for each $n\in\N_{>0}$, $\nu\in\R\mapsto \Tilde{\mu}_{n}(\nu)$ is analytic and thus 
the maps $\mu_n$ are smooth with respect to the parameters $(\delta,\beta) \in \fg_{3}^*\setminus\{0\}\times \fr_{2}^*$. 

\subsection{Semiclassical time-scale and quantum-classical correspondence}
Our first result concerns the propagation properties of the semiclassical measures associated to a solution 
$(\psi^{\hbar}(t))_{\hbar>0}$ of the Schrödinger equation \eqref{eq:SchEngel} for the choice $\tau=1$ and for a bounded family of initial data $(\psi_{0}^\hbar)_{\hbar>0}\in L^2(\Eng)$.
We are then able to relate the semiclassical measures of the solution at time $t$ with the one of the initial condition through simple propagation law.

\begin{theorem}
    \label{theo:propagintro}
    Let $(\psi_0^\hbar)_{\hbar>0}$ be a bounded family in $L^2(\Eng)$ and $\psi^\hbar(t) = {\rm e}^{i \hbar t \Delta_{\Eng}} \psi^\hbar_0$ the solution to \eqref{eq:SchEngel} for $\tau=1$.
    Then any time-averaged semiclassical measure 
    $t\mapsto \Gamma_t d\gamma_t\in L^\infty(\R_t, {\mathcal M}_{ov}^+(\Eng\times \Enghat))$ for the family $(\psi^\hbar)_{\hbar>0}$ satisfies the following additional properties:
    \begin{itemize}
    \item[(i)] For $d\gamma_t$-almost all generic representations $ \Enghat_{\rm gen} = \{\pi^{\delta,\beta}\,:\,(\delta,\beta) \in \fg_{3}^*\setminus\{0\} \times \fr_{2}^{*}\}$, we have the following decomposition 
    \begin{equation*}
    \Gamma_t(x,\pi^{\delta,\beta})=\sum_{n\in\N_{>0}} \Gamma_{n,t}(x,\pi^{\delta,\beta})\;\;{ with}\;\; \Gamma_{n,t}(x,\pi^{\delta,\beta}):= \Pi_n(\pi^{\delta,\beta})\Gamma_t(x,\pi^{\delta,\beta}) \Pi_n(\pi^{\delta,\beta}) \in \cL(\cH_{\pi^{\delta,\beta}}),
    \end{equation*}
    where the $\Pi_n$ are the measurable symbols given by the spectral projections of $H$ for the eigenvalues $\mu_{n}$, $n\in\N_{>0}$. 
    \item[(ii)] For $n\in \N_{>0}$, the distribution $d\gamma_{n,t}(x,\pi^{\delta,\beta})dt = {\rm Tr}\left(\Gamma_{n,t}(x,\pi^{\delta,\beta})\right) d\gamma_t(x,\pi^{\delta,\beta})dt$ on $\R_t\times \Eng\times (\fg_{3}^*\setminus\{0\}\times \fr_{2}^*)$ is a weak solution of the following transport equation:
    \begin{equation*}
    \left(\partial_t -\partial_{\beta}\mu_{n}(\delta,\beta)X_{2}\right)\left(d\gamma_{n,t}(x,\pi^{\delta,\beta})\right)=0.
    \end{equation*}
    \end{itemize}
\end{theorem} 

\begin{remark}
The last transport equation allows us to determine the time-averaged semiclassical measure at any time $t\in\R_t$ as one can prove 
the continuity of each map $t \mapsto \gamma_{n,t}$, $n\in\N_{>0}$, valued in the space of distributions (see Proposition \ref{prop:continuitytimesm} in Section \ref{sect:semiclassEng}). 
The initial condition at time $t=0$ is then given by any semiclassical measure of the family $(\psi_{0}^{\hbar_k})_{k\in\N}$, where $(\hbar_k)_{k\in\N}$ is the sequence 
converging towards 0 associated to the time-averaged semiclassical measure $t \mapsto \Gamma_t d\gamma_t$.
\end{remark}

\subsection{Engel manifolds, quasi-contact structure and singular curves}
We introduce here the relevant material to understand the geometrical meaning of the objects we have met so far.
Let $M$ denote a manifold. By a $k$-distribution we mean a smooth subbundle $\cD\subset TM$ of rank $k$. We use the notation 
$[\cD,\cD]$ for the sheaf generated by all Lie brackets $[X,Y]$ of sections $X,Y$ of the distribution $\cD$. This way,  
setting $\cD^1 := \cD$, for $j\geq 1$ we set
\begin{equation*}
    \cD^{j+1} := \cD^j + [\cD,\cD^j].
\end{equation*}
\begin{definition}
    A four-dimensional manifold $M$ is an \textit{Engel manifold} when it is endowed with an \textit{Engel distribution}, meaning a 
    2-distribution $\cD$ with the property that $\cD^2$ is a distribution of rank 3 and $\cD^3$ is the whole tangent bundle $TM$.
\end{definition}

The Engel group is then trivially an Engel manifold when endowed with the distribution $\cD_\Eng$ given by the first stratum $\fg_1$ when seen as the subbundle obtained by pushforwards of translations. 

The interest for Engel manifolds comes from the following result.

\begin{proposition}[\cite{Mon3}]
    Let $k$ and $n$ be the rank of a distribution and the dimension of the manifold on which it lives. Then the only stable distributions (in the sense of singularity theory)
    occur for $(k,n) = (1,n), (n-1,n)$ or $(2,4)$. Any stable regular 2-distribution on a manifold of dimension 4 defines an Engel distribution.
\end{proposition}

As this proposition shows, Engel manifolds are a particularity of dimension 4. The other stable manifolds correspond to line fields, contact manifolds and their 
even-dimensional counterpart, quasi-contact manifolds. In this article, we will refer to a four dimensional manifold $M$ endowed with a 
non-integrable distribution of rank 3 as a quasi-contact structure.

Another interesting feature of Engel manifolds is the existence of a Darboux-like theorem.

\begin{theorem}[Engel normal form,\cite{Mon2}]
    Any two Engel manifolds are locally diffeomorphic in the following sense: for any point of an Engel manifold $(M,\cD)$, there exist a neighborhood $U\subset M$, 
    an open set $V\subset \Eng$ and a diffeomorphism $\Phi:U\rightarrow V$ such that 
    \begin{equation*}
        \Phi^*(\cD) = \cD_\Eng.
    \end{equation*}
\end{theorem}

\subsubsection*{Charateristic line field} 
A central object in the study of an Engel manifold $(M,\cD)$ is its characteristic line field $L$, that one can define by the property 
\begin{equation*}
    [L,\cD^2] \equiv 0, \mod \cD^2.
\end{equation*}
Any Engel manifold is then endowed with a canonical flag 
\begin{equation*}
    L\subset \cD \subset \cD^2 \subset TM.
\end{equation*}

Observe that the line field $L$ is a feature of the quasi-contact manifold $(M,\cD^2)$ rather than the Engel structure directly.
It is not difficult to check that on the Engel group $\Eng$, the characteristic line field corresponds to the subbundle $\fr_2$, i.e spanned by the vector field $X_2$.
We write 
\begin{equation*}
    \R_2 = \Exp_\Eng(\fr_2),
\end{equation*}    
for the associated one-dimensional subgroup.

\subsubsection*{Singular curves and comments on Theorem \ref{theo:propagintro}}
One striking property of the characteristic line field is the following: any sufficiently short integral curve $c$ of $L$ is 
a minimizing subRiemannian geodesic between its endpoints, and this holds independently of the choice of the metric put on the 2-distribution $\cD$. Moreover, for a generic 
metric, $c$ is not the projection of any solution to Hamilton's equations for the subRiemannian Hamiltonian (see \cite{Mon3}). The curve $c$ is called 
a \textit{singular curve} on $(M,\cD)$.

Singular curves are a particularity of subRiemannian geometry. These curves admit multiple characterizations but the following microlocal one we give is
straightforward: we use the notation $\cD^\perp$ for the annihilator of $\cD$ (thus a subbundle of the cotangent bundle $T^* M$) and $\Bar{\omega}$ for the restriction
to $\cD^\perp$ of the canonical symplectic form $\omega$ on $T^* M$. 

\begin{definition}
    A characteristic curve for $\cD^\perp$ is an absolutely continuous curve $\lambda(\cdot)\in\cD^\perp$ that never intersects the zero section and that satisfies
    $\dot{\lambda}(t) \in \ker(\Bar{\omega}(\lambda(t)))$ at every point $t$ for which the derivative $\dot{\lambda}$ exists.
\end{definition}

\begin{theorem}[\cite{Mon3}]
    Let $(M,\cD)$ be a subRiemannian manifold. A curve in $M$ is singular if and only if it is the projection of a characteristic $\lambda$ for $\cD^\perp$ with square-integrable derivative; $\lambda$ is then called an 
    abnormal extremal lift of the singular curve.
\end{theorem}

In the setting of rank 2 distribution, one can show that any abnormal extremal lift is in fact a section of $\cD^{2,\perp}\setminus\{0\}$ (see \cite{Mon3}). On the Engel group, this would correspond to a characteristic
valued in the subbundle $\cD_{\Eng}^{2,\perp}\setminus\{0\} = \fg_{3}^* \setminus\{0\}$. 

We observe that the propagation of the semiclassical measures as described in Theorem \ref{theo:propagintro} takes place along 
the abnormal extremal lifts of the characteristic line field $\fr_{2}$ that are constant-valued in $\fg_{3}^*\setminus\{0\}$. 
Moreover, the speed of propagation is entirely determined and depends also of the impulsion variable $\beta\in \fr_{2}^*$. This description 
contrasts with the results obtained so far concerning propagation of singularities for subelliptic operators: in \cite{L22}, it is indeed shown that for the wave 
equation associated to subLaplacians on a subRiemannian manifolds, singularities (in the usual sense of wave front set) can propagate along abnormal extremal lifts but 
nothing has been said in a general setting about their speed of propagation. In \cite{CvDL}, it has been shown that for the special case of flat Martinet distribution, which 
is closely related to the Engel group, there exists a continuum of such speeds. It is remarkable that our analysis, taking advantage of the homogeneous structure of the Engel group and 
by consequence fundamentally anisotropic, is able to isolate the different speeds of propagation with the natural addition of the variable $\beta\in\fr_{2}^*$.

\subsection{Time evolution of microlocally concentrated wave packets}
\label{subsect:wpevolution}

Theorem \ref{theo:propagintro} has shown the importance of particular abnormal extremal lifts of the singular curves in $(\Eng,\cD_\Eng)$ for the quantum evolution. We present here particular solutions to 
Equation \eqref{eq:SchEngel} that are microlocally concentrated on such abnormal extremal lifts up to finite time. This construction is based on the notion of wave packets for Lie groups and is to be 
compared with the particular solutions to the Martinet subelliptic wave equation exhibited in \cite{CvDL}.

Let us first briefly recall the definition of classical Euclidean wave packets. Given $(x_0,\xi_0)\in\R^d\times\R^d$ and $a\in\mathcal S(\R^d)$, we consider the family (indexed by $\hbar$) of functions
\begin{equation*}
    \psi_{\rm eucl}^\hbar(x)= \hbar^{-d/4} a\left(\frac{x-x_0}{\sqrt\hbar}\right) {\rm e}^{\frac i\hbar \xi_0\cdot (x-x_0)}, \;\; x\in\R^d.
\end{equation*}
Such a family is called a Euclidean wave packet.

The oscillation along $\xi_0$ is forced by the term ${\rm e}^{\frac i\hbar \xi_0\cdot(x-x_0)}$ and the concentration on $x_0$ is performed at the scale $\sqrt\hbar$ for symmetry reasons : the $\hbar$-Fourier transform of $\psi^\hbar_{\rm eucl}$,
$\hbar^{-d/2}\widehat \psi^\hbar_{\rm eucl} (\xi/\hbar)$ presents a concentration on $\xi_0$ at the scale $\sqrt\hbar$. 

The following definition takes inspiration from the Euclidean setting, and has already been introduced in previous articles (see \cite{FF2},\cite{FL} and \cite{BFF} where a Wick-calculus on general compact
and graded nilpotent Lie groups is developed thanks to these adapted wave packets).

\begin{definition}
    \label{def:wp}
    Let $a_0 \in C^{\infty}(\Eng)$, $(x_0,\pi^{\delta_0,\beta_0}) \in \Eng \times \Enghat_{\rm gen}$ and $\Phi_1$,$\Phi_2 \in \mathcal{H}_{\pi^{\delta_0,\beta_0}}^{\infty}$, i.e the space of smooth vectors in $\mathcal{H}_{\pi^{\delta_0,\beta_0}}$. 
    We call a wave packet the family of maps $\left(WP_{x_0,\pi^{\beta_0,\delta_0}}^{\hbar}(a,\Phi_1,\Phi_2)\right)_{\hbar >0}$, defined by 
    \begin{equation*}
        \forall x \in \Eng,\ WP_{x_0,\pi^{\delta_0,\beta_0}}^{\hbar}(a_0,\Phi_1,\Phi_2)(x) = \hbar^{-Q/4} a_0(\hbar^{-1/2}\cdot (x_{0}^{-1}x)) \left(\pi^{\delta_0,\beta_0}(\hbar^{-1}\cdot (x_{0}^{-1}x)) \Phi_1, \Phi_2\right).
    \end{equation*}
\end{definition}

In order for a wave packet to be in $L^2(\Eng)$, following \cite{FF}, \cite{P}, it is enough for the profile $a$ to be square-integrable only with respect to the variables that do not belong 
to the jump set $e\subset\{1,2,3,4\}$ associated to the representation $\pi^{\delta_0,\beta_0}$ and the basis $(X_1,X_2,X_3,X_4)$. In our case, $e$ is equal to the subset of indices $\{1,3\}$ and for this reason, we introduce $\fg_{e^c}$ the Lie algebra 
spanned by $X_2$ and $X_4$, and $\R^{2}_{e^c} = \Exp_\Eng(\fg_{e^c})$ the associated Lie subgroup. It is thus enough to take a profile $a$ in $\cS(\R^{2}_{e^c})$ for the associated wave packet to define a bounded family 
in $L^2(\Eng)$.

\begin{theorem}
\label{thm:prop_wave_packet}
    Let $\psi^{\hbar}(t)$ be the solution to the Schrödinger equation \eqref{eq:SchEngel} for $\tau=1$ with initial data being a wave packet:
    \begin{equation*}
        \psi_{0}^{\hbar} = WP_{x_0,\pi^{\delta_0,\beta_0}}^{\hbar}(a_0,\Phi_1,\Phi_2),
    \end{equation*}
    where $(x_0,\pi^{\delta_0,\beta_0})$, $\Phi_1$,$\Phi_2$ as introduced in Definition \ref{def:wp} and $a_0\in \cS(\R^{2}_{e^c})$. 
    Suppose that $\Phi_{1}$ is an eigenvector of $H(\pi^{\delta_0,\beta_0})$ for the $n$-th eigenvalue ($n\in\N_{>0}$):
    \begin{equation*}
        H(\pi^{\delta_0,\beta_0})\Phi_1 = \mu_{n}(\delta_0,\beta_0)\Phi_1.
    \end{equation*}
    Then there exists a time-varying profile $(t,y) \mapsto a(t,y)$ in $C^{1}(\R,\cS(\R^{2}_{e^c}))$ satisfying $a(0,\cdot) = a_0(\cdot)$ and such that 
    \begin{equation*}
        \psi^{\hbar}(t) =  \hbar^{-Q/4}e^{-\frac{i}{\hbar}\mu_{n}(\delta_0,\beta_0)t} a(\hbar^{-1/2}\cdot (x(t)^{-1}x)) \left(\pi^{\delta_0,\beta_0}(\hbar^{-1}\cdot (x_{0}^{-1}x)) \Phi_1, \Phi_2\right) + O_{L^2(\Eng)}(\hbar^{1/2}),
    \end{equation*}
    where
    \begin{equation*}
        x(t) = x_{0}\Exp_{\Eng}(\partial_{\beta}\mu_{n}(\delta_0,\beta_0)t X_{2}),\  t\in\R.
    \end{equation*}
\end{theorem}

\begin{remark}
    This is of course coherent with Theorem \ref{theo:propagintro} as the semiclassical measure of $(\psi_{0}^\hbar)_{\hbar>0}$ is given by the Dirac measure on $(x_0,\pi^{\delta_0,\beta_0})$ and by the projector $\ket{\Phi_1}\bra{\Phi_1}$, see 
    \cite{FF2}, \cite{FF}.
\end{remark}

\subsection{Second microlocalization on the Engel group}
\label{subsect:secondmircointro}

We now investigate the Schrödinger equation for the time-scale $\tau=2$, once again through the use of time-averaged semiclassical measures. 

\subsubsection*{Support condition}
With a similar analysis as for the previous results, we obtain the following proposition. 

\begin{proposition}
    \label{prop:dispersionsupp}
    Let $(\psi^{\hbar}_0)_{\hbar>0}$ be a bounded family in $L^2(\Eng)$ and $\psi^\hbar(t)= e^{it\Delta_\Eng}\psi_{0}^\hbar$ the solution to the Schrödinger equation \eqref{eq:Sch_engel} for $\tau=2$.
    Then any time-averaged semiclassical measure $t\mapsto \Gamma_t d\gamma_t\in L^\infty(\R, {\mathcal M}_{ov}^+(\Eng\times \Enghat))$ for the family $(\psi^\hbar(t))_{\hbar>0}$ and the sequence $(\hbar_k)_{k\in\N}$ satisfies the 
    following additional properties:
    \begin{itemize}
        \item[(i)] Similarly as in Theorem \ref{theo:propagintro}, for $d\gamma_t$-almost all generic representations $\pi^{\delta,\beta}$, the field of operators $\Gamma$ decomposes along the spectral projectors of $H$:
        \begin{equation*}
        \Gamma_t(x,\pi^{\delta,\beta})=\sum_{n\in\N_{>0}} \Gamma_{n,t}(x,\pi^{\delta,\beta})\;\;{ with}\;\; \Gamma_{n,t}(x,\pi^{\delta,\beta}):= \Pi_n(\pi^{\delta,\beta})\Gamma_t(x,\pi^{\delta,\beta}) \Pi_n(\pi^{\delta,\beta}) \in \cL(\cH_{\pi^{\delta,\beta}}).
        \end{equation*}
        \item[(ii)] For $n\in\N_{>0}$ and almost all $t\in\R_t$, the distribution $d\gamma_{n,t}(x,\pi^{\delta,\beta}) = {\rm Tr}\left(\Gamma_{n,t}(x,\pi^{\delta,\beta})\right) d\gamma_t(x,\pi^{\delta,\beta})$ on $\Eng\times (\fg_{3}^*\setminus\{0\}\times \fr_{2}^*)$, we now have the following support condition:
        \begin{equation}
            \label{eq:suppgammaintro}
            {\rm supp}(\gamma_{n,t}) \subset \Eng \times \{\pi^{\delta,\beta}\in\Enghat_{\rm gen}\,:\,\partial_{\beta}\mu_{n}(\delta,\beta) = 0\}.
        \end{equation}
    \end{itemize}
\end{proposition}

Using the relation given by Equation \eqref{eq:releigenvalues} between the eigenvalues of the symbol $H$ and the Montgomery operators, we can write 
\begin{equation*}
    \partial_{\beta}\mu_{n}(\delta,\beta) = \delta^{1/3}\partial_{\nu}\Tilde{\mu}_{n}(\beta\delta^{-1/3}).
\end{equation*}
Thus for any critical point $\nu_0$ of $\Tilde{\mu}_{n}$, the set 
\begin{equation*}
    C_{\nu_0} = \Eng \times \{\pi^{\delta,\beta}\in \Enghat_{\rm gen}\,:\,\beta = \nu_0 \delta^{1/3}\},
\end{equation*}
is contained in the right-hand side's set of Equation \eqref{eq:suppgammaintro}, and more precisely any point in ${\rm supp}(\gamma_{n,t})$ is contained into 
some set $C_{\nu_c}$ for $\nu_c$ a critical point of $\Tilde{\mu}_n$. Remark that these sets are invariant by the dilations on $\Enghat$, and we call such subsets \textit{cones}. 

The question of the existence of non-degenerate critical points of the curves $(\nu\mapsto \Tilde{\mu}_n(\nu))_{n\in\N_{>0}}$ has already been studied and 
we recall here the most recent results concerning such critical points.

\begin{theorem}[\cite{HL}, Theorem 1.5]
    There exists $n_0\in \N_{>0}$ such that for $n=1$ or $n\geq n_0$, $\Tilde{\mu}_n$ admits a unique critical point $\nu_{n,c}$. Moreover, it corresponds to a global minimum, i.e $\Tilde{\mu}_n(\nu_{n,c}) = \min_{\nu\in\R} \Tilde{\mu}_{n}(\nu)$,
    and this minimum is non-degenerate.
\end{theorem}

This result is conjectured to hold for all the eigenvalues of the Montgomery operators. Note also that, for any $n\in\N_{>0}$, $\nu\mapsto\Tilde{\mu}_{n}(\nu)$ admits at least one critical point 
as it diverges to $+\infty$ as $\nu$ goes to $\pm \infty$.

\subsubsection*{Second-microlocal semiclassical measures} 
Contrary to the results obtained for the time-scale $\tau=1$, we cannot here relate the semiclassical measures at time $t\neq 0$ with the one at $t=0$ through propagation laws as the tools introduced so far do not
allow us to go past the previous support condition. For this reason, we develop a second microlocalization analysis on the cones $C_{\nu_0}$ by extending our phase space and our space of symbols (see Section \ref{sect:secondmicro}).
The next result follows from the existence and the propagation of second-microlocal semiclassical measures on the cone $C_{\nu_0}$.

\begin{theorem}
    \label{thm:propag2microintro}
    Using the same notations as in Proposition \ref{prop:dispersionsupp}, let $n\in\N_{>0}$ and $\nu_0$ be a non-degenerate critical point of the map $\nu\in\R\mapsto \Tilde{\mu}_n(\nu)$. Then up to a further extraction 
    of the sequence $(\hbar_k)_{k\in\N}$, there exist a positive measure 
    $\gamma_{n}^{2}$ on $\Heis\times\{\pi^{\delta,\beta}\in \Enghat_{\rm gen}\,:\,\beta = \nu_0 \delta^{1/3}\}$ and a field of trace-class operators 
    \begin{equation*}
        (h,\pi)\in \Heis\times \{\pi^{\delta,\beta}\in\Enghat_{\rm gen}\,:\,\beta=\nu_0\delta^{1/3}\} \mapsto \Gamma_{n,0}^{2}(h,\pi)\in\cL^1(L^2(\R_2)),
    \end{equation*}
    both depending solely on the subsequence $(\psi^{\hbar_k}_0)_{k\in\N}$ and 
    such that the semiclassical measure $t\mapsto \gamma_{n,t}$ of Proposition \ref{prop:dispersionsupp} satisfies for almost all time $t\in\R_t$:
    \begin{equation*}
        \gamma_{n,t}(h,x_2,\pi^{\delta,\beta}) 
        \geq {\bold 1}_{C_{\nu_0}} k_{n,t}^{2}(h,\pi^{\delta,\nu_0\delta^{1/3}},x_2,x_2)dx_2 \otimes \gamma_{n}^{2}(h,\pi^{\delta,\nu_0\delta^{1/3}}),
    \end{equation*}
    where $k_{n,t}^{2}(h,\pi^{\delta,\nu_0\delta^{1/3}},\cdot,\cdot)$ is the kernel of the trace-class operator 
    \begin{equation*}
        \Gamma_{n,t}^{2}(h,\pi^{\delta,\nu_0\delta^{1/3}})= e^{-it\frac{\Tilde{\mu}_{n}''(\nu_0)}{2}\partial_{x_2}^2}\Gamma_{n,0}^2(h,\pi^{\delta,\nu_0\delta^{1/3}}) e^{it\frac{\Tilde{\mu}_{n}''(\nu_0)}{2}\partial_{x_2}^2}.
    \end{equation*} 
\end{theorem} 

We have used in the precedent theorem the writing of an element $x\in\Eng$ as a couple $(h,x_2)$ in $\Heis\times\R_2$. For further details, 
we refer to Section \ref{sect:preliEngel}.

\begin{remark}
    One can reformulate Theorem \ref{thm:propag2microintro} saying the map $t\mapsto \Gamma_{n,t}^{2}$ is continuous and is determined by the following weak Heisenberg equation:
    \begin{equation*}
        \begin{cases}
        & i\partial_{t}\Gamma_{n,t}^{2} = \left[-\frac{\Tilde{\mu}_{n}''(\nu_0)}{2} \partial_{x_2}^{2},\Gamma_{n,t}^{2}\right],\\
        & \Gamma_{n,t}^{2}{}_{|t=0} = \Gamma_{n,0}^{2}.
        \end{cases}
    \end{equation*}
\end{remark}

When there is exactly one critical point to the map $\nu\in\R\mapsto \Tilde{\mu}_n(\nu)$, which is true for all but a finite number of $n\in\N_{>0}$, the previous inequality on $\gamma_{n,t}$ is actually an equality.

\subsection{Obstruction to dispersion}
We present here the consequences of Theorem \ref{thm:propag2microintro} for the study of the Schrödinger equation \eqref{eq:SchEngel} for $\tau=2$, i.e the non-semiclassical regime.
Let $(\psi^\hbar(t)=e^{i t\Delta_{\Eng}}\psi_{0}^{\hbar})_{\hbar>0}$ be a solution to this equation. As dispersion often translates into a better regularity 
of the densities $(|\psi^{\hbar}(t,x)|^2 dx dt)_{\hbar>0}$ than expected at first, we wish to undertand the regularity
of the weak limits of this family. More precisely, we are interested in the behaviour of quantities
\begin{equation*}
    \int_{0}^{T}\int_{\Eng} \chi(x) |e^{i t\Delta_{\Eng}}\psi_{0}^{\hbar}(x)|^2\,dxdt,
\end{equation*}
in the asymptotic $\hbar \rightarrow 0$, for $\chi \in C_{0}^{\infty}(\Eng)$, $T\in \R_{+}$. Our understanding of such weak limits comes from the fact 
that they can be obtained as marginals of the previously studied semiclassical measures.

Concerning the existence of Strichartz estimates, first note that their study in the setting of nilpotent Lie groups has already been the topic of multiple articles (see \cite{BBG},\cite{BGX},\cite{BFI},\cite{DH}),
but has been limited to nilpotent groups of step 2. The following result restricts the range of exponents for which one can hope to prove Strichartz estimates on the Engel group.
First observe that the only couples of exponents  $(q,p)\in [2,\infty]^2$ for which there exists $C_{p,q}>0$ such that
\begin{equation}
    \label{eq:strichartz}
    \forall \psi_0 \in L^2(\Eng),\ \|e^{it\Delta_\Eng}\psi_0 \|_{L^q(\R_t,L^p(\Eng))}\leq C_{p,q}\|\psi_0\|_{L^2(\Eng)},
\end{equation}
must satisfy the following scaling admissibility condition 
\begin{equation*}
   \frac{2}{q}+\frac{7}{p} = \frac{7}{2}.
\end{equation*}
It is interesting to see that only one couple of exponents can actually been obtained aside from the trivial one $(\infty,2)$.
\begin{proposition}
    \label{prop:obsstrichartz}
    The only exponents $(q,p)$ for which the previous Strichartz estimates are possible are 
    \begin{equation*}
        (q,p) = (\infty,2) \ {\rm or }\ (2,14/5).
    \end{equation*}
\end{proposition}

Our second result concerns the study of the local smoothing effect of the Schrödinger equation on the Engel group.
Thanks to the contribution of many authors (see for example \cite{BAD},\cite{CS},\cite{K},\cite{KPV},\cite{S},\cite{V}), it is now well-known that dispersive equations give rise to smoothing-type estimates. 
In the simple case of the free Laplacian on $\R^n$, given any $\eps>0$ and any ball $B\subset\R^n$, it is possible to find $C>0$ such that the following inequality 
\begin{equation}
    \label{eq:smoothingRn}
    \int_{0}^{\eps}\||D_x|^{1/2}\left(e^{it\Delta}\psi_0\right)\|_{L^2(B)}^2\,dt \leq C\|\psi_0\|_{L^2(\R^n)}^2,
\end{equation}
holds uniformly for every $\psi_0 \in C_{c}^\infty(\R^n)$.

One can wonder if a similar phenomenon can be observed on the Engel group: given $\eps >0$, $s>0$ and a bounded open set $\Omega \subset \Eng$, is it possible to find a constant $C>0$ such that 
the following inequality 
\begin{equation}
    \label{eq:smoothingEng}
    \int_{0}^{\eps}\||\Delta_{\Eng}|^{s/2}\left(e^{it\Delta_\Eng}\psi_0\right)\|_{L^2(\Omega)}^2\,dt \leq C\|\psi_0\|_{L^2(\Eng)}^2,
\end{equation}
holds uniformly for every $\psi_0 \in C_{c}^\infty(\Eng)$?

\begin{theorem}
    \label{thm:obstructionsmoothing}
    Given any $\eps>0$, $s>0$ and any bounded open set $\Omega\subset\Eng$, it is not possible to find a constant $C>0$ such that the estimate \eqref{eq:smoothingEng} 
    holds uniformly for every solution of Schrödinger equation with initial datum $\psi_0 \in C_{c}^\infty(\Eng)$.
\end{theorem}

Actually, one can prove a finer result as the previous theorem does not exclude that smoothness could be gained in some particular directions. 
The only direction where one could expect to gain regularity is along the characteristic leaves but, in the spirit of \cite{Burq}, the following result shows that this is still not attainable.

\begin{theorem}
    \label{thm:refinedobstructionsmoothing}
    Given any $\eps>0$, $s>0$ and any bounded open set $\Omega\subset\Eng$, it is not possible to find a constant $C>0$ such that the estimate
    \begin{equation*}
        \int_{0}^{\eps}\||X_2|^s\left(e^{it\Delta_\Eng}\psi_0\right)\|_{L^2(\Omega)}^2\,dt \leq C\|\psi_0\|_{L^2(\Eng)}^2,
    \end{equation*} 
    holds uniformly for every solution of Schrödinger equation with initial datum $\psi_0 \in C_{c}^\infty(\Eng)$.
\end{theorem}

\begin{remark}
    The previous results still hold for a Schrödinger operator $-\Delta_\Eng + V$ where $V\in C^\infty(\Eng)$ is bounded as well as all its derivatives, as such an addition would only modify the propagation laws 
    of the time-averaged second-microlocal semiclassical measures. Indeed, the operator $-\frac{\Tilde{\mu}_{n}''(\nu_0)}{2} \partial_{x_2}^{2}$ is then replaced by the Schrödinger operators $-\frac{\Tilde{\mu}_{n}''(\nu_0)}{2} \partial_{x_2}^{2}+ V(h,\cdot)$
    on the real line. For the sake of clarity, we will focus on the case $V=0$.
\end{remark}

\subsection{Organization of the paper} In the next section, we present the fundamental results concerning the Engel group. The following 
two sections develop the semiclassical analysis of $\Eng$ and present the proofs of Theorems \ref{theo:propagintro} and \ref{thm:prop_wave_packet} and Proposition \ref{prop:dispersionsupp}.
In Section \ref{sect:secondmicro}, we develop a second-microlocal analysis on the cones $C_{\nu_0}$ of the phase space and apply these tools in Section \ref{sect:sch2micro} to the study 
of solutions to Schrödinger equation \eqref{eq:SchEngel} for the time-scale $\tau=2$. At the end of this section, we give the proofs of the corollaries of this analysis to the study of dispersion.

\subsection*{Aknowledgement} The author is deeply grateful to Clotilde Fermanian Kammerer for her numerous advices during the realization of this project. Many thanks also 
go to Véronique Fischer for very interesting conversations.

\section{Preliminaries on the Engel group}
\label{sect:preliEngel}

\subsection{Semidirect decomposition}
\label{subsect:semidirect}

As seen in the introduction, the exponential map defines a diffeomorphism between the Lie algebra $\fg_\Eng$ and the Engel group. With the particular choice of basis $(X_i)_{1\leq i\leq 4}$, one can then realize the Engel group 
in its isomorphism class as $\R^4$ equipped with the polynomial group law induced by the Baker-Campbell-Hausdorff formula and 
the relations \eqref{eq:bracketrel}. 
However, other diffeomorphisms exist between $\fg_{\Eng}$ and the group $\Eng$ and as a matter of fact, different parametrisations of the Engel group have been used in the litterature, see \cite{BBGL}, \cite{BGR}. 

The one we present here is well suited for our analysis as it results from an underlying algebraic decomposition of the Engel group and allows for a clear presentation of its irreducible representations.

\begin{proposition}
    \label{prop:engsemidirect}
    The following map 
    \begin{equation*}
         (x_1,x_2,x_3,x_4)\in\R^4 \mapsto \Exp(x_1 X_1 + x_3 X_3 + x_4 X_4)\Exp(x_2 X_2) \in \Eng,
    \end{equation*}
    defines a diffeomorphism.
\end{proposition}

This parametrisation allows us to single out a realization in its isomorphism class.

\begin{definition}
    In the rest of the article, we realize the Engel group $\Eng$ as the manifold $\R^4$ endowed with the following group product: for $x = (x_1,x_2,x_3,x_4),\,y = (y_1,y_2,y_3,y_4) \in \R^4$, we have
    \begin{equation*}
    \begin{aligned}
        xy ~=~ &\left(x_1+y_1,x_2+y_2,x_3+y_3-x_2y_1,x_4+y_4+\frac{1}{2}(x_1y_3-x_3y_1)-\frac{1}{2}x_1x_2y_1\right).
    \end{aligned}
    \end{equation*}
    The inverse of an element $x$ is given by
    \begin{equation*}
        x^{-1} = (-x_1,-x_2,-x_3-x_2 x_1, -x_4).
    \end{equation*}
\end{definition}

The parametrisation given by Proposition \ref{prop:engsemidirect} leads us to consider the subalgebras
\begin{equation*}
    \fh = \R X_1 \oplus \R X_3 \oplus \R X_4,
\end{equation*}    
and $\fr_{2} = \R X_2$, already introduced before.
As $\fh$ satisfies the commutation relations of the Heisenberg Lie algebra, we denote by $\Heis = \Exp(\fh)$ this injection of the Heisenberg group into $\Eng$. 
We will also denote by $\R_{2} = \Exp(\fr_{2})$ the one parameter subgroup associated to $X_2$. 

Any element of the Engel group can then be decomposed
as the product of an element of $\Heis$ and an element of $\R_{2}$. Furthermore, for $x_2\in\R$ and $(y_1,y_3,y_4)\in\R^3$, thanks to the Baker-Campbell-Hausdorff 
formula, we compute that
\begin{equation}
    \label{eq:acionsemidirect}
    \Exp(x_2 X_2) \Exp(y_1 X_1 + y_3 X_3 + y_4 X_4) \Exp(-x_2 X_2) = \Exp(y_1 X_1 + (y_3-x_2 y_1)X_3 + y_4 X_4).
\end{equation}
In other words, we have proved that the Engel group $\Eng$ can be written as a semidirect product:
\begin{equation}
    \label{eq:decompositionsemidirect}
    \Eng = \Heis \rtimes \R_{2},
\end{equation}
where the action of $\R_{2}$ on $\Heis$ is given by \eqref{eq:acionsemidirect}. For this reason, the coordinates introduced in Proposition \ref{prop:engsemidirect} will be called the {\it semidirect coordinates}.
It is straightforward to check that the pushforward of the Lebesgue measure on $\R^4$ by the semidirect coordinates corresponds to the same Haar measure already introduced.

We can compute the expressions of the vector fields $X_i$, $1\leq i \leq 4$ in semidirect coordinates.
Recall that we may identify $\fg_{\Eng}$ with the space of left-invariant vector fields on $\Eng$ via
\begin{equation}
    \label{eq:leftinv}
    Xf(x) = \frac{d}{dt}f(x\Exp(tX))_{|t=0},\quad f\in C^{\infty}(\Eng),\ x\in\Eng.
\end{equation}
The basis $(X_1,X_2,X_{3},X_4)$ thus corresponds to the following vector fields:
\begin{equation*}
\begin{aligned}
    \left\{
    \begin{split}
    X_1 &= \partial_1 - x_2\partial_3 - \frac{1}{2}\left(x_3+x_1x_2\right)\partial_4,\\
    X_2 &= \partial_2,\\
    X_3 &= \partial_3 + \frac{1}{2}x_1\partial_4,\\
    X_4 &= \partial_4.
    \end{split}
    \right.
\end{aligned}
\end{equation*}

\subsection{The unitary dual set and the Fourier transform}

We present here the harmonic analysis of the Engel group that we will use in the following sections.

\subsubsection{Irreducible unitary representations and dual set}
\label{subsubsect:dualEng}
The first step in understanding the harmonic analysis of the Engel group is to compute its irreducible representations.

For any nilpotent Lie group $G$, Kirillov's theory and the orbit method allow for an explicit description of its irreducible representations. Furthermore, it gives
an homeomorphism between the dual set $\widehat{G}$ and the set of coadjoint orbits $\fg^*/{\rm Ad}^*(G)$. For a more extensive discussion of the orbit method, see \cite{CG}, \cite{Dix}, \cite{FR}, and in 
the particular case of the Engel group see \cite[Appendix A]{BBGL}.

In the case of the Engel group, this gives the following proposition.
\begin{proposition}
    \label{prop:dualset}
    The dual of $\Eng$ is the disjoint union of three classes of representations
    \begin{equation*}
        \Enghat = \{\pi^{\delta,\beta}\,:\,(\delta,\beta) \in \fg_{3}^*\setminus\{0\}\times\fr_{2}^{*} \} \sqcup \{\pi^{\lambda}\,:\,\lambda \in \fg_2^{*}\setminus\{0\}\} 
       \sqcup \{\pi^{(0,\alpha)}\,:\,\alpha \in \fg_{1}^* \}. 
    \end{equation*}
    Let $x=(x_1,x_2,x_3,x_4)\in\Eng$. For $(\delta,\beta) \in \fg_{3}^*\setminus\{0\}\times\fr_{2}^{*}$, the irreducible unitary representation $\pi^{\delta,\beta}$ is defined by: for $\varphi \in \cH_{\pi^{\delta,\beta}} := L^{2}(\R_\xi)$,
    \begin{equation*}
        \pi^{\delta,\beta}(x)\varphi(\xi) = 
        {\rm exp}\left[i \delta \left(x_4 + \xi x_3 +\frac{1}{2}x_{1}x_3\right)\right] 
        {\rm exp}\left[ i\left(\beta + \frac{\delta}{2}(\xi+x_1)^{2}\right)x_{2}\right]\varphi(\xi +x_{1}).
    \end{equation*}
    For $\lambda \in \fg_{2}^* \setminus \{0\}$, the unitary representation $\pi^{\lambda}$ acts on $\mathcal{H}_{\lambda} := L^{2}(\R_\xi)$ by
    \begin{equation*}
        \pi^{\lambda}(x) \varphi(\xi) = 
        {\rm exp}\left[ i\lambda \left(x_3 + \xi x_2 + \frac{1}{2}x_1x_2\right) \right] \varphi(\xi+x_1).
    \end{equation*}
    And finally, the last unitary irreducible representations are given by the characters of the first stratum: for every $\alpha = (\alpha_1,\alpha_2) \in \fg_{1}^{*}$, we set
    \begin{equation*}
        \pi^{(0,\alpha)}(x) = {\rm exp}\left[ i(\alpha_1 x_1 + \alpha_2 x_{2}) \right] \in \C.
    \end{equation*}
    Note that the trivial representation $1_{\Eng}$ corresponds to the class of $\pi^{(0,\alpha)}$ with $\alpha=0$.
\end{proposition}

As already pointed out in the introduction, the family of irreducible representations $\{\pi^{\delta,\beta}\,:\,(\delta,\beta)\in\fg_{3}^*\setminus\{0\}\times\fr_{2}^{*}\}$ are called the generic representations and is denoted by $\Enghat_{\rm gen}$. This set
is open and dense in $\Enghat$, and homeomorphic to $\R\setminus\{0\}\times\R$.

\begin{remark}
The non-generic representations can be understood in the following way. Consider the realization of the Heisenberg group as the manifold $\R^3$ with the following group product: 
for $(x_1,x_2,x_3),(y_1,y_2,y_3)\in\R^3$, we put
\begin{equation*}
    (x_1,x_2,x_3)\star_{\Bar{\Heis}}(y_1,y_2,y_3) = (x_1+y_1, x_2+y_2, x_3+y_3-x_2y_1),
\end{equation*}
and we denote by $\Bar{\Heis}$ this realization. We easily ckeck that the map $p_{\Bar{\Heis}}:\Eng\rightarrow\Bar{\Heis}$ defined by
\begin{equation*}
    p_{\Bar{\Heis}}:(x_1,x_2,x_3,x_4)\in\Eng \mapsto (x_1,x_2,x_3)\in\Bar{\Heis},
\end{equation*}
is a surjective group morphism. 
The non-generic irreducible representations of $\Eng$, i.e the families of representations 
$\{\pi^{\lambda}\,:\,\lambda \in \fg_{2}^{*}\setminus\{0\}\}$ and $\{\pi^{(0,\alpha)}\,:\,\alpha \in \fg_{1}^*\}$, are obtained as the irreducible representations of $\Bar{\Heis}$ 
composed with the morphism $p_{\Bar{\Heis}}$. As such we will write 
\begin{equation*}
    \widehat{\Bar{\Heis}} = \{\pi^{\lambda}\,:\,\lambda \in \fg_{2}^{*}\setminus\{0\}\}\sqcup\{\pi^{(0,\alpha)}\,:\,\alpha \in \fg_{1}^*\}.
\end{equation*}
\end{remark}

\subsubsection{Fourier transform}

For $f$ a function in $L^{1}(\Eng)$ and $\pi$ an irreducible representation, the Fourier transform of $f$ at $\pi$ is defined by
\begin{equation*}
    \hat{f}(\pi) = \mathcal{F}f(\pi) := \int_{\Eng} f(x) \pi(x)^{*} \,dx.
\end{equation*}
Note that, as each representation $\pi$ is a group morphism from $\Eng$ to the group of unitary operators $U(\mathcal{H}_{\pi})$, we have that $\mathcal{F}f(\pi)$ is a bounded operator on $\mathcal{H}_{\pi}$ with a norm satifisfying
\begin{equation*}
    \Vert \mathcal{F}f(\pi)\Vert_{\mathcal{L}(\mathcal{H}_\pi)} \leq \int_{\Eng} |f(x)| \Vert \pi(x)^{*}\Vert_{\mathcal{L}(\mathcal{H}_\pi)}\,dx  = \Vert f \Vert_{L^{1}(\Eng)}.
\end{equation*}
Observe that if $\pi_1$ and $\pi_2$ are two representations intertwined by an operator $\cI$, we have 
\begin{equation*}
    \mathcal{F}f(\pi_1) = \cI^{-1}\circ\mathcal{F}f(\pi_2)\circ\cI,
\end{equation*}
and we thus intepret the Fourier transform as a field of operators on $\Enghat$ up to equivalence through intertwining operators.

We also introduce here the semiclassical Fourier transform $\cF_{\Eng}^\hbar$. First, we extend the dilations on the group $\Eng$ to its unitary dual $\Enghat$: for $\pi\in\Enghat$ and $r\in\R_{>0}$,
we define the representation $r\cdot \pi$ by 
\begin{equation*}
    \forall x\in\Eng,\ (r\cdot \pi)(x) := \pi(r\cdot x).
\end{equation*}
For each kind of representations given by Proposition \ref{prop:dualset}, we have
\begin{equation*}
    [r\cdot\pi^{\delta,\beta}] = [\pi^{r^3\delta,r\beta}],\ [r\cdot\pi^{\lambda}] = [\pi^{r^2\lambda}],\ [r\cdot\pi^{(0,\alpha)}] = [\pi^{(0,r\alpha)}].
\end{equation*}
For $f$ a function in $L^{1}(\Eng)$ and $\pi$ an irreducible representation, the semiclassical Fourier transform of $f$ at $\pi$ is defined by
\begin{equation}
    \label{eq:semiFourier}
    \mathcal{F}_{\Eng}^\hbar f(\pi) := \int_{\Eng} f(x) (\hbar^{-1}\cdot \pi(x))^{*} \,dx = \cF_\Eng f (\hbar^{-1}\cdot \pi).
\end{equation}

\subsubsection{Plancherel formula}
Following \cite{Dix}, the Plancherel theorem states there exists a unique positive $\sigma$-finite measure $\mu_{\Enghat}$, called the {\it Plancherel measure} of $\Eng$, such that for any 
$f\in \cS(\Eng)$, we have 
\begin{equation}
    \label{eq:Plancherelformula}
    \int_{\Eng}|f(x)|^2\,dx = \int_{\Enghat} \|\mathcal{F}f(\pi)\|_{{\rm HS}(\cH_\pi)}^2\,d\mu_{\Enghat}(\pi).
\end{equation}

The orbit method allows us to compute explicitly this measure with respect to the parametrisation of $\Enghat$ given by Proposition \ref{prop:dualset}.
\begin{proposition}
    \label{prop:plancherelmeasure}
    The Plancherel measure of the Engel group is supported in $\Enghat_{\rm gen}$. Moreover, there exists $c_{\Eng} > 0$ such that 
    \begin{equation*}
        d\mu_{\Enghat}(\pi^{\delta,\beta}) = c_{\Eng} |\delta|d\delta d\beta.
    \end{equation*}
\end{proposition}

Equation \eqref{eq:Plancherelformula} may be reformulated in the following way:
the group Fourier transform can be extended to a surjective isometry from $L^2(\Eng)$ onto the Hilbert
space 
\begin{equation*}
    L^2(\Enghat) := \int_{\Enghat}^{\oplus}{\rm HS}(\cH_\pi)\,d\mu_{\Enghat}(\pi),
\end{equation*}
i.e the space of $\mu_{\Enghat}$-measurable fields of Hilbert-Schmidt operators $\{\sigma(\pi)\in{\rm HS}(\cH_\pi)\,:\,\pi\in\Enghat\}$ which are square-integrable in the 
sense that 
\begin{equation*}
    \|\sigma\|_{L^2(\Enghat)} := \int_{\Enghat}\|\sigma(\pi)\|_{{\rm HS}(\cH_\pi)}^2 \,d\mu_{\Enghat}(\pi) < \infty.
\end{equation*}

The Plancherel formula also yields an inversion formula: for any $f \in {\mathcal S}(\Eng)$ and $x\in \Eng$, we have
\begin{equation}
\label{eq:inversionformula} 
f(x) = \int_{\Enghat} {\rm Tr}_{\cH_\pi} \left(\pi(x) {\mathcal F}f(\pi)\right)\, d\mu(\pi) \,,
\end{equation}
where ${\rm Tr}_{\cH_\pi}$ denotes the trace of trace-clas operators on $\cH_{\pi}$.
This formula makes sense since for $f \in {\mathcal S}(\Eng)$, the operators ${\mathcal F}f(\pi)$ for $\pi\in \Enghat$, are trace-class and the right-side term is finite, see \cite{CG}.

\subsubsection{Convolution and the von Neumann algebra of the group}
\label{subsubsec_convop}
Let $f,g\in L^1(\Eng)$ be two integrable functions on $\Eng$, the convolution $f*g$ is defined by
\begin{equation*}
    \forall x\in\Eng,\ (f * g)(x) := \int_{\Eng} f(y)g(y^{-1}x)\,dy.
\end{equation*}
The Young inequality is still valid if one takes the homogeneous dimension $Q$ instead of the dimension of the manifold $n=4$.

As expected, the Fourier transform  sends the convolution to composition in the following way:
\begin{equation}\label{fourconv}
 \forall \pi\in\Enghat,\ {\mathcal F}( f * g )( \pi ) 
 = {\mathcal F} g(\pi)\circ
 {\mathcal F}f (\pi) \, .
\end{equation}
We recall that a convolution operator $T$ with integrable convolution kernel $\kappa\in L^1(\Eng)$ is defined by  $Tf=f*\kappa$ and we have $\mathcal F (Tf) = \mathcal F \kappa \circ \mathcal F f$ by \eqref{fourconv}. 
Hence, $T$ appears as a Fourier multiplier with Fourier symbol $\mathcal F \kappa$ acting on the left of $\mathcal F f$.
Consequently, $T$ is invariant under left-translation and  bounded on $L^2(\Eng)$ with operator norm 
\begin{equation*}
    \| T\|_{{\mathcal L} (L^2(\Eng))}\leq \sup_{\pi\in \Enghat} \|\mathcal F \kappa(\pi) \|_{{\mathcal L}(\mathcal{H}_{\pi})}.
\end{equation*}
In other words, $T$ is in the space $ {\mathcal L} (L^2(\Eng))^\Eng$ of the left-invariant bounded operators on $L^2(\Eng)$.

Let us denote by $L^\infty(\Enghat)$ the space of measurable fields of operators $\sigma=\{\sigma(\pi):\pi \in \Enghat\}$ which are bounded in the sense that the essential supremum 
with respect to the Plancherel measure $\mu_{\Enghat}$,
\begin{equation*}
\|\sigma\|_{L^\infty (\Enghat)}
:={\rm supess}_{\pi\in \Enghat} \|\sigma\|_{{\mathcal L}(\mathcal H_\pi)},
\end{equation*}
is finite.
The space $L^\infty(\Enghat)$ with the norm $\|\cdot\|_{L^\infty(\Enghat)}$ is a von Neummann algebra, and is called the {\it von Neumann algebra of the group}.
As explained above, we already know we have $\mathcal F \left(L^1(\Eng)\right)\subset L^\infty(\Enghat)$,  but this inclusion is strict.

The full Plancherel theorem (see \cite{Dix}) implies that the von
Neumann algebras $L^\infty(\Enghat)$ and the space $ {\mathcal L} (L^2(\Eng))^\Eng$ of bounded operators on $L^2(\Eng)$ invariant by left-translations introduced above are isomorphic via the mapping
$\sigma \mapsto \fM(\sigma)$ where $\fM(\sigma)$ is the Fourier multiplier defined by
\begin{equation}
\mathcal F \left(\fM(\sigma)f \right) = \sigma \circ \mathcal F f, 
\quad f\in L^2(\Eng).
\end{equation}
The isomorphism between $L^\infty(\Enghat)$ and ${\mathcal L} (L^2(\Eng))^\Eng$ allows us to naturally extend the group Fourier transform to distributions $\kappa\in \mathcal S'(\Eng)$ such that the convolution operator $f\mapsto f* \kappa$ is bounded on $L^2(\Eng)$ 
by setting that $\mathcal F \kappa$  is the symbol of the corresponding operator in ${\mathcal L} (L^2(\Eng))^\Eng$. Following \cite{FR}, the space of such distributions $\kappa$
will be denoted by $\cK(\Eng)$ and constitutes another realization of the von Neumann algebra of $\Eng$.

Another functional space that will be of use in our study is the space of the Fourier transform of Schwartz functions.
\begin{definition}
    \label{def:schwartenghat}
    We denote by $\cS(\Enghat)\subset L^2(\Enghat)$ the image by the Fourier transform of the Schwartz space $\cS(\Eng)$, i.e 
    \begin{equation*}
        \cS(\Enghat) = \cF(\cS(\Eng)).
    \end{equation*}
    Moreover, we put on $\cS(\Enghat)$ the topology such that the Fourier transform $\cF:\cS(\Eng)\rightarrow\cS(\Enghat)$ is a homeomorphism.
\end{definition}

\subsubsection{Infinitesimal representations and Fourier transforms of left-invariant differential vector fields}
The group Fourier transform can also be extended to certain classes of distributions whose convolution operators yield left-invariant operators, but not necessarily bounded on $L^2(\Eng)$ as previously.
Indeed, denoting by $\pi(X)$ the infinitesimal representation of $\pi$ at $X\in \mathfrak g$, i.e. $\pi(X) = \frac{d}{dt} \pi ({\rm Exp}(tX))|_{t=0}$, we have
\begin{equation*}\label{F(V)}
\mathcal F(Xf)(\pi) = \pi(X) \circ \mathcal Ff (\pi). 
\end{equation*}
For instance, we have for $(\delta,\beta)\in\fg_{3}^{*}\setminus\{0\}\times\fr_{2}$, we have
\begin{equation}
\label{eq_piDB}
    \left\{
    \begin{split}
     \pi^{\delta,\beta}(X_1) &= \partial_{\xi}, \\
     \pi^{\delta,\beta}(X_2) &= i\left(\beta + \frac{\delta}{2}\xi^2\right), \\
     \pi^{\delta,\beta}(X_{3}) &= i\delta \xi, \\
     \pi^{\delta,\beta}(X_4) &= i\delta.
    \end{split}
    \right.
\end{equation}
These operators need to be seen as acting on the space of smooth vectors $\cH_{\pi}^{\infty}\subset\cH_\pi$ for $\pi=\pi^{\delta,\beta}$, which can be identified with the Schwartz space $\cS(\R_\xi)\subset L^2(\R_\xi)$, see \cite{FR}.
We also compute for any $\lambda\in \fg_{2}^*\setminus\{0\}$ that
\begin{equation}
\label{eq_piG}	
\pi^\lambda(X_1)= \partial_{\xi},\ \pi^\lambda(X_2)=i\lambda{\xi},\ \pi^\lambda(X_{3})=i\lambda
\ \mbox{and}\ \pi^\lambda(X_{4})=0,
\end{equation}
and for $\alpha = (\alpha_1,\alpha_2)\in \fg_{1}^*$,
\begin{equation}\label{Fourieromega}
\pi^{(0,\alpha)} (X_1)=i\alpha_1,\ \pi^{(0,\alpha)} (X_2)=i\alpha_2\ \mbox{and}\ \pi^{(0,\alpha)}(X_{3}) = \pi^{(0,\alpha)}(X_4) =0.
\end{equation}
The infinitesimal representation of $\pi$ extends to the universal enveloping Lie algebra $\fU(\fg_\Eng)$ of $\fg_{\Eng}$ that we identify with the left-invariant differential operators on $\Eng$.
Then for such a differential operator~$T$, we have $\mathcal F(Tf)(\pi) = \pi(T) \mathcal Ff (\pi)$
and we may write $\pi(T)={\mathcal F}(T)$. 
For instance, if we denote by ${X}^{\aleph}$ the product of $|\aleph|$ left-invariant vector fields ($\aleph \in\N^4$), then 
\begin{equation}
\mathcal F({X}^\aleph f)(\pi) = \pi({X})^\aleph \mathcal Ff (\pi)\;\;{\rm and}\;\;
\mathcal F({X}^\aleph ) = \pi({X})^\aleph.
\end{equation}
Note that $\pi({X})^\aleph$ may be considered as a field of unbounded operators on $\Enghat$ defined on the smooth vectors of the representations, (see \cite{FR}).

For the particular case of the subLaplacian $\Delta_\Eng$, we compute its infinitesimal representation $H = \cF(-\Delta_\Eng)$ thanks to the equalities in \eqref{eq_piDB}, \eqref{eq_piG} and \eqref{Fourieromega}: 
at $\pi^{(0,\alpha,\beta)}$, $(\alpha,\beta) \in \fg_{1}^*$, we find the scalar 
\begin{equation}
    H(\pi^{(0,\alpha)}) = \cF(-\Delta_{\Eng})(\pi^{(0,\alpha)}) = |\alpha|^2.
\end{equation}
For a representation of the form $\pi^{\lambda}$, $\lambda \in \fg_{2}^*$, we obtain
\begin{equation}
    H(\pi^\lambda) = \cF(-\Delta_\Eng)(\pi^\lambda) = -\partial_{\xi}^2 + \lambda^2 \xi^2,
\end{equation}
and finally for $\pi^{\delta,\beta}\in\Enghat_{\rm gen}$, we get 
\begin{equation}
    H(\pi^{\delta,\beta}) = \cF(-\Delta_\Eng)(\pi^{\delta,\beta}) = -\partial_{\xi}^2 + \left(\beta + \frac{\delta}{2}\xi^2\right)^2.
\end{equation}

\subsubsection{Difference operators}

For a general nilpotent Lie group $G$ with Lie algebra $\fg$, we introduce the following operators acting on fields of operators on $\fghat$.

\begin{definition}
    \label{def:diffop}
    For any $q \in C^{\infty}(G)$ with atmost polynomial growth, we set 
    \begin{equation*}
        \forall \pi \in \fghat,\ \Delta_{q}\widehat{\kappa}(\pi) := \widehat{q\kappa}(\pi),
    \end{equation*}
    for any Schwartz function $\kappa\in\cS(G)$.
\end{definition}

Note that for two such functions $q_1, q_2$, the corresponding difference operators commute:
\begin{equation*}
    \Delta_{q_1}\Delta_{q_2} = \Delta_{q_2}\Delta_{q_1} = \Delta_{q_1 q_2}.
\end{equation*}

\begin{remark}
    The difference operators generalise the Euclidean derivatives with respect to the Fourier variable on $(\R^n,+)$ and will serve as a mean 
    of measuring regularity with respect to the dual variable $\pi\in\widehat{G}$.
\end{remark}

\begin{proposition}
    \label{prop:IPP}
    Let $G$ a nilpotent Lie group with Plancherel measure $d\mu_{\widehat{G}}$: for $f\in\cS'(G)$ and $g\in\cS(G)$, we have 
    \begin{equation*}
        \int_{\widehat{G}} {\rm Tr}(\Delta_i \widehat{f}(\pi) \widehat{g}(\pi))\,d\mu_{\widehat{G}}(\pi) 
        = -\int_{\widehat{G}} {\rm Tr}(\widehat{f}(\pi) \Delta_i \widehat{g}(\pi))\,d\mu_{\widehat{G}}(\pi),
    \end{equation*}
    where $\Delta_i$ is the difference operator associated to the exponential coordinate $x_i$ and $\hat{f}$ is in the domain of $\Delta_i$.
\end{proposition}

\begin{proof}
    With the same notations, we can write
    \begin{equation*}
        \begin{aligned}
            \int_{\widehat{G}} {\rm Tr}(\Delta_i \widehat{f}(\pi) \widehat{g}(\pi))\,d\mu_{\widehat{G}}(\pi) 
            &= \left((-x_i f) * g\right) (e_G)\\
            &= \langle f, y_i g(\cdot^{-1})\rangle_{\cS',\cS}\\
            &= -\left(f * (x_i g)\right)(e_G)\\
            &= -\int_{\widehat{G}} {\rm Tr}(\widehat{f}(\pi) \Delta_i \widehat{g}(\pi))\,d\mu_{\widehat{G}}(\pi).
        \end{aligned}
    \end{equation*}
\end{proof}

In the case of the Engel group, we will more simply note as $\Delta_i$ the difference operators associated to the coordinate function 
$q_i(x_1,x_2,x_3,x_4) = -x_i$ and write, for $\alpha \in \N^4$, $\Delta^{\alpha}$ the product of $|\alpha|$ such operators. 

As one might expect, the difference operators enjoy an avatar of Leibniz formula.
\begin{proposition}\cite[Chapter 5]{FR}
    \label{prop:formuleLeibniz}
    For any $\alpha\in\N_{>0}^4$, there exist constants $c_{\alpha,\alpha_1,\alpha_2}\in\R$, $\alpha_1,\alpha_2 \in\N_{>0}^4$, with $[\alpha_1]+[\alpha_2] = [\alpha]$, such that for any $\kappa_1,\kappa_2\in\cS(\Eng)$, we have
    \begin{equation*}
        \Delta^{\alpha}\left(\hat{\kappa}_1\hat{\kappa}_2\right) = \sum_{[\alpha_1]+[\alpha_2] = [\alpha]}c_{\alpha,\alpha_1,\alpha_2}\Delta^{\alpha_1}\hat{\kappa}_1\Delta^{\alpha_2}\hat{\kappa}_2,
    \end{equation*}
    with 
    \begin{equation*}
        c_{\alpha,\alpha_1,0} = \left\{\begin{split}1 &\ if\  \alpha_1 = \alpha\\0 &\ otherwise\ \end{split}\right.,\ c_{\alpha,0,\alpha_2} = \left\{\begin{split}1 &\ if\  \alpha_2 = \alpha\\0 &\ otherwise\ \end{split}\right..
    \end{equation*}
\end{proposition}

Finally, we can explicitly compute the action of difference operators on elements of $\cS(\Enghat)$.

\begin{proposition}
    \label{prop:diffopformula}
    Using the parametrisation of the dual $\Enghat$ given by Proposition \ref{prop:dualset}, for $\kappa \in \cS(\Eng)$ a Schwartz function, the restriction of the Fourier transform $\widehat{\kappa}$ to the generic representations $\Enghat_{\rm gen}$ 
    \begin{equation*}
        (\delta,\beta) \in \fg_{3}^*\setminus\{0\}\times \fr_2^* \mapsto \widehat{\kappa}(\pi^{\delta,\beta}) \in L^2(\R_\xi),
    \end{equation*}
    identifies with a smooth map in $C^{\infty}(\fg_{3}^*\setminus\{0\}\times \fr_2^*,L^2(\R_\xi)$. Moreover, for $\pi^{\delta,\beta}\in\Enghat_{\rm gen}$ we have: 
    \begin{equation*}
        \begin{aligned}
        \Delta_{1}\widehat{\kappa}(\pi^{\delta,\beta}) &= \frac{i}{\delta} \left[\pi^{\delta,\beta}(X_3),\widehat{\kappa}(\pi^{\delta,\beta})\right],\\
        \Delta_{2}\widehat{\kappa}(\pi^{\delta,\beta}) &= \frac{1}{i}\partial_{\beta}\widehat{\kappa}(\pi^{\delta,\beta}),\\
        \Delta_{3}\widehat{\kappa}(\pi^{\delta,\beta}) &= \frac{i}{\delta} \pi^{\delta,\beta}(X_3) \Delta_2 \widehat{\kappa}(\pi^{\delta,\beta}) - \frac{i}{\delta}\left[\pi^{\delta,\beta}(X_1),\widehat{\kappa}(\pi^{\delta,\beta})\right],\\
        \Delta_{4}\widehat{\kappa}(\pi^{\delta,\beta}) &= \frac{1}{i}\partial_{\delta}\widehat{\kappa}(\pi^{\delta,\beta}) + \frac{1}{2}\Delta_1 \Delta_3 \widehat{\kappa}(\pi^{\delta,\beta}) + \frac{i}{\delta} \pi^{\delta,\beta}(X_3) \Delta_3 \widehat{\kappa}(\pi^{\delta,\beta})\\
        &-\frac{1}{2} \xi^2 \Delta_2 \widehat{\kappa}(\pi^{\delta,\beta}).
        \end{aligned}
    \end{equation*}
\end{proposition}

\begin{proof}
    Let $\varphi \in \cH_{\pi^{\delta,\beta}}^{\infty}$ a smooth vector. By definition we have 
    \begin{equation*}
        \Delta_{1}\widehat{\kappa}(\pi^{\delta,\beta})\varphi = -\int_{\Eng} x_1 \kappa(x) \pi^{\delta,\beta}(x)^* \varphi\,dx.
    \end{equation*}
    It follows from direct computation that we have 
    \begin{equation*}
        \left[\frac{i}{\delta}\pi^{\delta,\beta}(X_3),\pi^{\delta,\beta}(x)^*\right] = -x_1 \pi^{\delta,\beta}(x)^*,
    \end{equation*}
    and we deduce the corresponding formula for $\Delta_1$. The formula for $\Delta_2$ follows from the equation 
    \begin{equation*}
        \partial_\beta \pi^{\delta,\beta}(x)^* = -x_2 \pi^{\delta,\beta}(x)^*,
    \end{equation*}
    and we argue similarly for $\Delta_3$. Finally, to establish the formula concerning $\Delta_4$, observe that we have
    \begin{equation*}
        \frac{1}{i}\partial_{\delta} \widehat{\kappa}(\pi^{\delta,\beta}) \varphi(\xi) = \int_{\Eng} \kappa(y)\left(-y_4 +\frac{1}{2}y_1 y_3 - \xi y_3 - \frac{1}{2}\xi^2 y_2\right)\pi^{\delta,\beta}(y)^* \varphi(\xi) \,dy 
    \end{equation*}
    We recognize in the first two terms the operator $\Delta_4 - \frac{1}{2}\Delta_1 \Delta_3$, while we can also write the third one as follows
    \begin{equation*}
        \int_{\Eng} \kappa(y) (-\xi y_3) \pi^{\delta,\beta}(y)^*\, dy = -\frac{i}{\delta} \pi^{\delta,\beta}(X_3) \Delta_3 \widehat{\kappa}(\pi^{\delta,\beta}),
    \end{equation*}
    and we obtain the desired formula.
\end{proof}

\begin{remark}
    These formulae depend on the choice of parametrisation of $\Enghat$ given by Proposition \ref{prop:dualset}. Changing our choice of representation in its 
    equivalence by an intertwining operator would lead to a change in the expressions of the difference operators.
\end{remark}

\section{Semiclassical analysis on the Engel group}
\label{sect:semiclassEng}

In this section, we recall the elements of semiclassical pseudodifferential calculus necessary for our analysis on the Engel group, but the presentation 
would hold for any graded nilpotent Lie group. We present also the results on the Schrödinger equation for the semiclassical time-scale $\tau=1$.

\subsection{Semiclassical pseudodifferential operators}

\subsubsection{The space ${\mathcal A}_0$ of semiclassical symbols}
We denote by $\mathcal A_0(\Eng)$ the space of symbols 
$\sigma = \{\sigma(x,\pi) : (x,\pi)\in \Eng\times \Enghat\}$ of the form 
\begin{equation*}
\sigma(x,\pi)=\mathcal{F} \kappa_x (\pi) = \int_\Eng \kappa_x(z) (\pi(z))^* \,dz, 
\end{equation*}
where $x\mapsto \kappa_x(\cdot)$ is a smooth and compactly supported function from $\Eng$ to $\mathcal{S}(\Eng)$.

As the Fourier transform is injective, it yields a one-to-one correspondence between the symbol~$\sigma$ and the function $\kappa$:
we have $\sigma(x,\pi)=\mathcal F \kappa_x(\pi)$ and conversely the Fourier inversion formula \eqref{eq:inversionformula}
yields 
\begin{equation*}
\forall x,z\in G,\ \kappa_x(z)= \int_{\Enghat} {\rm Tr} \left( \pi(z) \sigma(x,\pi)\right) d\mu_{\Enghat}(\pi).
\end{equation*}
The set $\mathcal A_0$ is an algebra for the composition of symbols since if $\sigma_1(x,\pi)=\mathcal F \kappa_{1,x} (\pi)$
and $\sigma_2(x,\pi)=\mathcal F \kappa_{2,x} (\pi)$
are in $\mathcal A_0$, then so is $\sigma_1(x,\pi)\sigma_2(x,\pi)
=\mathcal F (\kappa_{2,x}*\kappa_{1,x}) (\pi)$
by \eqref{fourconv}. It is also stable under the action of the difference operators and by differentiation in the first variable.

\subsubsection{General classes of symbols}
\label{subsubsect:generalclasssym}
It is possible to define classes of symbols {\it à la Hörmander} in the context of nilpotent Lie group. Indeed, following~ \cite[Section 5.2]{FR},
for $\rho,\delta \in [0,1]$ and $m\in\R$ fixed, the class $S_{\rho,\delta}^{m}(\Eng)$ of {\it symbols of order $m$} is the space of the fields of operators 
$\sigma(x,\pi)$ such that for every $a,b,c\in\N_{>0}$, we have
\begin{equation}
\label{eq_def_Smrhodelta}
\|\sigma\|_{S_{\rho,\delta}^{m},a,b,c} = \sum_{\substack{|\gamma|\leq c,\\ [\alpha]\leq a,[\beta]\leq b}}\sup_{x\in \Eng, \pi\in \Enghat}
\|({\rm Id} +H(\pi))^{\frac{ \rho[\alpha]- \delta[\beta]-m+\gamma}2}
X_x^\beta\Delta^\alpha \sigma(x,\pi) 
({\rm Id} +H(\pi))^{-\frac{\gamma}2 }\|_{\mathcal L(\mathcal H_\pi)}<\infty.
\end{equation}
These spaces are naturally Fréchet spaces when equipped with the seminorms $\|\cdot\|_{S_{\rho,\delta}^{m},a,b,c}$.

\begin{remark}
    It has been shown in \cite{FR} that the topology on $S_{\rho,\delta}^{m}(\Eng)$ is also generated by the seminorms $\|\cdot\|_{S_{\rho,\delta}^{m},a,b} := \|\cdot\|_{S_{\rho,\delta}^{m},a,b,0}$, meaning we
    can restrict ourselves to $c=0$.
\end{remark}

The class of {\it smoothing symbols} $S^{-\infty}(\Eng)$ is then defined as the intersection of all theses classes $\bigcap_{m\in\R}S^m$ with the topology of the projective limit.
We can then check that our symbols $\cA_0$ corresponds to smoothing symbols with compact support on the first variable.

\subsubsection{Quantization and semiclassical pseudodifferential operators} 
Given $\hbar>0$, the {\it semiclassical parameter}, we quantify a symbol $\sigma\in S^{m}(\Eng)$, $m\in\R$, by setting for $f\in{\mathcal S}(\Eng)$,
\begin{equation}
	\label{eq_quantization}
	\Op_\hbar(\sigma) f(x)= \int_{\Enghat} {\rm Tr} \left(  \pi(x) \sigma(x,\hbar\cdot\pi) {\mathcal F} f(\pi) \right)d\mu_{\Enghat}(\pi),
\end{equation}
see \cite[Section 5.1]{FR} for the well-posedness of this formula. Such an operator is called a \textit{semiclassical pseudodifferential operator}.

The integral kernel of the operator $\Op_\hbar(\sigma)$ is the distribution 
\begin{equation*}
\Eng\times \Eng\ni (x,y)\mapsto \kappa^{\hbar}_x(y^{-1} x),
\end{equation*}
where $\kappa^{\hbar}_{x}(z)=\hbar^{-Q} \kappa_{x}\left(\hbar^{-1}\cdot z\right)$ 
and $\kappa_{x}\in\cS'(\Eng)$ is such that $\mathcal{F}\kappa_{x}(\pi)=\sigma(x,\pi)$. The distribution $\kappa$ is called the {\it convolution kernel} associated to $\sigma$. 
For $f\in\cS(\Eng)$, we can write the action of $\Op_\hbar(\sigma)$ as follow
\begin{equation*}
    \Op_\hbar(\sigma) f(x) = f * \kappa_{x}^\hbar(x).
\end{equation*}

As an example, the subLaplacian $-\hbar^2 \Delta_\Eng$ belongs to this pseudodifferential calculus and its symbol is given by the field $H$.

\medskip 

As already observed in \cite{FF}, the action of symbols in $\cA_0$ on $L^2(\Eng)$ is easily seen to be bounded:
\begin{equation*}
\exists C>0,
\forall \sigma\in {\mathcal A}_0,
\forall \hbar>0,\ \| \Op_{\hbar}(\sigma)\|_{{\mathcal L}(L^2(\Eng))}\leq C \, \int_{\Eng} \sup_{x\in \Eng} |\kappa_{x} (z)|\,dz.
\end{equation*}

For a symbol in a general class $S_{\rho,\delta}^m(\Eng)$, $m\in\R$, we have the following result.

\begin{theorem}[Calder\'on-Vaillancourt,\cite{FR}]
    \label{thm:CalderonVaillancourt}
    Let $\sigma \in S_{0,0}^{0}(\Eng)$. Then $\Op_1(\sigma)$ extends to a bounded operator on $L^{2}(\Eng)$. Moreover, there exist a constant $C>0$ and 
    integers $a,b\in\N_{>0}$, both independent of $\sigma$, such that for the seminorm $N(\cdot):= \|\cdot\|_{S_{0,0}^{0},a,b}$ we have
    \begin{equation*}
        \forall \phi \in \cS(\Eng),\ \|\Op_1(\sigma)\phi\|_{L^2(\Eng)} \leq C N(\sigma)\|\phi\|_{L^2(\Eng)}.
    \end{equation*}
\end{theorem}

\subsubsection{Symbolic calculus}

Composition of two semiclassical pseudodifferential operators is again a semiclassical pseudodifferential operator. Moreover, there exists a symbolic calculus giving the symbol of the composition. 
In this paper, we will mainly use the description of the commutator between the subLaplacian and a smoothing semiclassical pseudodifferential operator, which comes from an explicit computation and writes: 
for all $\sigma\in{\mathcal A}_0$,
\begin{equation}\label{eq:commutatorDelta}
[\Op_{\hbar} (\sigma), -\hbar^2 \Delta_\Eng ] 
=
  \Op_{\hbar} \left([\sigma,H]\right)
\ + \ 
2\hbar  \, \Op_{\hbar} \left((\pi(V)\cdot V) \sigma\right)
\ + \ 
\hbar^2 \Op_{\hbar} \left(\Delta_\Eng \sigma\right),
\end{equation}
where have introduced the notation $(\pi(V)\cdot V)\sigma = \pi(X_1) X_1\sigma + \pi(X_2) X_2 \sigma$.

\subsection{Semiclassical measures}
\label{subsect:semiclassicalmeasures}
To understand the concentration properties of a bounded family $(\psi_{0}^\hbar)_{\hbar>0}$ in $L^2(\Eng)$, we are interested in quantities of the form
\begin{equation*}
    \left(\Op_{\hbar}(\sigma) \psi_{0}^\hbar, \psi_{0}^\hbar\right), \quad \sigma \in \mathcal{A}_{0},
\end{equation*}
in the limit $\hbar$ goes to 0.  
We follow the presentation in \cite[Section 4]{FF2} and introduce the following vocabulary for operator-valued measures:

\begin{definition}
\label{def_gammaGamma}
	Let $Z$ be a complete separable metric space, 
	and let $\xi\in Z\mapsto {\mathcal H}_\xi$ a measurable field of complex Hilbert spaces on $Z$.
\begin{itemize}
\item 
	The set $ \widetilde{\mathcal M}_{ov}(Z,({\mathcal H}_\xi)_{\xi\in Z})$
is the set of pairs $(\gamma,\Gamma)$ where $\gamma$ is a positive Radon measure on~$Z$ 
and $\Gamma=\{\Gamma(\xi)\in {\mathcal L}^1({\mathcal H}_\xi):\xi \in Z\}$ is a measurable field of trace-class operators
such that
$$\|\Gamma d \gamma\|_{\mathcal M}:=\int_Z{\rm Tr}_{{\mathcal H}_\xi} |\Gamma(\xi)|d\gamma(\xi)
<\infty.$$		
\item 
	Two pairs $(\gamma,\Gamma)$ and $(\gamma',\Gamma')$ 
in $\widetilde {\mathcal M}_{ov}(Z,({\mathcal H}_\xi)_{\xi\in Z})$
are {equivalent} when there exists a measurable function $f:Z\to \mathbb C\setminus\{0\}$ such that 
$$d\gamma'(\xi) =f(\xi)  d\gamma(\xi)\;\;{\rm  and} \;\;\Gamma'(\xi)=\frac 1 {f(\xi)} \Gamma(\xi)$$ for $\gamma$-almost every $\xi\in Z$.
The equivalence class of $(\gamma,\Gamma)$ is denoted by $\Gamma d \gamma$,
and the resulting quotient set is denoted by ${\mathcal M}_{ov}(Z,({\mathcal H}_\xi)_{\xi\in Z})$.
\item 
A pair $(\gamma,\Gamma)$ 
in $ \widetilde {\mathcal M}_{ov}(Z,({\mathcal H}_\xi)_{\xi\in Z})$
 is {positive} when 
$\Gamma(\xi)\geq 0$ for $\gamma$-almost all $\xi\in Z$.
In  this case, we may write  $(\gamma,\Gamma)\in  \widetilde {\mathcal M}_{ov}^+(Z,({\mathcal H}_\xi)_{\xi\in Z})$, 
and $\Gamma d\gamma \geq 0$ for $\Gamma d\gamma \in {\mathcal M}_{ov}^+(Z,({\mathcal H}_\xi)_{\xi\in Z})$.
\end{itemize}
\end{definition}

By convention and if not otherwise specified,  a representative of the class $\Gamma d\gamma$ is chosen such that ${\rm Tr}_{{\mathcal H}_\xi} \Gamma=1$ on the support of $\gamma$. 
In particular, if ${\mathcal H}_\xi$ is $1$-dimensional, $\Gamma=1$ and  $\Gamma d\gamma$ reduces to the measure $d\gamma$.
One checks readily that $\mathcal M_{ov} (Z,({\mathcal H}_\xi)_{\xi\in Z})$ equipped with the norm $\| \cdot\|_{{\mathcal M}}$ is a Banach space.

\medskip 

When the field of Hilbert spaces is clear from the setting, 
we may write 
$$
\mathcal M_{ov} (Z) = \mathcal M_{ov} (Z,({\mathcal H}_\xi)_{\xi\in Z}),
\quad
\mbox{and}\quad
\mathcal M_{ov}^+ (Z) = \mathcal M_{ov}^+ (Z,({\mathcal H}_\xi)_{\xi\in Z}),
$$
for short.
For instance, in the context of a nilpotent Lie group $G$, we will consider the space $Z = G \times \widehat{G}$ with the field of Hilbert spaces $(\cH_\pi)_{(x,\pi)\in G\times\widehat{G}}$.

\subsubsection{The $C^*$-algebras $\cA$ and its topological dual}

A measurable symbol $\sigma = \{\sigma(x,\pi)\,:\,(x,\pi)\in\Eng\times\Enghat\}$ is said to be bounded when there exists a constant $C>0$ such that for $dx d\mu_{\Enghat}(\pi)$-almost 
all $(x,\pi)\in\Eng\times\Enghat$, we have $\|\sigma(x,\pi)\|_{\cH_\pi} \leq C$. We denote by $\|\sigma\|_{L^{\infty}(\Eng\times\Enghat)}$ the smallest such constant $C>0$ and by 
$L^{\infty}(\Eng\times\Enghat)$ the space of bounded measurable symbols. This vector space endowed with this norm is a von Neumann algebra.

Clearly, $\cA_0$ is a subalgebra of $L^{\infty}(\Eng\times\Enghat)$ and we consider its closure for the norm $\|\cdot\|_{L^\infty(\Eng\times\Enghat)}$, denoted
by $\cA$. The following property clarifies the introduction of operator-valued measures.

\begin{proposition}[\cite{FF2}]
    \label{prop:dualA}
    The topological dual $\cA^*$ of the sub-$C^*$-algebra $\cA$ is isomorphic to the Banach space of operator-valued measures $\cM_{ov}(\Eng\times\Enghat)$ via
    \begin{equation*}
        \Gamma d\gamma \in \cM_{ov}(\Eng\times\Enghat) \mapsto \ell_{\Gamma d\gamma},\quad \ell_{\Gamma d\gamma}(\sigma) := \int_{\Eng\times\Enghat} {\rm Tr}\left(\sigma(x,\pi)\Gamma(x,\pi)\right)d\gamma(x,\pi).
    \end{equation*}
    Moreover, the isomorphism is isometric:
    \begin{equation*}
        \|\ell_{\Gamma d\gamma}\|_{\cA^*} = \|\Gamma d\gamma\|_{\cM_{ov}(\Eng\times\Enghat)},
    \end{equation*}
    and the positive linear functionals on $\cA$ are the $\ell_{\Gamma d\gamma}$ with $\Gamma d\gamma \in \cM_{ov}^{+}(\Eng\times\Enghat)$.
\end{proposition}

\subsubsection{Semiclassical measures and $\hbar$-oscillations}
We are now able to characterize the limits the following linear forms
\begin{equation*}
    \sigma\in\cA_0\mapsto \left(\Op_\hbar(\sigma)\psi_{0}^\hbar,\psi_{0}^\hbar\right)_{L^2(\Eng)}\in\C.
\end{equation*}

\begin{theorem}[\cite{FF2}]
    \label{theo:measures}
    Let $(\psi_{0}^\hbar)_{\hbar>0}$ be a bounded family of $L^2(\Eng)$. There exist a sequence $(\hbar_k)_{k\in\N}$ in $(0,+\infty)$
    with $\hbar_k \Tend{k}{+\infty} 0$ and a unique element $\Gamma d\gamma\in\cM_{ov}^+(\Eng\times\Enghat)$ such that we have
    \begin{equation*}
        \forall \sigma\in\cA_0,\ \left(\Op_{\hbar^k}(\sigma)\psi_{0}^{\hbar_k},\psi_{0}^{\hbar_k}\right)_{L^2(\Eng)}\Tend{k}{+\infty}
        \int_{\Eng\times\Enghat} {\rm Tr}\left(\sigma(x,\pi)\Gamma(x,\pi)\right)d\gamma(x,\pi).
    \end{equation*}
    Moreover, $\Gamma d\gamma$ satisfies the following inequality
    \begin{equation*}
        \int_{\Eng\times\Enghat}{\rm Tr}\left(\Gamma(x,\pi)\right)d\gamma(x,\pi)\leq \limsup_{\hbar>0}\|\psi_{0}^{\hbar}\|_{L^2(\Eng)}^2.
    \end{equation*}
    Given the sequence $(\hbar_k)_{k\in\N}$, $\Gamma d\gamma$ is the only operator-valued measure for which the above convergence holds.
\end{theorem}

Any equivalence class $\Gamma d\gamma$ satisfying Theorem \ref{theo:measures} is called a \textit{semiclassical measure} of the family $(\psi_{0}^\hbar)_{\hbar>0}$.

As the last inequality in Theorem \ref{theo:measures} shows, a loss of mass is possible. It is in this context that the assumptions \eqref{eq:assumptionoscillation} of $\hbar$-oscillations on 
a family $(\psi_{0}^\hbar)_{\hbar>0}$ becomes particularly relevant as it legitimates the use of our semiclassical pseudodifferential calculus to investigate the weak limits of 
the probabilty densities $|\psi_{0}^\hbar(x)|^2 dx$.

\begin{proposition}
    \label{prop:hbaroscillsemimeas}
    Let $(\psi_{0}^\hbar)_{\hbar>0}$ be a $\hbar$-oscillating bounded family of $L^2(\Eng)$ admitting a semiclassical measure $\Gamma d\gamma$ for the sequence $(\hbar_k)_{k\in\N}$, then for 
    all $\phi\in C_{0}^\infty(G)$
    \begin{equation*}
        \lim_{k\rightarrow+\infty}\int_{\Eng}\phi(x)|\psi_{0}^\hbar(x)|^2 dx = \int_{\Eng\times\Enghat}\phi(x){\rm Tr}\left(\Gamma(x,\pi)\right)d\gamma(x,\pi).
    \end{equation*}
\end{proposition}

\subsubsection{The sub-$C^*$-algebra $\cB$ and its topological dual}
We introduce the subspace $\cB_0(\Eng)$ of $\cA_0(\Eng)$ of symbols commuting with $H$. This subspace is not trivial as it contains all symbols of the form 
$a(x)\chi(H)$, $a\in C_{c}^\infty(\Eng)$, $\chi\in\cS(\R)$, by Hulanicki's theorem (see \cite{Hul}):

\begin{theorem}[Hulanicki]
    The convolution kernel of a spectral multiplier $\chi(\cR)$ of a Rockland operator $\cR$ on nilpotent Lie group $G$ for a Schwartz function $\chi\in\cS(\R)$ is Schwart on $G$.
\end{theorem}

We denote by $\cB$ the closure of $\cB_0$ for the norm $\|\cdot\|_{L^{\infty}(\Eng\times\Enghat)}$. We check readily that $\cB$ is a sub-$C^*$-alegbra of $\cA$ and that $\cB_0 = \cA_0 \cap\cB$.
The next statement identifies the dual of the topological dual of $\cB$.

\begin{proposition}[\cite{FFF}]
    \label{prop:meascommute}
    Via $\Gamma d\gamma \mapsto \restriction{\ell_{\Gamma d\gamma}}{\cB}$, the topological dual $\cB^*$ of $\cB$ is isomorphic to the closed subspace $\cM_{ov}(\Eng\times\Enghat)^{(H)}$ of operator-valued measures 
    $\Gamma d\gamma\in\cM_{ov}(\Eng\times\Enghat)$ such that the operator $\Gamma(x,\pi)$ commutes with the spectral projector $\Pi_n(\pi)$ for all $n\in\N_{>0}$ and for $d\gamma$-almost every $(x,\pi)\in\Eng\times\Enghat$.
\end{proposition}

Such a proposition allows us to use symbols only in $\cB_0$ to study our semiclassical measure once we have proved it belongs to the subspace $\cM_{ov}(\Eng\times\Enghat)^{(H)}$.

\subsubsection{Time-averaged semiclassical measures}
\label{subsec_timeav_scm}
When investigating semiclassical measures associated to solutions of the Schrödinger equation \eqref{eq:SchEngel}, time-averaging allows one to make their study more tractable.
For this reason, for a bounded family $(\psi^\hbar(t))_{\hbar>0}$ in $L^\infty(\R_t, L^2(\Eng))$, we associate the quantities 
\begin{equation}
 \label{def:leps}
\ell_{\hbar}(\theta, \sigma)=\int_{\R} \theta(t)  \left(\Op_\hbar (\sigma) \psi^\hbar(t),\psi^\hbar(t)\right)_{L^2(\Eng)} dt,\;\;\sigma\in \mathcal A_0,\;\;\theta\in L^1(\R_t),
\end{equation}
the limits of which are characterized by a map in $L^\infty(\R_t,{\mathcal M}_{ov}^+(\Eng\times \Enghat)$.

\begin{theorem}\label{theo:timemeasures}
Let $(\psi^\hbar(t))_{\hbar>0}$ be a bounded family in $L^\infty(\R_t,L^2(\Eng))$. There exists a sequence $(\hbar_k)_{k\in \N}$ in $(0,+\infty)$ with  $\hbar_k\Tend{k}{+\infty}0$
and a  map  $t\mapsto \Gamma_t d\gamma_t $ in
$L^\infty(\R_t, {\mathcal M}_{ov}^+(\Eng\times \Enghat))$ such that we have for all $ \theta\in L^1(\R_t)$ and $\sigma\in {\mathcal A}$,
\begin{equation}
\int_{\R} \theta(t) \left({\rm Op}_{\hbar_k} (\sigma) \psi^{\hbar_k}(t),\psi^{\hbar_k}(t)\right)_{L^2(\Eng)}dt 
\Tend {k}{+\infty} 
\int_{\R\times \Eng\times \Enghat} \theta(t){\rm Tr}\left(\sigma(x,\pi) \Gamma_t(x,\pi)\right)d\gamma_t(x,\pi) dt.
\end{equation}
Given the sequence $(\hbar_k)_{k\in \N}$, 
the map $t\mapsto \Gamma_t d\gamma_t$  
is unique up to equivalence. Besides,
\begin{equation}
\int_{\R}\int_{\Eng\times \Enghat} 
{\rm Tr}\left( \Gamma_t(x,\pi) \right) d\gamma_t(x,\pi) \, dt \leq  \limsup_{\hbar\rightarrow 0}\| \psi^\hbar\|_{L^\infty(\R,L^2(\Eng))}.
\end{equation}
Given the sequence $(\hbar_k)_{k\in\N}$, $t\mapsto \Gamma_t d\gamma_t$ is the only operator-valued measure for which the above convergence holds.
\end{theorem}

We call the map $t\mapsto\Gamma_t d\gamma_t$  satisfying Theorem~\ref{theo:timemeasures} (for some subsequence $\hbar_k$) a {\it time-averaged semiclassical measure} of the family $(\psi^\hbar(t))$.
When this family is solution of a Schrödinger equation, we expect the time-averaged semiclassical measures to enjoy additionnal properties. 

\subsubsection{Spectral projectors $\Pi_n$ and approximation by smoothing symbols}
\label{subsec:preliminaries}
In this section, we analyse the  fields of the spectral projectors $\Pi_n$ of $H$ in the light of the elements of harmonic analysis and pseudodifferential calculus introduced above. 
In contrast with the spectral projectors associated to the symbol of the subLaplacian on a H-type group, where we can prove their appartenance 
in homogeneous regular symbol classes (see \cite{FF2},\cite{FF3}), the fields $\Pi_n$ do not enjoy the same property.
$\Pi_n$ is a measurable symbol, and even a 0-homogeneous element of $L^\infty(\Enghat)$, but the behaviour of the eigenvalues $(\mu_n)_{n\in\N_{>0}}$ 
at infinity is an obstruction to the boundedness of their derivatives (i.e with respect to the difference operators).

We can however approximate these spectral projectors by smoothing symbols in a simple way.

\begin{proposition}
    \label{prop:Pin}
    Let $n\in\N_{>0}$ and consider the spectral projector $\Pi_n\in L^\infty(\Enghat)$. Let $(\chi_\eps)_{\eps\in(0,1)}$ a family of functions in $\cS(\R)$ null in a neighborhood of 0 and converging uniformly on
    any compact of $\R\setminus\{0\}$ towards 1. 
    Then $(\chi_\eps(H)\Pi_n)_{\eps\in (0,1)}$ is a family of smoothing symbols bounded in $L^\infty(\Enghat)$, and the following convergence holds punctually almost everywhere on $\Enghat$:
    \begin{equation*}
        \chi_\eps(H)\Pi_n \Tend{\eps}{0} \Pi_n.
    \end{equation*} 
\end{proposition}

\begin{proof}
    For $\eps>0$ fixed, observe that
    \begin{equation*}
        \chi_\eps(H)\Pi_n = \chi_\eps(\mu_n)\Pi_n,
    \end{equation*}
    and is thus supported inside a compact $K\subset \fg_{3}^* \setminus\{0\}\times\fr_{2}^*$: indeed, writing $\mu_{n}(\delta,\beta) = \delta^{2/3}\Tilde{\mu}_n(\beta \delta^{-1/3})$ and knowing that $\nu\in\R \mapsto \Tilde{\mu}_{n}$
    diverges to $+\infty$ when $\nu \rightarrow\pm \infty$, we conclude easily. 
    
    We would like to write $\Pi_n$ on $K$ as a Riesz projector of the symbol $H$.
    As the graphs of the eigenvalues $(\mu_m)_{m\in\N_{>0}}$ are not flat, we decompose $\chi_\eps = \sum_{j=1}^{N_\eps}\chi_{\eps,j}^2$, with the supports of the $\chi_{\eps,j}$ having a finite number of intersection 
    with one another and such that the $n$-th eigenvalue $\mu_{n}(\delta,\beta)$ is at a distance $d>0$ from the rest of the spectrum of $H(\pi^{\delta,\beta})$, uniformly on $K$. 
    We can then find a closed curve $\cC_{n}^j$ in $\C$ for which we can write
    \begin{equation*}
        \chi_{\eps,j}^2(H)\Pi_n = \chi_{\eps,j}(H)\int_{\cC_{n}^j}\chi_{\eps,j}(\mu_n)\left(H - z{\rm Id}\right)^{-1}dz.
    \end{equation*}
    Using this representation, we are going to estimate its seminorms introduced in Section \ref{subsubsect:generalclasssym}. 
    First notice that $\chi_{\eps,j}(H)$ is a smoothing operator by Hulanicki's theorem, and that the symbol $H$ is in $\dot{S}^2(\Eng)$, i.e a regular 2-homogeneous symbol
    in the sense of \cite[Definition 4.1]{FF}. By Leibniz formula \ref{prop:formuleLeibniz}, we can concentrate on the symbol given by the contour integral on $\cC_{n}^j$.
    First, for $\alpha=\beta=0$ and for all $z\in\cC_{n}^j$, we have 
    \begin{equation*}
        \sup_{\pi^{\delta,\beta}\in\Enghat_{\rm gen}} \left\|\chi_{\eps,j}(\mu_n(\delta,\beta))\left(H(\pi^{\delta,\beta}) - z{\rm Id}\right)^{-1}\right\|_{\cL\left(\cH_{\pi^{\delta,\beta}}\right)} \leq \frac{1}{d} < \infty,
    \end{equation*}
    so $\chi_{\eps,j}(\mu_n)\left(H - z{\rm Id}\right)^{-1}$ is in $L^\infty(\Enghat)$.

    Now for $i\in\{1,2\}$, since $\Delta_i {\rm Id} = 0$, the Leibniz formula \ref{prop:formuleLeibniz} implies 
    \begin{equation*}
        \Delta_{i} \left(H(\pi^{\delta,\beta}) - z{\rm Id}\right)^{-1} = -\left(H(\pi^{\delta,\beta}) - z{\rm Id}\right)^{-1}\Delta_i H(\pi^{\delta,\beta})\left(H(\pi^{\delta,\beta}) - z{\rm Id}\right)^{-1}.
    \end{equation*}
    So we have,
    \begin{multline*}
        \Delta_i \left(\chi_{\eps,j}(\mu_n(\delta,\beta))\left(H(\pi^{\delta,\beta}) - z{\rm Id}\right)^{-1}\right) = \Delta_i \left(\chi_{\eps,j}(\mu_n){\rm Id}\right)(\pi^{\delta,\beta}) \left(H(\pi^{\delta,\beta}) - z{\rm Id}\right)^{-1}\\ 
        -\chi_{\eps,j}(\mu_n(\delta,\beta))\left(H(\pi^{\delta,\beta}) - z{\rm Id}\right)^{-1}\Delta_i H(\pi^{\delta,\beta})\left(H(\pi^{\delta,\beta}) - z{\rm Id}\right)^{-1}.
    \end{multline*}
    The first term in the right-hand side is treated in the same way as the seminorm estimate for $\alpha=\beta=0$, since the difference operators act on $\chi_{\eps,j}(\mu_n)$ only by differentiation, and we conclude by the regularity of the eigenvalue and the support condition. For the second term,
    observe that the symbol $\Delta_i H$ is in $\dot{S}^{1}(\Eng)$ and by \cite[Proposition 4.2]{FF}, its composition with $\chi_{\eps,j}(H)$ gives a smoothing operator.
    We estimate the norms of the resolvents as previously thanks to the support condition of $\chi_{\eps,j}(\mu_n)$, and we obtain boundedness in $L^\infty(\Enghat)$.

    For $i\in\{3,4\}$, the Leibniz formula is more involved in the sense that terms involving difference operators of lower order terms appear, but are estimated thanks to the previous computation.
    We prove recursively this way that $\chi_{\eps,i}^2(H)\Pi_n$ is a symbol in $S^0(\Eng)$. 
    
    To prove that this symbol is in $S^{-\infty}(\Eng)$, we just observe that it would be equivalent to prove that, for $m\in\R$, the symbol 
    $\chi_{\eps,j}(\mu_n)\chi_{\eps,j}(H)({\rm Id + H})^{m}\left(H - z{\rm Id}\right)^{-1}$ is in $S^0$: the proof is the same as $\chi_{\eps,j}(H)({\rm Id}+H)^{m} = \chi_{\eps,j}^{m}(H)$ with $\chi_{\eps,j}^m (\lambda)= \chi_{\eps,j}(\lambda)(1+\lambda)^m$ and our analysis holds
    by similar support considerations.
\end{proof}

The previous proposition allows us to localize on the eigenspaces of $H$ through smoothing operators. More precisely, for $\sigma\in \cA_{0}$ a smoothing symbol, we can 
then define the symbol $\sigma^{(n,n')}:=\Pi_n\sigma \Pi_{n'}$ for each $n,n'$.
We see that for $n=n'$, $\sigma^{(n,n')}$ commutes with $H$, while for general pairs of integers $(n,n')$, it satisfies
\begin{equation*}	
H \circ \sigma^{(n,n')} =\mu_{n} \sigma^{(n,n')}
\quad\mbox{and}\quad
\sigma^{(n,n')}\circ H = \mu_{n'}\sigma^{(n,n')}.
\end{equation*}
However, $\sigma^{(n,n')}$ is usually not smoothing, but can by approximated nicely by symbols of $\cA_0$.

\begin{proposition}
    \label{prop:approxsymPi}
    Let $\sigma\in\cA_0(\Eng)$. With the same notations as in Proposition \ref{prop:Pin}, for $n,n'\in\N_{>0}$, the sequence of symbols $(\sigma_{\eps}^{(n,n')})_{\eps>0}$ defined by 
    \begin{equation*}
        \sigma_{\eps}^{(n,n')} = \chi_{\eps}(H)\Pi_n \sigma \Pi_n'\chi_{\eps}(H),
    \end{equation*}
    consists of symbols in $\cA_0$ uniformly bounded in $L^\infty(\Enghat)$ and converging towards $\sigma^{(n,n')}$ almost everywhere on $\Eng\times\Enghat$.
\end{proposition}

As proofs will make it clear, we sometimes need to work with symbols vanishing uniformly in a neighborhood of $\{\delta=0\}$. The following proposition allows us to do so.

\begin{proposition}
    Let $\sigma\in\cA_0(\Eng)$.
    \begin{enumerate}
        \item For any $g\in C_{c}^\infty(\R\setminus\{0\})$, the symbol given by $(x,\pi^{\delta,\beta})\in\Eng\times\Enghat{\rm gen}\mapsto g(\delta)\sigma(x,\pi^{\delta,\beta})$ is in $\cA_0$.
        \item Let $(g_n)_{n\in\N} \subset C_{c}^\infty(\R\setminus\{0\})$ be a sequence of smooth functions bounded by 1 and converging to the constant 1 uniformly on any compact of $\R\setminus\{0\}$. Then the 
        sequence of smooth symbols $(g_n \sigma(x,\cdot))_{n\in\N}$ converges to $\sigma(x,\cdot)$ in $L^{\infty}(\Enghat)$ uniformly in $x\in\Eng$.
    \end{enumerate}
\end{proposition}

\begin{proof}
   Let $\kappa_\sigma$ be the kernels associated to the symbols which may be identified with a map in $C_{c}^\infty(\Eng,\cS(\Eng))$. Part (1) follows from the kernel $\kappa_g$ associated to the symbol $g(\delta$) being the central
    distribution $\delta_{(x_1,x_2,x_3)=(0,0,0)}\otimes \cF_{\R}^{-1}g$ and the kernel associated to $g(\delta)\sigma$ being $\kappa_\sigma * \kappa_g$. Part (2) follows from the Lebesgue dominated convergence theorem and the Plancherel formula.
\end{proof}

\subsection{Propagation of semiclassical measures at the semiclassical time-scale}

We investigate here the propagation properties of the time-averaged semiclassical measures associated to a family $(\psi^{\hbar}(t))_{\hbar>0}$ of solutions to the Schrödinger equation for the semiclassical time-scale:
\begin{equation}
    \label{eq:Sch_engel}
    \begin{cases}
        & i\hbar\partial_{t}\psi^{\hbar} = -\hbar^{2}\Delta_{\Eng}\psi^{\hbar},\\
        & \psi_{|t=0}^{\hbar} = \psi_{0}^{\hbar}.
    \end{cases}
\end{equation}
Our main result is the following one:

\begin{theorem}
\label{thm:prop_measures} 
Let  $(\psi_0^\hbar)_{\hbar>0}$ be a bounded family in $L^2(\Eng)$ and $\psi^\hbar(t) = {\rm e}^{i \hbar t \Delta_{\Eng}} \psi^\hbar_0$ be the solution to \eqref{eq:Sch_engel}.
Then any time-averaged semiclassical measure 
$t\mapsto \Gamma_t d\gamma_t\in L^\infty(\R_t, {\mathcal M}_{ov}^+(\Eng\times \Enghat))$ for the family $(\psi^\hbar(t))_{\hbar>0}$ satisfies the following additional properties:
\begin{itemize}
\item[(i)] 
For almost every $(t,x,\pi)\in \R_t\times \Eng\times \Enghat$, the operator $\Gamma_t(x,\pi)$ commutes with $\mathcal{F}(-\Delta_{\Eng})(\pi)=H(\pi)$. 

\item[(ii)] For a generic representation $\pi^{\delta,\beta}$, $(\delta,\beta) \in \fg_{3}^*\setminus\{0\} \times \fr_{2}^{*}$, we have the following decomposition 
\begin{equation}\label{eq:decomp}
\Gamma_t(x,\pi^{\delta,\beta})=\sum_{n\in\N } \Gamma_{n,t}(x,\pi^{\delta,\beta})\;\;{ with}\;\; \Gamma_{n,t}(x,\pi^{\delta,\beta}):= \Pi_n(\pi^{\delta,\beta})\Gamma_t(x,\pi^{\delta,\beta}) \Pi_n(\pi^{\delta,\beta}).
\end{equation}

Moreover for each $n\in \N$, the operator-valued measure $d\gamma_{n,t}(x,\pi^{\delta,\beta}) = \Gamma_{n,t}(x,\pi^{\delta,\beta}) d\gamma_t(x,\pi^{\delta,\beta})$ on $\R_t\times \Eng\times (\fg_{3}^*\setminus\{0\} \times \fr_{2}^{*})$ is a weak solution of the following transport equation:
\begin{equation*}
\left(\partial_t -\partial_{\beta}\mu_{n}(\delta,\beta)X_{2}\right)\left(d\gamma_{n,t}(x,\pi^{\delta,\beta})\right)=0.
\end{equation*}

\item[(iii)] On the non-generic part of the dual, i.e above $\delta=0$, the semiclassical measure splits into two parts. The first one corresponds to the one supported in the Schrödinger representations
$\{\pi^{\lambda}\,:\,\lambda\in\fg_{2}^*\setminus\{0\}\}$, that we denote by $\bm{1}_{\delta=0,\gamma\neq 0}\Gamma_{t}d\gamma_{t}$ and is invariant with respect to time:
\begin{equation*}
    \partial_{t}\left(\bm{1}_{\delta=0,\gamma\neq 0}\Gamma_{t}(x,\pi^{\lambda})d\gamma_t(x,\pi^{\lambda})\right) = 0.
\end{equation*}
The second part is supported in the finite dimensional ones $\{\pi^{(0,\alpha)}\,:\,\alpha\in\fg_{1}^*\}$, that we denote by $\bm{1}_{\delta=0,\gamma= 0}\Gamma_{t}d\gamma_{t}$, and satisfies the following 
transport equation:
\begin{equation*}
    (\partial_t -\alpha\cdot(X_1,X_2))\left(\bm{1}_{\delta=0,\gamma= 0}\Gamma_{t}(x,\pi^{(0,\alpha)})d\gamma_{t}(x,\pi^{(0,\alpha)})\right)=0.
\end{equation*}
\end{itemize}
\end{theorem} 

This result extends Theorem \ref{theo:propagintro} as it also covers the propagation of the part of the time-averaged semiclassical measure supported in the non-generic representations. Unsurprisingly, we find the same
propagation laws that have been found in \cite{FF3} for the time-averaged semiclassical measures on the Heisenberg group: we find that for the finite dimensional representations, the propagation is given by 
the Hamiltonian flow with impulsion in the dual of the first stratum $\fg_{1}^*$, while for the Schrödinger representations $\{\pi^\lambda\,:\,\lambda\in\fg_{2}^*\setminus\{0\}\}$ the 
semiclassical measures are invariant in time.  

In order to be able to determine fully the time-averaged semiclassical measures of solutions to \eqref{eq:Sch_engel}, we need the following result concerning their continuity with respect to the time variable.

\begin{proposition}
    \label{prop:continuitytimesm}
    Consider as in Theorem \ref{thm:prop_measures} the time-averaged semiclassical measure $t\mapsto \Gamma_{t} d\gamma_{t}\in {\mathcal M}_{ov}^+(\Eng\times \Enghat)$ 
    corresponding to the family of solutions to the Schrödinger equation \eqref{eq:Sch_engel}. Then for any symbol $\sigma\in \cB_{0}(\Eng)$, i.e commuting
    with the field of $H$, the map
    \begin{equation*}
        t\in\R \mapsto \int_{\Eng\times \Enghat}{\rm Tr}\left(\sigma(x,\pi)\Gamma_{t}(x,\pi) \right)d\gamma_{t}(x,\pi),
    \end{equation*}
    is locally Lipshitz on $\R$.
\end{proposition}

\subsection{Proof of Theorem \ref{thm:prop_measures} and Proposition \ref{prop:dispersionsupp}}

Let $(\psi_{0}^\hbar)_{\hbar>0}$ a bounded family in $L^2(\Eng)$ and write $\psi^\hbar(t) = e^{it\hbar\Delta_\Eng}\psi_{0}^\hbar$ for $t\in\R$. We suppose given a time-averaged semiclassical measure 
$t\mapsto\Gamma_t d\gamma_t$ of the family $(\psi^\hbar(t))_{\hbar>0}$ as in Theorem \ref{theo:timemeasures}. In what follows we will make the abuse of writing $\hbar \rightarrow 0$ instead of taking the limit along the subsequence 
defining $\Gamma_t d\gamma_t$. We also introduce the following notation:
\begin{equation*}
    \forall \theta\in C_{c}(\R_t),\,\forall \sigma\in\cA_0,\ \ell_{\infty}(\theta,\sigma) = \int_{\R}\theta(t)\int_{\Eng\times\Enghat}{\rm Tr}\left(\sigma(x,\pi)\Gamma_t(x,\pi)\right)d\gamma_t(x,\pi)dt.
\end{equation*}

\subsubsection{Proof of Theorem~\ref{thm:prop_measures} (i)}
\label{subsubsect:thmpropmeasuresi}
We take $\sigma\in{\mathcal A}_0$. 
The Schrödinger equation \eqref{eq:Sch_engel} yields
\begin{equation*}
i\hbar \frac{d}{dt} \left(\Op_{\hbar}(\sigma) \psi^{\hbar}(t) ,\psi^{\hbar}(t) \right)_{L^2(\Eng)} 
=  \left(\left[\Op_{\hbar}(\sigma) , -\hbar^2\Delta_{\Eng} \right] \psi^{\hbar}(t) ,\psi^{\hbar}(t) \right)_{L^2(\Eng)}.
\end{equation*}
By use of the symbolic calculus given by \eqref{eq:commutatorDelta}, we obtain 
\begin{eqnarray*}
\left(\Op_{\hbar}\left(\left[\sigma,H\right]\right) \psi^{\hbar}(t) ,\psi^{\hbar}(t) \right)_{L^2(\Eng)}
&=&  
\left(\left[\Op_{\hbar}(\sigma) , -\hbar^{2}\Delta_{\Eng} \right] \psi^{\hbar}(t) ,\psi^{\hbar}(t) \right)_{L^2(\Eng)} + \cO(\hbar) +\cO(\hbar)\\
&=&i\hbar \frac{d}{dt} \left(\Op_{\hbar}(\sigma)\psi^{\hbar}(t) ,\psi^{\hbar}(t) \right)_{L^2(\Eng)}+\cO(\hbar).
\end{eqnarray*}
Therefore, for $\theta \in \mathcal{C}_{c}^{\infty}(\R_t)$, by integration by part, we find
\begin{equation}
\label{eq:commutator}
\int_{\R} \theta(t) \left(\Op_{\hbar}\left([\sigma,H]\right) \psi^{\hbar}(t) ,\psi^{\hbar}(t) \right)_{L^2(\Eng)} dt = \cO(\hbar).
\end{equation}

Taking the limit as $\hbar\to0$ in \eqref{eq:commutator}, we obtain 
\begin{equation}
    \label{eq:proofcomuttator}
\ell_{\infty}(\theta,[\sigma,H])=0.
\end{equation}
Using Proposition \ref{prop:approxsymPi}, we can apply this to the smoothing symbol $\sigma_{\eps}^{(n,n')}$ for some $\eps\in(0,1)$ and $n,n'\in\N_{>0}$: 
by properties of the trace, we see that for $\pi^{\delta,\beta}\in\Enghat_{\rm gen}$,
\begin{multline*}
    {\rm Tr}\left([\sigma_{\eps}^{(n,n')}(x,\pi^{\delta,\beta}),H(\pi^{\delta,\beta})]\, \Gamma_t(x,\pi^{\delta,\beta})\right) 
    = {\rm Tr}\Bigl(\sigma(x,\pi^{\delta,\beta}) \mu_{n'}(\delta,\beta)\chi_{\eps}(H)\Pi_{n'}(\pi^{\delta,\beta}) \Gamma_t(x,\pi^{\delta,\beta}) \\
    -\mu_{n}(\delta,\beta)\chi_{\eps}(H)\Pi_n(\pi^{\delta,\beta}) \sigma(x,\pi^{\delta,\beta})\Gamma_t(x,\pi^{\delta,\beta})\Bigr) \\
    = {\rm Tr}\left(\sigma(x,\pi^{\delta,\beta})[H,\chi_{\eps}(H)\Pi_{n'}(\pi^{\delta,\beta})\Gamma_t(x,\pi^{\delta,\beta})\chi_{\eps}(H)\Pi_n(\pi^{\delta,\beta})]\right).
\end{multline*}
and we deduce that for any $\sigma\in\cA_0$,
\begin{equation*}
    \int_{\R}\theta(t)\int_{\Eng\times \Enghat} {\rm Tr}\left(\sigma(x,\pi)[H(\pi),\chi_{\eps}(H)\Pi_{n'}(\pi)\Gamma_t(x,\pi)\chi_{\eps}(H)\Pi_n(\pi)]\right)d\gamma_t(x,\pi)dt=0.
    \end{equation*}
This implies that 
the element in $\mathcal M_{ov}(\R\times \Eng\times \Enghat)$ given by  
\begin{equation*}
[H(\pi) \, , \, 
\chi_\eps(H)\Pi_{n'}(\pi)
\Gamma_t(x,\pi)
\chi_\eps(H)\Pi_n(\pi)] \,
 d\gamma_t(x,\pi) dt 
\end{equation*}
is zero. For any pair of integers with $n\not=n'$, 
taking $\eps\to 0$, this implies that
$\Pi_{n'} \Gamma_t \Pi_{n}=0$
for $d\gamma_t dt$-almost every $(t,x,\pi)\in \R_t\times\Eng\times \Enghat$. This shows Point~(i) of Theorem~\ref{thm:prop_measures}.

For the first claim of Point~(ii), as $\Gamma_t$ is a positive compact operator, the spectral decomposition~\eqref{eq:decomp} follows easily.

\subsubsection{A more precise computation} 
We write down more precisely the equation that \eqref{eq:commutatorDelta} allows us to get:
\begin{equation*}
    \begin{aligned}
    i\hbar\frac{d}{dt}\left(\Op_{\hbar}(\sigma) \psi^{\hbar}(t) ,\psi^{\hbar}(t) \right)
    &= \left(\Op_{\hbar}\left(\left[\sigma,H\right]\right) \psi^{\hbar}(t) ,\psi^{\hbar}(t) \right)
    + 2\hbar \left(\Op_{\hbar}\left((\pi(V)\cdot V)\sigma\right) \psi^{\hbar}(t) ,\psi^{\hbar}(t) \right)\\
    &+ \hbar^2 \left(\Op_{\hbar}\left(\Delta_{\Eng}\sigma\right) \psi^{\hbar}(t) ,\psi^{\hbar}(t) \right).
    \end{aligned}
\end{equation*}
Using the notation $\ell_{\hbar}(\theta,\sigma)$ introduced in \eqref{def:leps}, we can rewrite the previous equation more concisely after integrating against $\theta$:
\begin{equation}
    \label{eq:ell_commutator_rel}
    -i\hbar\ell_{\hbar}(\theta',\sigma) = \ell_{\hbar}(\theta,\left[\sigma,H\right]) + 2\hbar\,\ell_{\hbar}(\theta,(\pi(V)\cdot V)\sigma) + \hbar^2 \ell_{\hbar}(\theta,\Delta_{\Eng}\sigma).
\end{equation}
 
To be able to see the propagation, we now use a symbol $\sigma\in\cB_0$: the previous equation then becomes 
\begin{equation}
    \label{eq:proofellB0}
    -\ell_{\hbar}(\theta',\sigma) = -2i \ell_{\hbar}(\theta,(\pi(V)\cdot V) \sigma) +\cO(\hbar).
\end{equation}
We have already observed that for any symbol $\Tilde{\sigma} \in \mathcal{A}_{0}$, we have 
\begin{equation*}
    \ell_{\hbar}(\theta,[\Tilde{\sigma},H]) = \cO(\hbar).
\end{equation*}
This property leads us to decompose the symbol $(\pi(V)\cdot V)\sigma$ as the sum of its $H$-diagonal part and its $H$-off-diagonal one: 
the off-diagonal part will write itself as a commutator with $H$ and thus will not be preserved in the limit $\hbar \rightarrow 0$,
while the diagonal one will give us the infinitesimal generator of the flow along which each measure $\gamma_{n,t} = \Gamma_{n,t} d\gamma_t$, $n\in\N$, propagates.

As the test symbol $\sigma$ has been in chosen in $\cB_0$, we are left with decomposing the operators $\pi^{\delta,\beta}(X_1)$
and $\pi^{\delta,\beta}(X_2)$. But first recall that by analyticity of the family of operators $(H(\pi^{\delta,\beta}))_{(\delta,\beta)\in\fg_{3}^*\setminus\{0\}\times\fr_{2}^*}$
and by differentiating with respect to the variable $\beta$, we have the following so-called Feynmann-Helmann formula:

\begin{lemma}
    \label{lem:FH1}
    For $n\in\N_{>0}$, we have
    \begin{equation}
        (H - \mu_{n}\,{\rm Id})\partial_{\beta}\Pi_n = (\partial_{\beta}\mu_n \,{\rm Id}-\partial_{\beta} H) \Pi_n.
    \end{equation}
\end{lemma}

\begin{remark}
    We observe easily that, as operators on the real line $\R_\xi$, we have 
    \begin{equation*}
        \partial_{\beta} H (\pi^{\delta,\beta}) = 2\left(\beta +\frac{\delta}{2}\xi^2\right) = -2 i \pi^{\delta,\beta}(X_2).
    \end{equation*}
\end{remark}

\begin{lemma}
    \label{lem:first_diag_part}
    Let $(\delta,\beta) \in \fg_{3}^*\setminus\{0\}\times\fr_{2}^*$. The operator $\pi^{\delta,\beta}(X_1)$ is entirely off-diagonal with respect to $H(\pi^{\delta,\beta})$:
    \begin{equation}
        \label{eq:diag_partX1}
        \pi^{\delta,\beta}(X_1) = \left[\frac{1}{2i\delta}\pi^{\delta,\beta}(X_{3}), H(\pi^{\delta,\beta})\right],
    \end{equation}
    while $\pi^{\delta,\beta}(X_2)$ decomposes as follows:
    \begin{equation}
        \label{eq:diag_partX2}
        \pi^{\delta,\beta}(X_2) = \sum_{n\in\N_{>0}} \frac{i}{2}\partial_\beta \mu_{n}(\delta,\beta)\Pi_n(\pi^{\delta,\beta}) +  \left[\frac{i}{2}\sum_{n\in\N_{>0}} \partial_{\beta}\Pi_{n}(\pi^{\delta,\beta}) \Pi_n(\pi^{\delta,\beta}),H(\pi^{\delta,\beta})\right].
    \end{equation}
\end{lemma}

\begin{proof}
    Relation \eqref{eq:diag_partX1} is a direct consequence of the following computation in the enveloping algebra of $\fg_\Eng$:
    \begin{equation*}
        [X_3, X_{1}^2+X_{2}^2] = 2X_1 X_4. 
    \end{equation*}
    To prove \eqref{eq:diag_partX2}, observe that Lemma \ref{lem:FH1} allows us to write for $n\in\N_{>0}$ and $\pi^{\delta,\beta}\in\Enghat$,
    \begin{equation*}
        -2i\pi^{\delta,\beta}(X_2)\Pi_n(\pi^{\delta,\beta}) = \partial_{\beta}\mu_{n}(\delta,\beta)\Pi_n(\pi^{\delta,\beta}) + (\mu_{n}(\delta,\beta)\,{\rm Id}-H(\pi^{\delta,\beta}))\partial_{\beta}\Pi_n(\pi^{\delta,\beta}).
    \end{equation*}
    Composing to the left by $\Pi_n(\pi^{\delta,\beta})$, we obtain 
    \begin{equation*}
        \Pi_n(\pi^{\delta,\beta})\pi^{\delta,\beta}(X_2)\Pi_n(\pi^{\delta,\beta}) = \frac{i}{2}\partial_{\beta}\mu_{n}(\delta,\beta)\Pi_n(\pi^{\delta,\beta}),
    \end{equation*}
    giving us the expected $H$-diagonal part of $\pi^{\delta,\beta}(X_2)$.

    For the $H$-off-diagonal one, observe that for $n'\in\N_{>0}$ with $n'\neq n$, composing to the left by $\Pi_{n'}(\pi^{\delta,\beta})$,  we obtain this time 
    \begin{equation}
        \label{eq:FHanticommut}
        \Pi_{n'}(\pi^{\delta,\beta})\pi^{\delta,\beta}(X_2)\Pi_n(\pi^{\delta,\beta})=(\mu_{n}(\delta,\beta)-\mu_{n'}(\delta,\beta))\Pi_{n'}(\pi^{\delta,\beta})\partial_{\beta}\Pi_n(\pi^{\delta,\beta}).
    \end{equation}
    We are left with finding a field of operator $C$ such that $\sum_{n'\neq n} \Pi_{n'}\pi(X_2)\Pi_{n} = [C,H]$.
    This field $C$ must then necessarily satisfy
    \begin{equation*}
        \Pi_{n'}(\pi^{\delta,\beta})\pi^{\delta,\beta}(X_2)\Pi_{n}(\pi^{\delta,\beta}) = (\mu_{n}(\delta,\beta) - \mu^{n'}(\delta,\beta))\Pi_{n'}(\pi^{\delta,\beta})C(\pi^{\delta,\beta})\Pi_{n}(\pi^{\delta,\beta}),
    \end{equation*}
    but confronting this equation with \eqref{eq:FHanticommut}, we deduce that we can take the field $C$ as follows
    \begin{equation*}
        C = \frac{i}{2}\sum_{n' \neq n} \Pi_{n'}\partial_{\beta}\Pi_{n}\Pi_{n} = \frac{i}{2}\sum_{n\in\N_{>0}} \partial_{\beta}\Pi_n \Pi_n,
    \end{equation*}
    last equality following from the fact that the $(\Pi_n)_{n\in\N_{>0}}$ are fields of projectors, and as such satisfy $\Pi_{n'}\partial_{\beta}\Pi_n \Pi_n = \partial_{\beta}\Pi_n \Pi_n$
    for $n'\neq n$.
\end{proof}

Accordingly, if we choose $\sigma$ to be concentrated on a particular eigenspace of $H$, meaning that for a certain $n\in\N_{>0}$ we have $\sigma =  \Pi_{n} \sigma \Pi_{n}$, and null in a neighborhood of $\{\delta=0\}$,
we obtain
\begin{equation*}
   -i(\pi(V)\cdot V)\sigma(x,\pi^{\delta,\beta}) = \frac{1}{2} \partial_{\beta} \mu_{n}(\delta,\beta) X_2 \sigma(x,\pi^{\delta,\beta}) + \frac{1}{2}\left[\sigma_{1}(x,\pi^{\delta,\beta}),H(\pi^{\delta,\beta})\right],
\end{equation*}
with 
\begin{equation}
    \label{eq:antidiagT}
    \sigma_{1} = -\frac{1}{\delta} \pi^{\delta,\beta}(X_3) X_1 \sigma +  \partial_{\beta}\Pi_n(\pi^{\delta,\beta}) X_2 \sigma \in\cA_0.
\end{equation}
Equation \eqref{eq:proofellB0} then can be rewritten as 
\begin{equation*}
    -\ell_{\hbar}(\theta',\sigma) = \ell_{\hbar}\left(\theta,\partial_{\beta}\mu_{n}X_2 \sigma\right) +\cO(\hbar),
\end{equation*}
which gives, when $\hbar$ goes to 0,
\begin{multline*}
    -\int_{\R}\theta'(t)\int_{\Eng\times\Enghat_{\rm gen}}{\rm Tr}\left(\sigma(x,\pi^{\delta,\beta})\Gamma_{n,t}(x,\pi^{\delta,\beta})\right)d\gamma_t(x,\pi^{\delta,\beta})dt = \\\int_{\R}\theta(t)\int_{\Eng\times\Enghat_{\rm gen}}{\rm Tr}\left(\partial_{\beta}\mu_n(\delta,\beta) X_2\sigma(x,\pi^{\delta,\beta})\Gamma_{n,t}(x,\pi^{\delta,\beta})\right)d\gamma_t(x,\pi^{\delta,\beta})dt.
\end{multline*}
We conclude by Proposition \ref{prop:meascommute} and Point (ii) of Theorem \ref{thm:prop_measures}.

\subsubsection{Proof of Proposition \ref{prop:continuitytimesm}} 
The following argument can already be found in \cite{FF3}. For any symbol $\sigma \in \cB_{0}(\Eng)$, we write
\begin{equation*}
    \ell_{\hbar,t}(\sigma) = \left(\Op_{\hbar}(\sigma)\psi^{\hbar}(t),\psi^{\hbar}(t)\right)_{L^2(\Eng)}.
\end{equation*}
Observe that by Equation \eqref{eq:proofellB0}, we have 
\begin{equation*}
    \frac{d}{dt} \ell_{\hbar,t}(\sigma) = \cO_{\sigma}(1),
\end{equation*}
in the sense that the map $t\mapsto \frac{d}{dt}\ell_{\hbar,t}(\sigma)$ is uniformly bounded in a given bounded interval of $\R$ and this uniformly for $\hbar\in (0,1)$.
Thus the family of maps $(\ell_{\hbar,\cdot}(\sigma))_{\hbar>0}$ is bounded and equicontinuous in $C_{b}(\R)$, so that by Arzelà-Ascoli theorem we can extract 
a subsequence $(\hbar_{j_k})_{k\in\N}$ such that, as $k\rightarrow +\infty$, $\hbar_{j_k} \rightarrow 0$ and $(\ell_{\hbar_{j_k},\cdot}(\sigma))_{k\in\N}$ converges uniformly 
on any bounded interval to a locally Lipshitz function $t\mapsto \ell_{t}(\sigma)$ for all $\sigma \in \cB_{0}(\Eng)$ (this last assertion follows from considering a dense 
subset of such symbols and a diagonal extraction procedure). 

It is straightforward to see that for $t\in\R$, $\ell_t$ defines a state on the $C^*$-algebra $\cB$. By Proposition \ref{prop:meascommute}, it is thus given by 
an operator-valued measure $\Gamma^{\ell_t} d\gamma^{\ell_t}$ in $\cM_{ov}^+(\Eng\times\Enghat)^{(H)}$. By Theorem \ref{theo:timemeasures} and by the property of uniform convergence on any bounded interval,
we conclude that the map $t\in\R \mapsto \Gamma^{\ell_t} d\gamma^{\ell_t}$ necessarily coincides with the time-averaged semiclassical measure $t\in\R \mapsto \Gamma_t d\gamma_t$ and the proposition is proved.

\subsubsection{Proof of Point~(iii)} 
We are interested here with the part of the semiclassical measure supported in the non-generic representations $\widehat{\Bar{\Heis}}$. We will denote by ${\bm 1}_{\widehat{\Bar{\Heis}}}\Gamma_t d\gamma_t$ the restrictions
of the operator-valued measures which can be seen as elements of $\cM_{ov}^+(\Eng\times\widehat{\Bar{\Heis}})$. More precisely, as Point (i) of Theorem \ref{thm:prop_measures} has been proved, the restriction of the
time-averaged semiclassical measure to $\Eng\times\widehat{\Bar{\Heis}}$ is in $L^\infty(\R_t,\cM_{ov}^+(\Eng\times\widehat{\Bar{\Heis}})^{(H)})$, where the we restrict the symbol $H$ to the non-generic representations.
As such we can only deal with symbols commuting with $H$.

To be able to study this restricted measure, we need to enlarge our set of test symbols. To this end we introduce the closures $vN\cA$ and $vN \cB$ of the $C^*$-algebras $\cA$ and $\cB$ respectively for the strong operator topology in $L^\infty(\Eng\times\Enghat)$. They correspond to the von Neumann algebras
generated by $\cA$ and $\cB$ respectively and are contained in $L^\infty(\Eng\times\Enghat)$. The semiclassical measures $\Gamma_t d\gamma_t \in \cA^*$ each have a unique linear continous extension to $vN \cA$, meaning it makes sense to 
integrate these measures on elements of this von Neumann algebra.

\begin{lemma}
    \label{lem:vNsymbol}
    If $\sigma\in\cA_0$, then ${\bm 1}_{\widehat{\Bar{\Heis}}}\sigma$ is in  $vN \cA$. Similarly, if $\sigma\in\cB_0$, then ${\bm 1}_{\widehat{\Bar{\Heis}}}\sigma$ is in  $vN \cB$. Moreover, for $\sigma\in\cB_0$, there exists a family $(\sigma_\eps)_{\eps\in(0,1)}$ of commuting symbols in $\cB_0$ such that 
    $\sigma_\eps$ and $(\pi(V)\cdot V)\sigma_\eps$ converge to ${\bm 1}_{\widehat{\Bar{\Heis}}}\sigma$ and ${\bm 1}_{\widehat{\Bar{\Heis}}}(\pi(V)\cdot V)\sigma$ respectively in $vN \cB$ and $vN\cA$, as $\eps$ goes to 0.
\end{lemma}

\begin{proof}
    Let $\sigma\in\cB_0$ and $\eta\in C_{c}^\infty(\fg_{3}^*)$ satisfying $\eta(0)=1$. Writing $\kappa$ the convolution kernel of $\sigma$, one easily check that 
    the symbol given by 
    \begin{equation*}
        \forall (x,\pi)\in\Eng\times\Enghat_{\rm gen},\ (\sigma \eta)(x,\pi^{\delta,\beta}) = \sigma(x,\pi^{\delta,\beta})\eta(\delta),
    \end{equation*}
    is in $\cB_0$ as its convolution kernel is given by $(x,y)\in\Eng^2 \mapsto (\kappa_{x}(y_1,y_2,y_3,\cdot)*_{\fg_3} \cF_{\fg_3}^{-1}\eta)(y_4)$.
    Introducing for $\eps\in(0,1)$ the function $\eta_\eps(\cdot) = \eta\left(\frac{\cdot}{\eps}\right)$, we consider the family of symbols given by 
    $\sigma_\eps = \sigma\eta_\eps$ as above. We check readily that the limit of $\sigma_\eps$ for the strong operator topology is given by ${\rm 1}_{\widehat{\Bar{\Heis}}}\sigma$.
    The same result holds for $(\pi(V)\cdot V)\sigma$.
\end{proof}

Thus, for $\sigma\in\cA_0$, the following equation makes sense:
\begin{equation*}
    \ell_{\infty}(\theta,{\bm 1}_{\widehat{\Bar{\Heis}}}\sigma) = \int_{\R}\theta(t)\int_{\Eng\times\widehat{\Bar{\Heis}}} {\rm Tr}\left(\sigma(x,\pi)\Gamma_t(x,\pi)\right)d\gamma_t(x,\pi)dt.
\end{equation*}

However, if the previous lemma allow us to consider more test symbols than those given by the $C^*$-algebras $\cA$ and $\cB$, we need to check that there are enough
of them to fully characterize the restricted measures ${\bm 1}_{\widehat{\Bar{\Heis}}}\Gamma_t d\gamma_t \in\cM_{ov}^+(\Eng\times\widehat{\Bar{\Heis}})^{(H)}$.
We will note $\cA_{0}(\Eng\times\widehat{\Bar{\Heis}})$ the space of fields of operators $\sigma = \{\sigma(x,\pi)\in\cL(\cH_\pi)\,:\,(x,\pi)\in \Eng\times\widehat{\Bar{\Heis}}\}$ 
given by the Fourier transform of kernels $\kappa$ in $C_{c}^\infty(\Eng,\cS(\Bar{\Heis}))$:
\begin{equation*}
    \sigma(x,\pi) = \cF_{\Bar{\Heis}}\kappa_x(\pi).
\end{equation*}
We will denote by $\cA(\Eng\times\widehat{\Bar{\Heis}})$ the closure of $\cA_{0}(\Eng\times\widehat{\Bar{\Heis}})$ for the norm $\|\cdot\|_{L^\infty(\Eng\times\widehat{\Bar{\Heis}})}$. 
We define the spaces $\cB_{0}(\Eng\times\widehat{\Bar{\Heis}})$ and $\cB(\Eng\times\widehat{\Bar{\Heis}})$ in the obvious way, similarly as the definition of the spaces $\cB_0$ and $\cB$.

\begin{lemma}\cite[Lemma 4.3]{FFF}
    The map $\Theta: \sigma\in\cA_0(\Eng)\mapsto \restriction{\sigma}{\Eng\times\widehat{\Bar{\Heis}}}\in\cA_0(\Eng\times\widehat{\Bar{\Heis}})$ is surjective and extends 
    to a $C^*$-algebra morphism from $\cA(\Eng)$ onto $\cA(\Eng\times\widehat{\Bar{\Heis}})$. The restriction of this map to $\cB(\Eng)$ is also surjective on 
    $\cB(\Eng\times\widehat{\Bar{\Heis}})$.
\end{lemma}

Now, for $\sigma\in\cB_{0}(\Eng)$, Equation \eqref{eq:proofellB0} gives in the limit $\hbar$ goes to 0,
\begin{equation*}
    -\ell_{\infty}(\theta',\sigma)=-2i\ell_{\hbar}(\theta,(\pi(V)\cdot V)\sigma).
\end{equation*}
Using the family of symbols $(\sigma_{\eps})_{\eps\in(0,1)}$ introduced in Lemma \ref{lem:vNsymbol}, we obtain 
\begin{equation*}
    -\ell_{\infty}(\theta',{\bm 1}_{\widehat{\Bar{\Heis}}}\sigma)=-2i\ell_{\hbar}(\theta,{\bm 1}_{\widehat{\Bar{\Heis}}}(\pi(V)\cdot V)\sigma),
\end{equation*}
or more explicitly 
\begin{multline*}
    -\int_{\R}\theta'(t)\int_{\Eng\times\widehat{\Bar{\Heis}}}{\rm Tr}\left(\sigma(x,\pi)\Gamma_{n,t}(x,\pi)\right)d\gamma_t(x,\pi)dt = \\\int_{\R}\theta(t)\int_{\Eng\times\widehat{\Bar{\Heis}}}{\rm Tr}\left((\pi(V)\cdot V)\sigma(x,\pi)\Gamma_{n,t}(x,\pi)\right)d\gamma_t(x,\pi)dt.
\end{multline*}
We have reduced the study of the propagation of the restriction of the semiclassical measure to the non-generic representation to the case of the Heisenberg group $\Bar{\Heis}$. Point (iii) of Theorem \ref{thm:prop_measures} then
follows from \cite[Theorem 2.10]{FF3}.

\subsubsection{Proof of Proposition \ref{prop:dispersionsupp}}

As the scheme of the proof is similar as the one presented above, we investigate here the properties of the time-averaged semiclassical measures
of a family of solutions, still denoted by $(\psi^\hbar(t))_{\hbar>0}$, for the Schrödinger equation with time-scale $\tau=2$:
\begin{equation*}
    \forall t\in\R,\ \psi^\hbar(t) = e^{it\Delta_\Eng}\psi_{0}^\hbar,
\end{equation*}
for $(\psi_{0}^\hbar)_{\hbar>0}$ a bounded family in $L^2(\Eng)$. We still denote by $\ell_\hbar$, $\hbar>0$, the following quantities
\begin{equation*}
    \forall \theta\in C_{c}(\R_t),\forall \sigma\in\cA_{0}(\Eng),\ \ell_{\hbar}(\theta,\sigma)= \int_{\R}\theta(t)\left(\Op_\hbar(\sigma)\psi^\hbar(t),\psi^\hbar(t)\right)_{L^2(\Eng)}dt.
\end{equation*}
Let $t\mapsto \Gamma_t d\gamma_t$ a time-averaged semiclassical measure for $(\psi^\hbar(t))_{\hbar>0}$ and we once again denote by $\ell_\infty$ denote the associated linear form.

The fact that $\psi^\hbar(t)$ is solution to \eqref{eq:SchEngel} for $\tau=2$ and the symbolic calculus now give us that for $\theta\in C_{c}^\infty(\R_t)$ and $\sigma\in\cA_0$,
\begin{equation*}
    -i\hbar^2 \ell_\hbar(\theta',\sigma) = \ell_\hbar(\theta,[\sigma,H]) + 2\hbar \ell_\hbar(\theta,(\pi(V)\cdot V)\sigma) + \hbar^2\ell_\hbar(\theta,\Delta_\Eng \sigma).
\end{equation*}

Thus for a general symbol $\sigma\in\cA_0$, we obtain once again at the limit $\hbar \rightarrow 0$, that
\begin{equation*}
    \ell_\infty(\theta,[\sigma,H]) = 0.
\end{equation*}

From this we deduce that once again, for almost all $t\in\R$, the operator $\Gamma_t$ admits the following spectral decomposition
\begin{equation*}
    \Gamma_t(x,\pi^{\delta,\beta})=\sum_{n\in\N } \Gamma_{n,t}(x,\pi^{\delta,\beta})\;\;{ with}\;\; \Gamma_{n,t}(x,\pi^{\delta,\beta}):= \Pi_n(\pi^{\delta,\beta})\Gamma_t(x,\pi^{\delta,\beta}) \Pi_n(\pi^{\delta,\beta}).
\end{equation*}

Now, using a symbol $\sigma\in\cB_0(\Eng)$, i.e commuting with $H$, and uniformly 0 in a neighborhood of $\{\delta=0\}$, we obtain
\begin{equation*}
    \ell_\hbar(\theta,(\pi(V)\cdot V)\sigma) = \cO(\hbar).
\end{equation*}
Using the decomposition in $H$-diagonal and $H$-off-diagonal parts of $(\pi(V)\cdot V)\sigma$ and assuming we have chosen $\sigma$ such that we have $\Pi_{n}\sigma \Pi_n = \sigma$ 
for some $n\in\N_{>0}$, we can write 
\begin{equation*}
    \ell_\hbar(\theta,\partial_\beta \mu_n(\delta,\beta)X_2 \sigma) = \cO(\hbar).
\end{equation*}
Taking the limit $\hbar \rightarrow 0$, we obtain that for almost all time $t\in\R_t$, we have (in a weak sense)
\begin{equation*}
    \partial_\beta \mu_n(\delta,\beta)X_2 \gamma_{n,t} = 0,
\end{equation*} 
where $\gamma_{n,t}= \Gamma_{n,t} d\gamma_t$.

As we are in a non-compact setting and the flow $X_2$ sends us to infinity, and as for $n\in\N_{>0}$, the operator-valued measure $\gamma_{n,t}$ is of finite mass,
we deduce that these measures are necessarily zero where the flows $\partial_\beta \mu_n(\delta,\beta)X_2$, $n\in\N_{>0}$, are non-zero. We have thus proved the result of Proposition \ref{prop:dispersionsupp}.

\subsection{Proof of Theorem \ref{thm:prop_wave_packet}}

This section is dedicated to the study of the semiclassical Schrödinger equation for the special case of the initial data being
a wave packet:
\begin{equation*}
    \psi_{0}^{\hbar} = WP_{x_0,\pi^{\delta_0,\beta_0}}^{\hbar}(a,\Phi_1,\Phi_2),
\end{equation*}
for $a \in \cS(\R^{2}_{e^c})$, $(x_0,\pi^{\delta_0,\beta_0}) \in \Eng \times \Enghat_{\rm gen}$ and $\Phi_1$,$\Phi_2 \in \cH_{\pi^{\delta_0,\beta_0}}^\infty$ satisfying
\begin{equation}
    H(\pi^{\delta_0,\beta_0}) \Phi_1 = \mu_{n}(\delta_0,\beta_0) \Phi_1,
\end{equation}
i.e $\Phi_1$ is an eigenvector for the $n$-th eigenvalue, $n\in\N_{>0}$. In what follows, we will simply write $\pi^0$ for the representation $\pi^{\delta_0,\beta_0}$.
Moreover, as the subLaplacian is invariant by translation for the Engel group law, we can assume that $x_0 = 0$.

The idea of the proof is to consider the ansatz
\begin{equation*}
    \varphi^{\hbar}(t,x) = \hbar^{-Q/4} e^{i\frac{S(t)}{\hbar}} \left(a^{\hbar}(t,\hbar^{-1/2}\cdot (x(t)^{-1}x))\Phi_1, \pi^0(\hbar^{-1}\cdot x)^{*}
    \Phi_2\right),
\end{equation*}
where $a^{\hbar}(t,x) \sim a(t,x){\rm Id} + \hbar^{1/2}\sigma_1 + \hbar \sigma_2 + \dots$ is valued in the space of operators on $\cH_{\pi^0}$, $S$ is real-valued and $x(t)$ is given by
\begin{equation*}
    x(0) = 0 \;\;\mbox{and}\;\;\Dot{x}(t) = \partial_{\beta}\mu_n(\delta_0,\beta_0)X_{2}(x(t)),\  t \in\R.
\end{equation*}
Plugging this ansatz into Equation \eqref{eq:Sch_engel}, we are led to solve a cascade of equations appearing at each power of $\hbar^{1/2}$. More precisely, as we want to approximate the real solution in $L^2(\Eng)$ up to an error of order $\cO(\hbar^{1/2})$, it is enough to 
construct an approximate solution to the following equation 
\begin{equation}
    \label{eq:approx_schrodinger}
    i\hbar\partial_{t}\varphi^{\hbar} + \hbar^{2}\Delta_{\Eng}\varphi^{\hbar} = O_{L^2(\Eng)}(\hbar^{3/2}).
\end{equation}
Then an energy estimate allows us to conclude the approximation is of order $O(\hbar^{1/2})$, as desired.

First, observe that for $f\in\cS(\Eng)$, we have
\begin{equation*}
    \frac{d}{dt}f(x(t)^{-1} x) = \frac{d}{dt}f\left(\Exp_{\Eng}(-t\partial_{\beta}\mu_{n}(\delta_0,\beta_0) X_2) x\right) = -\partial_{\beta}\mu_{n}(\delta_0,\beta_0)\frac{d}{ds}f\left(\Exp_{\Eng}(s X_2) x\right),
\end{equation*}
and we write $\Tilde{X}_2$ the vector field defined by the last term:
\begin{equation*}
    \frac{d}{ds}f\left(\Exp_{\Eng}(s X_2) x\right) = \Tilde{X}_2 f(x).
\end{equation*}
The vector field $\Tilde{X}_2$ corresponds to the right-invariant extension of the Lie alegbra element $X_2\in\fg_\Eng$ that we have identified up to now with its left-invariant extension, see Equation \eqref{eq:leftinv}.
It is possible to compute this vector field in semidirect coordinates: we obtain 
\begin{equation*}
    \Tilde{X}_2 = \partial_2 - x_1 \partial_3.
\end{equation*}

Our analysis mainly rely on the following two following formulae:
\begin{equation*}
    \begin{aligned}
    i\hbar\partial_t \varphi^{\hbar}(t,x) =& -\Dot{S}(t)\varphi^{\hbar}(t,x) - i\hbar^{1/2} \partial_{\beta}\mu_{n}(\delta_0,\beta_0)\left(\Tilde{X}_{2}a^{\hbar}(t,\hbar^{-1/2}\cdot (x(t)^{-1}x))\Phi_1,\pi^0(\hbar^{-1}\cdot x)^{*} \Phi_2\right) \hbar^{-Q/4}e^{i\frac{S(t)}{\hbar}}\\
    &+ i \hbar \left(\partial_t a^{\hbar}(t,\hbar^{-1/2}\cdot (x(t)^{-1}x))\Phi_1,\pi^0(\hbar^{-1}\cdot x)^{*}\Phi_2\right) \hbar^{-Q/4}e^{i\frac{S(t)}{\hbar}},
    \end{aligned}
\end{equation*}
and
\begin{equation*}
    \begin{aligned}
    \hbar^2 \Delta_{\Eng}\varphi^{\hbar}(t,x)& = -\left(a^{\hbar}(t,\hbar^{-1/2}\cdot (x(t)^{-1}x))\Phi_1,H(\pi^0) \pi^0(\hbar^{-1}\cdot x)^{*} \Phi_2\right) \hbar^{-Q/4}e^{i\frac{S(t)}{\hbar}} \\
    &+ 2\hbar^{1/2}\left((\pi^0(V)\cdot V)a^{\hbar}(t,\hbar^{-1/2}\cdot (x(t)^{-1}x))\Phi_1,\pi^0(\hbar^{-1}\cdot x)^{*}\Phi_2\right) \hbar^{-Q/4}e^{i\frac{S(t)}{\hbar}}\\
    &+ \hbar\left(\Delta_{\Eng}a^{\hbar}(t,\hbar^{-1/2}\cdot (x(t)^{-1}x))\Phi_1,\pi^0(\hbar^{-1}\cdot x)^{*}\Phi_2\right) \hbar^{-Q/4}e^{i\frac{S(t)}{\hbar}}.
    \end{aligned}
\end{equation*}

We will also make the assumption that for all $t\in\R$, $a(t,\cdot)$ is in $\cS(\R^{2}_{e^c})$, i.e in semidirect coordinates $a(t,\cdot)$ depends only of the variables $x_2$ and $x_4$. 
Naturally we will need to check at the end of our analysis that this assumption is indeed realized.

\subsubsection{Error term of order $O(1)$}
The first terms that one observes in Equation \eqref{eq:approx_schrodinger} must satisfy
\begin{equation*}
    - \Dot{S}(t)a(t,\hbar^{-1/2}\cdot (x(t)^{-1}x))\Phi_1 - a(t,\hbar^{-1/2}\cdot (x(t)^{-1}x))H(\pi^0)\Phi_1 = 0.
\end{equation*}
By the condition on $\Phi_1$, this leads to taking the phase $S$ as follows:
\begin{equation*}
    S(t) = -\mu_{n}(\delta_0,\beta_0)t.
\end{equation*}

\subsubsection{Error term of order $O(\hbar^{1/2})$} 
The second equation coming from Equation \eqref{eq:approx_schrodinger} and that one must solve is
\begin{multline*}
    \mu_{n}(\delta_0,\beta_0)\sigma_{1}(t,\hbar^{-1/2}\cdot (x(t)^{-1}x))\Phi_1
    - i \partial_{\beta}\mu_{n}(\delta_0,\beta_0)\Tilde{X}_{2}a(t,\hbar^{-1/2}\cdot (x(t)^{-1}x))\Phi_1 = \\
    H(\pi^0)\sigma_{1}(t,\hbar^{-1/2}\cdot (x(t)^{-1}x))\Phi_1
    - 2(\pi^{0}(V) \cdot V) a(t,\hbar^{-1/2}\cdot (x(t)^{-1}x))\Phi_1.
\end{multline*}

This equation can be rewritten as follows:
\begin{equation*}
    [\sigma_{1}(t,y),H(\pi^0)]\Pi_{n}(\pi^{0}) =
    \left(i\partial_{\beta}\mu_{n}\Tilde{X}_{2}a(t,y){\rm Id}
     - 2(\pi^{0}(V) \cdot V)a(t,y)\right)\Pi_{n}(\pi^{0}).
\end{equation*}

Now by assumption on $a(t,y)$, as it does not depend on the variable $x_3$, we have for all $t\in\R$
\begin{equation*}
    \Tilde{X}_2 a(t,\cdot) = X_2 a(t,\cdot).
\end{equation*}

By Lemma \ref{lem:first_diag_part}, it is thus enough to choose $\sigma_1$ to be 
\begin{equation}
    \label{def:sigma_1}
    \sigma_{1}(t,x) = \frac{i}{\delta_0} X_{1}a(t,x)\pi^{\delta_0,\beta_0}(X_{3})\Pi_{n}(\pi^0) - iX_{2}a(t,x)\partial_{\beta}\Pi_{n}(\pi^0)\Pi_{n}(\pi^0).
\end{equation}

\subsubsection{Error term of order $O(\hbar)$} 
\label{subsubsect:wpOhbar}
To obtain a solution to the approximate equation \eqref{eq:approx_schrodinger}, we need to compute one last corrector term satifisfying
\begin{multline*}
    \mu_{n}(\delta_0,\beta_0)\sigma_{2}(t,y)\Phi_1
    -i\partial_{\beta}\mu_{n}(\delta_0,\beta_0)\Tilde{X}_{2}\sigma_{1}(t,y))\Phi_1
    +i\partial_t a(t,y)\Phi_1 \\
    = H(\pi^0)\sigma_{2}(t,y)\Phi_1 
    - 2(\pi^{0}(V) \cdot V)\sigma_{1}(t,y)\Phi_1
    -\Delta_{\Eng}a(t,y)\Phi_1.
\end{multline*}

As previously, this equation can be rewritten simply as
\begin{multline}
    \label{eq:sigma_2}
    [\sigma_{2}(t,y),H(\pi^0)]\Pi_{n}(\pi^{0}) = (i\partial_{\beta}\mu_{n}\Tilde{X}_{2}\sigma_{1}(t,y) - 2(\pi^{0}(V) \cdot V)\sigma_{1}(t,y))\Pi_{n}(\pi^{0})\\ - (i\partial_{t}a + \Delta_{\Eng}a)(t,y)\Pi_{n}(\pi^{0}).
\end{multline}

Denoting $R$ the operator defined by the right-hand side term, we are to compute its diagonal part with respect to $\Pi_n(\pi^{0})$, i.e the operator $\Pi_n(\pi^{0}) R \Pi_n(\pi^{0})$. 
This will give the partial differential equation that the time-dependent $a$ need to satisfy as this operator needs to be zero.
Determining $\sigma_2$ will then follow easily as one would write 
\begin{equation*}
    [\sigma_{2},H(\pi^0)]\Pi_{n}(\pi^{0}) = (\mu_{n}(\delta_0,\beta_0){\rm Id}-H(\pi^0))\Pi_{n}^{\perp}(\pi^0) \sigma_{2} = \Pi_{n}^{\perp}(\pi^0) R \Pi_{n}(\pi^0),
\end{equation*}
where $\Pi_{n}^{\perp}(\pi^0) = {\rm Id}-\Pi_{n}(\pi^0)$. As $\mu_{n}(\delta_0,\beta_0){\rm Id}-H(\pi^0)$ is by definition invertible on ${\rm Ran}(\Pi_{n}^{\perp}(\pi^0))$, 
we can rewrite the last equation simply as
\begin{equation*}
    \sigma_{2} = \left(\mu_{n}(\delta_0,\beta_0){\rm Id}-H(\pi^0)\right)^{-1}\Pi_{n}^{\perp}(\pi^0) R \Pi_{n}(\pi^0).
\end{equation*}

We now proceed to determine the diagonal part term by term.

\vspace{0.5cm}
{\bf Term $i\partial_{\beta}\mu_{n}(\delta_0,\beta_0)\Tilde{X}_{2}\sigma_{1}$:} By Equation \eqref{def:sigma_1} we have
\begin{equation*}
    \Tilde{X}_{2}\sigma_{1}(t,y) = \frac{i}{\delta_0} \Tilde{X}_{2}X_{1}a(t,y)\pi^{0}(X_{3}) - i\Tilde{X}_2 X_{2}a(t,y)\partial_{\beta}\Pi_{n}(\pi^{0})\Pi_{n}(\pi^{0}).
\end{equation*}
We have already observed that the operator $\partial_{\beta}\Pi_{n}$ has no diagonal part.
However, it is not the case for $\pi^{0}(X_{3})$. We obtain
\begin{equation*}
    i\partial_{\beta}\mu_{n}(\delta_0,\beta_0)\Pi_{n}(\pi^{0})\Tilde{X}_{2}\sigma_{1}(t,y)\Pi_{n}(\pi^{0}) = -\frac{1}{\delta_0}\partial_{\beta}\mu_{n}(\delta_0,\beta)\Tilde{X}_{2}X_{1}a(t,y)\Pi_{n}(\pi^{0})\pi^{0}(X_{3})\Pi_{n}(\pi^{0}).
\end{equation*}

\vspace{0.5cm}
{\bf Term $2\pi^{0}(X_2)X_{2}\sigma_1$:} By Equation \ref{def:sigma_1}, we have
\begin{equation*}
    2X_{2} \pi^{0}(X_2)\sigma_1(t,y) = \frac{2i}{\delta_0}X_{2}X_{1}a(t,y) \pi^{0}(X_{2}X_{3}) - 2iX_{2}^{2}a(t,y)\pi^{0}(X_2)\partial_{\beta}\Pi_{n}(\pi^{0})\Pi_{n}(\pi^{0}).
\end{equation*}

The following lemma allows us to compute the diagonal part of this operator.

\begin{lemma}
\label{lem:diagpartsigma1_1}
    We have
    \begin{equation}
        \pi^{0}(X_{2}X_{3}) = \left[-\frac{1}{2}\pi^{0}(X_1),H(\pi^0)\right],
    \end{equation}
    and
    \begin{equation}
        \Pi_{n}(\pi^0)\pi^{0}(X_2)\partial_{\beta}\Pi_{n}(\pi^0)\Pi_{n}(\pi^0) = \frac{i}{2}\left(\frac{1}{2}\partial_{\beta}^{2}\mu_{n}(\delta_0,\beta_0) - 1\right)\Pi_{n}(\pi^0).
    \end{equation}
\end{lemma}

\begin{proof}
    The first equation results directly from the following computation in the enveloping algebra of $\fg_\Eng$:
    \begin{equation*}
        X_{2}X_{3} = \left[-\frac{1}{2}X_{1}, -X_{1}^2 - X_{2}^2\right].
    \end{equation*}
    For the second equation, recall that by Lemma \ref{lem:FH1} we have the following Feynmann-Hellmann formula:
    \begin{equation*}
        H(\pi^{\delta,\beta})\partial_{\beta}\Pi_{n}(\pi^{\delta,\beta}) - \partial_{\beta}\mu_{n}(\delta,\beta)\Pi_{n}(\pi^{\delta,\beta}) = \mu_{n}(\delta,\beta)\partial_{\beta}\Pi_{n}(\pi^{\delta,\beta}) + 2i\pi^{\delta,\beta}(X_2)\Pi_{n}(\pi^{\delta,\beta}),
    \end{equation*}
    holding for all $(\delta,\beta)\in\fg_{3}^*\setminus\{0\}\times\fr_{2}^*$. Deriving with respect to $\beta$, we obtain
    \begin{equation*}
        H \partial_{\beta}^{2}\Pi_{n} - 4i\pi^{\delta,\beta}(X_2)\partial_{\beta}\Pi_{n} + 2\Pi_{n} = \partial_{\beta}^{2}\mu_{n}\Pi_{n} + 2\partial_{\beta}\mu_{n}\partial_{\beta}\Pi_{n} + \mu_{n}\partial_{\beta}^{2}\Pi_{n}.
    \end{equation*}
    By evalutating at $\pi^0$ and multiplying on each side by $\Pi_{n}(\pi^0)$, we have
    \begin{equation*}
        \Pi_{n}(\pi^0)\pi^{0}(X_2)\partial_{\beta}\Pi_{n}(\pi^0)\Pi_{n}(\pi^0) = \frac{i}{2}\left(\frac{1}{2}\partial_{\beta}^{2}\mu_{n}(\delta_0,\beta_0) - 1\right)\Pi_{n}(\pi^0).
    \end{equation*}
\end{proof}

The diagonal part of $2X_{2} \pi^{0}(X_2)\sigma_1(t,y)$ is then given by 
\begin{equation*}
    \left(\frac{1}{2}\partial_{\beta}^{2}\mu_{n}(\delta_0,\beta_0) - 1\right)X_{2}^{2}a(t,y)\Pi_{n}(\pi^0).
\end{equation*}

\vspace{0.5cm}
{\bf Term $2\pi^{0}(X_1) X_{1} \sigma_1$:} By definition we have
\begin{equation*}
    2 \pi^{0}(X_1) X_{1}\sigma_1(t,y) = \frac{2i}{\delta_0} X_{1}^{2}a(t,y)\pi^{0}(X_{1}X_{3}) - 2iX_{1}X_{2}a(t,y)\pi^{0}(X_1)\partial_{\beta}\Pi_{n}(\pi^{0})\Pi_{n}(\pi^{0}).
\end{equation*}

\begin{lemma}
    \label{lem:diagpartsigma1_2} 
    We have
    \begin{equation*}
        \pi^{0}(X_{1}X_{3}) = \frac{i\delta_{0}}{2}{\rm Id} + \left[\frac{1}{4i\delta_0}\pi^{0}(X_{3}^2), H(\pi^0)\right],
    \end{equation*}
    and
    \begin{equation*}
        \Pi_{n}(\pi^{0})\pi^{0}(X_1)\partial_{\beta}\Pi_{n}(\pi^{0})\Pi_{n}(\pi^{0}) = \frac{1}{2i\delta_{0}}\partial_{\beta}\mu_{n}(\delta_0,\beta_0)\Pi_{n}(\pi^{0})\pi^{0}(X_{3})\Pi_{n}(\pi^{0}).
    \end{equation*}
\end{lemma}

\begin{proof}
    The first equation follows the next computation in the enveloping algebra of $\fg_\Eng$:
    \begin{equation*}
        \left[X_{3}^2,-X_{1}^2- X_{2}^2\right] = 4X_{1}X_{3}X_4 - 2X_{4}^2.
    \end{equation*}
    The second equation is less obvious and need to use the expression of $\pi^{0}(X_1)$ found in Lemma \ref{lem:first_diag_part}:
    \begin{equation*}
        \begin{aligned}
        \Pi_{n}\pi^{0}(X_1)\partial_{\beta}\Pi_{n}\Pi_{n} &=
        \frac{1}{2i\delta_0}\left(\Pi_{n}\pi^{0}(X_{3})H(\pi^{0})\partial_{\beta}\Pi_{n}\Pi_{n}-\mu_{n}(\delta_0,\beta_0)\Pi_{n}\pi^{0}(X_{3})\partial_{\beta}\Pi_{n}\Pi_{n}\right)\\
        &= \frac{1}{2i\delta_0}\Pi_{n}\pi^{0}(X_{3})(H(\pi^{0})-\mu_{n}(\delta_0,\beta_0){\rm Id})\partial_{\beta}\Pi_{n}\Pi_{n}.
        \end{aligned}
        \end{equation*}
    By Feynman-Hellman formula, this gives
    \begin{equation*}
         \Pi_{n}\pi^{0}(X_1)\partial_{\beta}\Pi_{n}\Pi_{n} =
         \frac{1}{2i\delta_0}\Pi_{n}\pi^{0}(X_{3})(2i\pi^{0}(X_2) + \partial_{\beta}\mu_{n}(\delta_0,\beta_0){\rm Id})\Pi_{n}.
    \end{equation*}
    By Lemma \ref{lem:diagpartsigma1_1} we can write $\pi^{\delta_0,\beta_0}(X_{3}X_2)$ as a commutator with $H(\delta_0,\beta_0)$ and thus it has no diagonal part. The remaining term
    gives the second equation of the lemma.
\end{proof}

The diagonal part of $2X_{1} \pi^{0}(X_1)\sigma_1(t,y)$ is then given by
\begin{equation*}
    -X_{1}^{2}a(t,y)\Pi_{n}(\pi^0) - \frac{1}{\delta_{0}}\partial_{\beta}\mu_{n}(\delta_0,\beta_0)X_{1}X_{2}a(t,y)\Pi_{n}(\pi^{0})\pi^{0}(X_{3})\Pi_{n}(\pi^{0}).
\end{equation*}

\vspace{0.5cm}
{\bf Summing all terms:} The sum of the diagonal parts of each term in the righ-hand side of Equation \eqref{eq:sigma_2} is 
\begin{equation*}
    \begin{aligned}
    &-\frac{1}{\delta_0} \partial_{\beta}\mu_{n}\Tilde{X}_{2}X_{1}a(t,y)\Pi_{n}\pi^{\delta_0,\beta_0}(X_{3})\Pi_{n} \\
    &+X_{1}^{2}a(t,y)\Pi_{n} + \frac{1}{\delta_0}\partial_{\beta}\mu_{n}X_{1}X_{2}a(t,y)\Pi_{n}\pi^{0}(X_{3})\Pi_{n}
    - \left(\frac{1}{2}\partial_{\beta}^{2}\mu_{n}-1\right)X_{2}^{2}a(t,y)\Pi_{n}\\
    &- (i\partial_{t}a + \Delta_{\Eng}a)(t,y)\Pi_{n}.
    \end{aligned}
\end{equation*}
As this sums need to be equal to zero, after regrouping some terms together we obtain
\begin{equation*}
    \left(i\partial_{t}a + \frac{1}{2}\partial_{\beta}^{2}\mu_{n}(\delta_0,\beta_0)X_{2}^{2}a\right)\Pi_n  = \frac{1}{\delta_0}\partial_{\beta}\mu_{n}\left(X_{1}X_{2}-\Tilde{X}_2 X_1\right)a \Pi_{n}\pi^{0}(X_{3})\Pi_{n}.
\end{equation*}
As we have $\Tilde{X}_2 = X_2 - x_1 \partial_3$, we compute in semidirect coordinates that
\begin{equation*}
    \begin{aligned}
        X_{1}X_{2}-\Tilde{X}_2 X_1 &= [X_1,X_2] -x_1 \partial_3 X_1 \\
        &= X_3 - x_1 X_1 \partial_3  + x_1[X_1, \partial_3] \\
        &= X_3 - x_1 X_1 \partial_3 -\frac{1}{2}x_1 \partial_4\\
        &= \partial_{3}- x_1 X_1 \partial_3.
    \end{aligned}
\end{equation*}
Since we made the assumption that at any time $t\in\R$, we have $\partial_3 a(t,y) = 0$, we deduce that the time-dependent profile $a$ must satisfy the following partial differential equation:
\begin{equation}
    \label{eq:profile_a}
    i\partial_{t}a + \frac{1}{2}\partial_{\beta}^{2}\mu_{n}(\delta_0,\beta_0)X_{2}^{2}a = 0.
\end{equation}

We easily check that, since $a_0$ is in $\cS(\R^{2}_{e^c})$, so does the solution of Equation \eqref{eq:profile_a} at any time $t\in\R$.
The analysis is thus complete and we have constructed a solution to the approximate semiclassical Schrödinger equation \eqref{eq:approx_schrodinger}.


\section{Second microlocalization on conic subsets}
\label{sect:secondmicro}

This section is devoted to developping the necessary tools to investigate in more detail the properties of the time-averaged semiclassical measures associated 
with a solution to the Schrödinger equation \eqref{eq:SchEngel} for time-scale $\tau=2$. As Proposition \ref{prop:dispersionsupp} has shown, these measures 
split into the measures $(\gamma_{n,t})_{n\in\N_{>0}}$ along the spectral projectors of the symbol $H$, and we have the following support condition 

\begin{equation}
    \label{eq:suppgamma}
    {\rm supp}(\gamma_{n,t}) \subset \Eng\times\{\pi^{\delta,\beta}\,:\,\partial_{\beta}\mu_{n}(\delta,\beta)=0, (\delta,\beta)\in\fg_{3}^*\setminus\{0\}\times\fr_{2}^*\}.
\end{equation}

Using the relation between the eigenvalues of the symbol $H$ and the Montgomery operators, we can write 
\begin{equation*}
    \partial_{\beta}\mu_{n}(\delta,\beta) = \delta^{1/3}\partial_{\nu}\Tilde{\mu}_{n}(\beta\delta^{-1/3}).
\end{equation*}
Thus any point in ${\rm supp}(\gamma_{n,t})$ is contained in a set of the form
\begin{equation*}
    C_{\nu_0} = \Eng \times \{\pi^{\delta,\beta}\in \Enghat_{\rm gen}\,:\,\beta = \nu_0 \delta^{1/3}\},
\end{equation*}
for any critical point $\nu_0$ of $\Tilde{\mu}_{n}$. Such sets can be seen to be \textit{cones} for the dilation we introduced 
on $\Enghat$.

In what follows, we fix $\nu_0\in\R$ and develop a second-microlocalization on the cone $C_{\nu_0}$.

\subsection{Extended phase space and second-microlocal semiclassical calculus}
\label{subsect:2microconicsubsets}

\subsubsection{Second-microlocal observables}

To achieve such an analysis, we need to extend our phase space $\Eng\times\Enghat$ by introducing a new variable $\eta\in\R$. 
The test symbols associated to this extended phase space are defined as follows:
\begin{definition}
\label{def:sym2micro}
    The space of second-microlocal symbols $\mathcal{A}_{0}^{2}(\Eng)$ is defined as the set of field of operators $\sigma = \{\sigma(x,\pi,\eta) \in \cL(\cH_\pi)\,:\,(x,\pi,\eta) \in \Eng \times \Enghat \times \R\}$
    for which there exists a application $\kappa: (x,y;\eta) \in \Eng\times\Eng\times\R \mapsto \kappa_{x}(y;\eta)$ in $C^{\infty}(\R_\eta,C_{c}^{\infty}(\Eng_x,\cS(\Eng_y)))$, called also the {\it associated convolution kernel}, such that
    \begin{itemize}
        \item For all $(x,\pi,\eta) \in \Eng\times\Enghat\times\R$ we have, $\sigma(x,\pi,\eta) = \cF_{\Eng}(\kappa_{x}(\cdot\,;\eta))(\pi)$.
        \item There exists a compact $K\subset \Eng$ such that for all $\eta \in \R$, the map $x\in\Eng \mapsto \kappa_{x}(\cdot\,;\eta) \in \cS(\Eng)$ is 
        a smooth function compactly supported in $K$.
        \item There exist two symbols $\sigma_{\pm 1} \in \cA_{0}(\Eng)$ and $R_0 > 0$ such that, if $|\eta| > R_0$, then 
        $\sigma(x,\pi,\eta) = \sigma_{\eta/|\eta|}(x,\pi)$, for $(x,\pi,\eta)\in\Eng\times\Enghat\times\R$. 
        \item We ask that $\sigma$ vanishes uniformly on a neighborhood of $\{\delta = 0\}$, meaning there exists $\delta_0 > 0$ such that
        \begin{equation}
            \label{eq:supportdelta0}
            \forall (x,\pi^{\delta,\beta},\eta)\in\Eng\times\Enghat_{\rm gen}\times\R,\ |\delta|<\delta_0 \implies \sigma(x,\pi^{\delta,\beta},\eta) = 0.
        \end{equation}
    \end{itemize} 
\end{definition}

\begin{remark}
    The third point of the previous definition can be reformulated this way: the second-microlocal symbol $\sigma$ is asked to be homogeneous of order $0$ for $\eta$ large enough uniformly. The two observables 
    $\sigma_{\pm 1}$ can be seen as a unique symbol $\sigma_{\infty}$ defined on $\Eng\times\Enghat\times\bS^{0}$ such that, for all $(x,\pi)\in\Eng\times\Enghat$, we have
    \begin{equation*}
        \forall \eta \in \R,\ |\eta| \geq R_0 \implies \sigma(x,\pi,\eta) = \sigma_{\infty}\left(x,\pi,\frac{\eta}{|\eta|}\right).
    \end{equation*}
    We will denote by $\cA_{0}^{2,\infty}(\Eng)$ the space of these symbols $\sigma_{\infty}$.
\end{remark}

Remark that for $\eta\in\R$ fixed, the symbol given by $\sigma(\cdot,\cdot,\eta)$ is in the space of smoothing symbols $\cA_0(\Eng)$.

\subsubsection{Quantization procedure}
The addition of the extra variable $\eta$ becomes clear once we define their quantization: for $\hbar >0$ and a second-microlocal symbol $\sigma\in\mathcal{A}_{0}^{2}(\Eng)$, we associate the symbol $\sigma_{\hbar}^{C_{\nu_0}}$ defined by
\begin{equation}
    \label{eq:symCnu}
    \sigma_{\hbar}^{C_{\nu_0}}(x,\pi^{\delta,\beta}) = \sigma\left(x,\pi^{\delta,\beta},\frac{\beta-\nu_0\delta^{1/3}}{\hbar}\right),
\end{equation}
and we define the quantization of $\sigma$ by
\begin{equation}
    \label{eq:quantization2scale}
    \Op_{\hbar}^{C_{\nu_0}}(\sigma) := \Op_{\hbar}(\sigma_{\hbar}^{C_{\nu_0}}).
\end{equation}

For this definition to make sense, we must make sure that, for $\hbar > 0$ fixed, $\sigma_{\hbar}^{C_{\nu_0}}$ is a symbol in a certain class $S^{m}_{\rho,\delta}(\Eng)$ (see Section \ref{subsubsect:generalclasssym}).

Note first that if \eqref{eq:quantization2scale} makes sense, we could rewrite the quantization of a second-microlocal symbol $\sigma \in \cA_{0}^{2}(\Eng)$ as follows:
\begin{equation*}
    \Op_{\hbar}^{C_{\nu_0}}(\sigma) = \Op_{1}\left(\sigma\left(x,\hbar\cdot\pi^{\delta,\beta},\beta-\nu\delta^{1/3}\right)\right),
\end{equation*}
so, up to a dilation, it is equivalent to work with the symbol in the right-hand side. 

The next proposition allows us to extend the formulae of the difference operators found in Proposition \ref{prop:diffopformula} for symbols in $\cS(\Enghat)$
and to use them to directly test if a given measurable field of operators on $\Enghat$ is in one of our classes of regular symbols. For the sake of clarity, in the following proposition 
and its proof we will denote by $\hat{\Delta}_i$, $i\in\{1,2,3,4\}$, the operators given by the right-hand side formulae of the corresponding difference operators
presented in Proposition \ref{prop:diffopformula}.

\begin{proposition}
    \label{prop:characterisationS0}
    Let $\sigma$ be a smooth application in $C^{\infty}\left(\R\setminus\{0\}\times\R,\cL(\cS(\R_\xi),\cS'(\R_\xi))\right)$ for which there exists $\delta_0>0$ sich that 
    \begin{equation*}
        \forall (\delta,\beta)\in\R\subset\{0\}\times\R,\ |\delta|<\delta_0 \implies \sigma(\delta,\beta) = 0,
    \end{equation*}
    and satisfying the following estimates: for $\alpha\in\N^4$, 
    \begin{equation}
        \label{eq:estimateS0}
        \sup_{(\delta,\beta)\in\R\setminus\{0\}\times\R}\|\hat{\Delta}^{\alpha}\sigma(\delta,\beta)\|_{\cL(L^2(\R_\xi))} < \infty.
    \end{equation}
    Then there exists a symbol $\tau$ in the class $S^{0}_{0}(\Eng)$ such that
    \begin{equation*}
        \forall \alpha\in\N^4,\,\forall \pi^{\delta,\beta}\in\Enghat_{\rm gen},\ \Delta^{\alpha}\tau(\pi^{\delta,\beta}) = \hat{\Delta}^\alpha \sigma(\delta,\beta).
    \end{equation*}
\end{proposition}

\begin{remark}
    Recall that a symbol $\tau = \left\{\tau(\pi)\,:\,\pi\in\Enghat\right\}$ is in $S_{0}^{0}(\Eng)$ if and only if,
    for all $\alpha\in\N^{4}$, the field of operators 
    \begin{equation*}
        \left\{\Delta^{\alpha}\tau(\pi)\,:\,\pi\in\Enghat\right\},
    \end{equation*}
    makes sense and is in $L^{\infty}(\Enghat)$, i.e
    \begin{equation*}
        \|\Delta^{\alpha}\tau\|_{L^{\infty}(\Enghat)} < \infty.
    \end{equation*}
\end{remark}

\begin{proof}
    By the estimate \eqref{eq:estimateS0} for $\alpha=0$, the measurable symbol $\tau$ defined by
    \begin{equation*}
        \forall \pi^{\delta,\beta}\in\Enghat_{\rm gen},\ \tau(\pi^{\delta,\beta}) = \sigma(\delta,\beta),
    \end{equation*}
    is in $L^\infty(\Enghat)$. By the discussion in the Section \ref{subsubsec_convop}, there exists a distribution $\kappa\in \cS'(\Eng)$
    such that 
    \begin{equation*}
        \tau = \cF_\Eng \kappa.
    \end{equation*}
    We wish to prove that $\kappa$ is in the domain of the difference operators $\Delta^\alpha$ in the sense of Definition \ref{def:diffop} and
    that their action corresponds to the one of the operators $\hat{\Delta}^\alpha$ on $\sigma$.
    We introduce $\phi\in\cS(\Eng)$ a Schwartz function satisfying $\phi(0)=1$, and for $\eps>0$, the following approximation of identity
    \begin{equation}
        \label{eq:dilatphi}
        \forall x\in\Eng,\ \phi_{\eps}(x) = \eps^{-Q}\phi(\eps^{-1}\cdot x).
    \end{equation}
    As $\kappa$ is a tempered distribution, we can consider the following sequence of smooth functions
    \begin{equation*}
        \kappa_\eps = \kappa * \phi_\eps \in C_{b}^\infty(\Eng),
    \end{equation*}
    satisfying $\kappa_\eps \rightarrow \kappa$ in $\cS'(\Eng)$ as $\eps$ goes to 0. Moreover, the Fourier transform of $\kappa_\eps$, that we
    denote by $\tau_\eps$, makes sense and satisfies 
    \begin{equation*}
        \tau_\eps = \hat{\phi}_\eps \tau.
    \end{equation*}
    We can actually prove that $\kappa_\eps$ is a Schwartz function. First, since $\phi_\eps$ is a Schwartz function, the inverse Fourier formula makes sense and we can write for $x\in\Eng$,
    \begin{equation*}
        \kappa_\eps(x) = \int_{\Enghat_{\rm gen}} {\rm Tr}\left(\pi^{\delta,\beta}(x)\tau_\eps(\pi^{\delta,\beta})\right)d\mu_{\Enghat}(\pi^{\delta,\beta}).
    \end{equation*}
    For $i\in\{1,2,3,4\}$, observe by direct computation that we have for $x\in\Eng$
    \begin{equation*}
        \hat{\Delta}_i \pi^{\delta,\beta}(x) = x_i \pi^{\delta,\beta}(x),
    \end{equation*}
    where the operator $\hat{\Delta}_i$ is applied to the field $\{\pi^{\delta,\beta}(x)\,:\,(\delta,\beta)\in\fg_{3}^*\setminus\{0\}\times\fr_{2}^*\}$ and $(x_1,x_2,x_3,x_4)$ are the semidirect
    coordinates of $x$.
    This allows us to write
    \begin{equation*}
        \begin{aligned}
            x_i \kappa_\eps(x) &= \int_{\Enghat_{\rm gen}} {\rm Tr}\left(x_i \pi^{\delta,\beta}(x)\tau_\eps(\pi^{\delta,\beta})\right)d\mu_{\Enghat}(\pi^{\delta,\beta}) \\
            &= \int_{\Enghat_{\rm gen}} {\rm Tr}\left(\hat{\Delta}_i \pi^{\delta,\beta}(x) \hat{\phi}_\eps(\pi^{\delta,\beta})\sigma(\delta,\beta)\right)d\mu_{\Enghat}(\pi^{\delta,\beta}) \\
            &= -\int_{\Enghat_{\rm gen}} {\rm Tr}\left(\pi^{\delta,\beta}(x) \hat{\Delta}_i(\hat{\phi}_\eps(\pi^{\delta,\beta})\sigma(\delta,\beta))\right)d\mu_{\Enghat}(\pi^{\delta,\beta}),
        \end{aligned}
    \end{equation*}
    the last equation following from the explicit formula given by Proposition \ref{prop:diffopformula} and by integration by parts, the remainder terms being zero by assumption on 
    the support of $\sigma$ and smoothness of $\kappa_\eps$. Now we can write 
    \begin{equation*}
        \hat{\Delta}_i(\hat{\phi}_\eps \sigma) = \hat{\Delta}_i \hat{\phi}_\eps \circ \sigma + \hat{\phi}_\eps \circ \hat{\Delta}_i \sigma,
    \end{equation*}
    and by estimates \eqref{eq:estimateS0}, since $\phi_\eps$ is in $\cS(\Eng)$, we conclude that $x_i \kappa_\eps$ is uniformly bounded. By recurrence, our claim is proved.
    
    Now, as $\kappa_\eps$ is a Schwartz function, we can write for $i\in\{1,2,3,4\}$  
    \begin{equation}
        \label{eq:difftaueps}
        \Delta_i \tau_\eps = \hat{\Delta}_i \tau_\eps = \hat{\Delta}_i \hat{\phi}_\eps\circ \tau + \hat{\phi}_\eps\circ\hat{\Delta}_i \sigma.
    \end{equation}
    As $\hat{\Delta}_i \hat{\phi}_\eps=\Delta_i \hat{\phi}_\eps$, we check that for $\pi\in\Enghat$:
    \begin{equation*}
        \Delta_i \hat{\phi}_\eps(\pi) = \int_\Eng -x_i \phi_{\eps}(x)\pi(x)^* dx
        = \eps^{\upsilon_i}\int_\Eng \psi^{i}_\eps(x)\pi(x)^* dx = \eps^{\upsilon_i} \hat{\psi}^{i}_\eps(\pi),
    \end{equation*}
    where we introduced $\psi^i = -x_i \phi$ and its associated dilated family $(\psi^{i}_\eps)_{\eps>0}$ as in \eqref{eq:dilatphi}.
    Similarly as in the discussion of estimate \eqref{eq:estimateS0} for $\alpha=0$, the measurable symbol $\tau_i$ defined by 
    \begin{equation*}
        \forall \pi^{\delta,\beta}\in\Enghat_{\rm gen},\ \tau_i(\pi^{\delta,\beta}) = \hat{\Delta}_i \sigma(\delta,\beta),
    \end{equation*}
    is in $L^\infty(\Enghat)$. Equation \eqref{eq:difftaueps} becomes 
    \begin{equation*}
        \Delta_i \tau_\eps = \hat{\phi}_\eps\circ \tau_i + \cO_{L^\infty(\Enghat)}(\eps^{\upsilon_i}).
    \end{equation*}
    We deduce that since $(\phi_\eps)_{\eps>0}$ is an approximation of the identity, $\Delta_i \tau_\eps$ converges towards $\tau_i$ in $L^\infty(\Enghat)$ for the strong operator topology as $\eps$ goes to 0,
    meaning seen as operators on $L^2(\Eng)$.
    By the various realization of the group von Neumann algebra discussed in Section \ref{subsubsec_convop}, it implies that $-x_i \kappa_\eps$ converges in the space of tempered distributions towards 
    $\cF_{\Eng}^{-1}\tau_i\in\cK(\Eng)$. Observing that $-x_i \kappa_\eps$ also converges towards $-x_i \kappa$ in $\cS'(\Eng)$, by unicity of the limit 
    we have $-x_i \kappa \in \cK(\Eng)$. It follows that $\Delta_i \tau$ makes sense and satisfies 
    \begin{equation*}
        \forall \pi^{\delta,\beta}\in\Enghat_{\rm gen},\ \Delta_i \tau(\pi^{\delta,\beta}) = \hat{\Delta}_i \sigma(\delta,\beta).
    \end{equation*}
    For difference operators of higher order, we conclude by recurrence.
\end{proof}

The following proposition makes sense of our quantization procedure for second-microlocal symbols as the measurable symbols $\sigma_{\hbar}^{C_{\nu_0}}$, $\hbar >0$, for $\sigma \in \cA_{0}^2(\Eng)$ introduced by 
Equation \eqref{eq:symCnu} are shown to be proper regular symbols.

\begin{proposition}
    \label{prop:2microbounded}
    For $\sigma \in \cA_{0}^2$, the family 
    \begin{equation*}
        \left(\left\{\Tilde{\sigma}_{\hbar}(x,\pi^{\delta,\beta}) = \sigma(x,\hbar\cdot\pi^{\delta,\beta},\beta-\nu_0\delta^{1/3})\,:\, (x,\pi^{\delta,\beta})\in\Eng\times\Enghat_{\rm gen}\right\}\right)_{\hbar>0},
    \end{equation*}
    is a bounded family of symbols in $S_{0,0}^{0}(\Eng\times\Enghat)$. 
    As a consequence, by Calder\`on-Vaillancourt's theorem \ref{thm:CalderonVaillancourt}, the family $\left(\Op_{\hbar}^{C_{\nu_0}}\left(\sigma\right)\right)_{\hbar>0}$ is a bounded family of operators on $L^2(\Eng)$ satisfying
    \begin{equation*}
        \forall \hbar \in (0,1),\ \left\|\Op_{\hbar}^{C_{\nu_0}}\left(\sigma\right)\right\|_{L^2\rightarrow L^2} \leq C \sum_{0\leq i\leq d}\sup_{\eta\in\R} N(\partial_{\eta}^{i}\sigma(\cdot,\cdot,\eta)),
    \end{equation*}
    for some fixed $d\in\N_{>0}$. We will denote by $N_2(\sigma)$ the seminorm defined by the right-hand side of the last inequality.
\end{proposition}

\begin{proof}
    By Proposition \ref{prop:characterisationS0}, to prove that we have defined symbols in the class $S^{0}_{0,0}$, it is enough to bound uniformly in $\hbar$ the seminorms $\|\sigma\|_{S_{0,0}^{0},\alpha,\beta}$
    using the explicit formulae of the difference operators (see Proposition \ref{prop:diffopformula}).
    We remark that all difference operators acts on an element of $\cS(\Enghat)$ in part by commutator and in part by differentiating along the variables $\delta$ and $\beta$. 
    
    For the commutator parts, we treat the case of $\Delta_1$: we have
    \begin{equation*}
        \Delta_1 \Tilde{\sigma}(x,\pi^{\delta,\beta}) = \frac{i}{\delta} \left[\pi^{\delta,\beta}(X_3),\sigma(x,\hbar\cdot\pi^{\delta,\beta},\beta-\nu\delta^{1/3})\right],  
    \end{equation*}
    thus, we have the following estimate:
    \begin{equation*}
        \begin{aligned}
        \sup_{x\in\Eng} \|\Delta_1 \Tilde{\sigma}_{\hbar}(x,\cdot)\|_{L^{\infty}(\Enghat)}
        &\leq \hbar \,\sup_{\eta \in \R} \|\Delta_1 \sigma(\cdot,\cdot,\eta)\|_{L^{\infty}(\Eng\times\Enghat)}.
        \end{aligned} 
    \end{equation*}
    For the difference operators acting by differentiation, we treat the case of $\Delta_2$. We observe that 
    \begin{equation*}
        \Delta_2 \Tilde{\sigma}_{\hbar}(x,\pi^{\delta,\beta}) = \hbar \Delta_2 \sigma(x,\hbar \cdot \pi^{\delta,\beta},\beta-\nu\delta^{1/3}) + \partial_{\eta}\sigma(x,\hbar \cdot \pi^{\delta,\beta},\beta-\nu\delta^{1/3}),
    \end{equation*}
    where in the right-hand side, the difference operator $\Delta_2$ is seen acting on the symbol $\sigma(\cdot,\cdot,\beta-\nu_0 \delta^{1/3})$. We deduce 
    \begin{equation*}
        \sup_{x\in\Eng} \|\Delta_2 \Tilde{\sigma}_{\hbar}(x,\cdot)\|_{L^{\infty}(\Enghat)}
        \leq \hbar \,\sup_{\eta \in \R} \|\Delta_2 \sigma(\cdot,\cdot,\eta)\|_{L^{\infty}(\Eng\times\Enghat)}
        + \sup_{\eta \in \R} \|\partial_{\eta}\sigma(\cdot,\cdot,\eta)\|_{L^{\infty}(\Eng\times\Enghat)}.
    \end{equation*}
    We treat similarly the difference operators $\Delta_3$ and $\Delta_4$, and more generally $\Delta^{\alpha}$ to get a bound
    of each seminorm by a uniform bound on $\eta$ of a seminorm in $S_{0,0}^{0}$ of $\partial_{\eta}^{i}\sigma(\cdot,\cdot,\eta)$, $i\leq |\alpha|$. 
    
    The Calder\`on-Vaillancourt theorem gives that the norm of $\Op_{1,\Eng}\left(\Tilde{\sigma}_{\hbar}\right)$ as operators on $L^2(\Eng)$ is bounded by the seminorm $N(\cdot) = \|\cdot\|_{S_{0,0,a,b}^{0}}$, and thus 
    we find there exists $d\in\N_{>0}$ such that
    \begin{equation*}
        \left\|\Op_{1,\Eng}\left(\Tilde{\sigma}_{\hbar}\right)\right\|_{L^2\rightarrow L^2} \leq C \sum_{0\leq i\leq d}\sup_{\eta\in\R} N(\partial_{\eta}^{i}\sigma(\cdot,\cdot,\eta)).
    \end{equation*}
\end{proof}

Thanks to Proposition \ref{prop:2microbounded}, it now makes sense to study the limiting objects of the linear forms
\begin{equation*}
    \sigma\in\cA_{0}^2(\Eng)\mapsto \left(\Op_{\hbar}^{C_{\nu_0}}(\sigma)\psi_{0}^\hbar,\psi_{0}^\hbar\right)_{L^2(\Eng)},
\end{equation*}
for a bounded family $(\psi_{0}^\hbar)_{\hbar>0}$ in $L^2(\Eng)$, as $\hbar$ goes to 0.

\subsection{Fourier integral operators}

This section, as well as the following one, presents technical but necessary results for the identification of the limiting objects 
given in Section \ref{subsect:secondmicromeas}. It is also here that our analysis establishes the link with the pseudodifferential calculus 
of the group $\Heis^{1,1}$.

\subsubsection{Quasi-contact manifolds and the group $\Heis^{1,1}$}
To continue the study of the operators obtained by the precedent quantization of the second-microlocal symbols, we further 
discuss the relation between Engel manifolds and quasi-contact ones.

As it has been pointed out in the introduction, every Engel manifold $(M,\cD)$ induces a quasi-contact distribution by considering 
$\cE=\cD + [\cD,\cD]$: this is indeed a non-integrable distribution of dimension 3 on a 4-manifold. And as we have already seen,
the characteristic field $L$ of an Engel manifold is actually a feature of the quasi-contact structure $(M,\cE)$ and is defined by a 
condition of stabilization of the distribution $\cE$:
\begin{equation*}
    [L,\cE] \equiv 0, \mod \cE.
\end{equation*}
Similarly to Engel manifolds, quasi-contact ones also have a Darboux-like theorem and are all locally modelled on a particular nilotent Lie group,
that we call in this article the \textit{quasi-Heisenberg group} and we denote by $\Heis^{1,1}$. Its Lie algebra $\fg_{\Heis^{1,1}}$ is spanned by four vector fields
$Y_1,Y_2,Y_3$ and $Y_4$ satisfying the relations 
\begin{equation*}
    [Y_1,Y_3] = Y_4.
\end{equation*}
It is thus trivially isomorphic to $\Heis\times\R_2$ where $\Heis = \Exp_{\Heis^{1,1}}(\R Y_1 \oplus\R Y_3\oplus\R Y_4)$ and $\R_2 = \Exp_{\Heis^{1,1}}(\R Y_2)$. 

These notations are voluntarily similar to the ones introduced in Section \ref{subsect:semidirect} for the semidirect decomposition of $\Eng$
for the following reason. The Engel group $\Eng$ is trivially an Engel manifold for the distribution $\cD_\Eng$ spanned by the vector fields $X_1$ and $X_2$. The
quasi-contact distribution $\cE_\Eng$ it induces is then given by the span of $X_1,X_2$ and $X_3$. It is remarkable that, when working in semidirect coordinates, this distribution, naturally invariant 
by translations for the Engel group law, is also invariant by translations for the quasi-Heisenberg group law when $\Heis^{1,1}$ is realized on the same manifold $\R^4$ through exponential coordinates.
In fact, in these coordinates it is possible to relate the vector fields $Y_1,Y_2,Y_3$ and $Y_4$ to the ones on the Engel group:
\begin{equation*}
    Y_1 = X_1 + x_2 X_2,\ Y_i = X_i,\ i\in\{2,3,4\},
\end{equation*}
and the characteristic direction is then unequivocally given by the vector field $X_2 = Y_2$. Moreover, as submanifolds, $\Heis$ and $\R_2$ 
are well-defined, independently of the choice of exponential maps $\Exp_\Eng$ or $\Exp_{\Heis^{1,1}}$ used, and are subgroups for both group laws.

As the characteristic direction play a primary role in the quantum evolution, the precedent observations lead us to relate the semiclassical calculi of the groups $\Eng$ and $\Heis^{1,1}$. 

To this end, we introduce the relevant element of harmonic analysis on $\Heis^{1,1}$ and fix some notations. 
The following proposition describe the generic part of the dual of $\Heis^{1,1}$.

\begin{proposition}
    The generic dual $\widehat{\Heis}^{1,1}_{\rm gen}$ of $\Heis^{1,1}$ is homeomorphic to $\fg_{3}^*\setminus\{0\}\times\fr_{2}^{*}$ and is made of the following irreducible unitary representations:
    for $(\delta,\beta) \in \fg_{3}^*\setminus\{0\}\times\fr_{2}^{*}$, the representation $\check{\pi}^{\delta,\beta}$ is defined on the Hilbert space $\cH_{\check{\pi}^{\delta,\beta}} := L^{2}(\R_\xi)$ by
    \begin{equation*}
        \forall \varphi \in \cH_{\check{\pi}^{\delta,\beta}},\ \check{\pi}^{\delta,\beta}(x)\varphi(\xi) = 
        {\rm exp}\left[i \delta \left(x_4 + \xi x_3 +\frac{1}{2}x_{1}x_3\right)\right] 
        e^{i\beta x_{2}}\varphi(\xi +x_{1}).
    \end{equation*}
    Moreover, the Plancherel measure is supported in $\widehat{\Heis}^{1,1}_{\rm gen}$ and is given by
    \begin{equation*}
        d\mu_{\Hhat^{1,1}}(\check{\pi}^{\delta,\beta}) = d\mu_{\Hhat}(\Tilde{\pi}^\delta) \frac{d\beta}{\sqrt{2\pi}} = c_{\Heis^{1,1}}|\delta|d\delta d\beta
    \end{equation*}
    with $c_{\Heis^{1,1}}>0$.
\end{proposition}

\begin{remark}
    In what follows, we will denote by $\{\Tilde{\pi}^\delta\,:\,\delta\in\fg_{3}^*\setminus\{0\}\}$ the Schrödinger representations of the Heisenberg group 
    $\Heis$ obtained as restriction of the representations $\pi^{\delta,\beta}$, $\beta\in\fr_{2}^*$.
    It is then straightforward to see that the previous proposition follows from the decomposition $\Heis^{1,1}=\Heis \times\R_2$ in a direct product and 
    the representation $\check{\pi}^{\delta,\beta}$ is given by the product of the Schrödinger representation $\Tilde{\pi}^{\delta}$ and the character $e^{i\beta \cdot}$. 
    The same observation holds for the Plancherel measure.
\end{remark}

We will also consider a particular family of dilations, that we call \textit{partial dilations}, on $\Heis^{1,1}$: for $r>0$ and $x = (x_1,x_2,x_3,x_4)\in\Heis^{1,1}$ in exponential coordinates, we set
\begin{equation*}
    r\cdot_{pt} x = (rx_1,x_2,r^2 x_3, r^3 x_4).
\end{equation*}
Extending this family of dilations to $\widehat{\Heis}^{1,1}$, we define the associated partial semiclassical Fourier transform $\cF_{\Heis^{1,1}}^{\hbar}$ on $\Heis^{1,1}$ as follows: for $f\in\cS(\Heis^{1,1})$,
\begin{equation*}
    \forall \check{\pi}\in\widehat{\Heis}^{1,1},\ \cF_{\Heis^{1,1}_{pt}}^{\hbar}f(\check{\pi}) = \int_{\Heis^{1,1}} f(x)\check{\pi}(\hbar^{-1}\cdot_{pt} x)^{*}\,dx.
\end{equation*}

This definition of the partial semiclassical Fourier transform on $\Heis^{1,1}$ leads us to introduce the following associated quantization of a symbol $\tau\in\cA_0(\Heis^{1,1})$:
\begin{equation*}
    \forall f\in\cS(\Heis^{1,1}),\forall x\in\Heis^{1,1},\ \Op_{\hbar,\Heis_{pt}^{1,1}}(\tau)f(x) = \int_{\Hhat^{1,1}}{\rm Tr}\left(\check{\pi}(x)\tau(x,\hbar \cdot_{pt}\check{\pi})\cF_{\Heis^{1,1}}f(\check{\pi})\right)d\mu_{\Hhat^{1,1}}(\check{\pi}),
\end{equation*}
where $(\hbar\cdot_{pt}\check{\pi})(y) = \check{\pi}(\hbar\cdot_{pt}y)$ for $y\in\Heis^{1,1}$.

\subsubsection{Fourier integral operators}

In the rest of this section, we fix $\nu_0\in\R$. We present here the construction of a particular unitary operator between the Hilbert spaces $L^2(\Heis^{1,1})$ and $L^2(\Eng)$.

\begin{definition}
    \label{def:identificationduals}
    We introduce the following map 
    \begin{equation*}
        \begin{array}{lrcl}
            T_{\nu_0} : & \widehat{\Heis}^{1,1}_{\rm gen} & \longrightarrow & \Enghat_{\rm gen}\\
                & [\check{\pi}^{\delta,\beta}] & \longmapsto & [\pi^{\delta,\beta+\nu_0 \delta^{1/3}}]. 
        \end{array}
    \end{equation*}
    For $\delta_0 >0$, we also define the following sets
    \begin{equation*}
        \begin{aligned}
            \cO_{\delta_0}(\widehat{\Heis}^{1,1}) &=\left\{[\check{\pi}^{\delta,\beta}]\in\widehat{\Heis}^{1,1}\,:\, |\delta| >\delta_0, \,\beta\in\R\right\},\\
            \cO_{\delta_0}(\Enghat) &= \left\{[\pi^{\delta,\beta}]\in\Enghat\,:\, |\delta| >\delta_0, \,\beta\in\R\right\},
        \end{aligned}
    \end{equation*}
    as well as the map $T_{\nu_0}^{\delta_0}$ between them as the restriction of $T$:
    \begin{equation*}
        \begin{array}{lrcl}
            T_{\nu_0}^{\delta_0} : & \cO_{\delta_0}(\widehat{\Heis}^{1,1}) & \longrightarrow & \cO_{\delta_0}(\Enghat)\\
                & [\check{\pi}^{\delta,\beta}] & \longmapsto & [\pi^{\delta,\beta+\nu_0 \delta^{1/3}}]. 
        \end{array}
    \end{equation*}
\end{definition}

\begin{proposition}
    \label{prop:identificationduals}
    For $\delta_0>0$ fixed, the sets $\cO_{\delta_0}(\Enghat)$ and $\cO_{\delta_0}(\widehat{\Heis}^{1,1})$ are open in 
    $\Enghat_{\rm gen}$ and $\widehat{\Heis}^{1,1}_{\rm gen}$ respectively. Moreover $T_{\nu_0}$ and $T_{\nu_0}^{\delta_0}$ are homeomorphisms and, up to a certain constant $c>0$,
    we have preservation of the Plancherel measure:
    \begin{equation*}
        T_{\nu_0}^{*} \mu_{\widehat{\Heis}^{1,1}} = c \mu_{\Enghat}.
    \end{equation*}
\end{proposition}

\begin{remark}
    In the rest of this paper, we will not keep track of the exact constant appearing in the Plancherel measures of the different groups,
    and we will make the abuse of writing $c=1$ in our computations.
\end{remark}

As $T_{\nu_0}$ preserves the Plancherel measures, for $\tau\in L^2(\Hhat^{1,1})$, we can consider its pushforward by $T_{\nu_0}$ defined by 
\begin{equation*}
    T_{\nu_0}^{*}\tau = \tau \circ T_{\nu_0}^{-1}.
\end{equation*} 
Moreover, we have $\|\tau\|_{L^2(\Hhat^{1,1})}=\|T_{\nu_0}^{*}\tau\|_{L^2(\Enghat)}$.

We introduce now the operator that will play the role of a Fourier integral operator between the quasi-Heisenberg group $\Heis^{1,1}$ and 
the Engel group $\Eng$.

\begin{definition}
\label{def:unitaryUhbar}
    For $\hbar>0$, we introduce the unitary operator $\bU_{\nu_0}^{\hbar}: L^{2}(\Heis^{1,1}) \rightarrow L^2(\Eng)$ defined by 
    \begin{equation*}
        \bU_{\nu_0}^{\hbar}  = \left(\cF_{\Eng}^{\hbar}\right)^{-1} \circ D^{\hbar} \circ T_{\nu_0}^* \circ \cF_{\Heis_{pt}^{1,1}}^{\hbar},
    \end{equation*} 
    where the operator $\tau \in L^2(\Enghat) \rightarrow D^{\hbar} \tau \in L^2(\Enghat)$ is a dilation along the $\beta$ variable, i.e is defined by
    \begin{equation*}
        D^{\hbar}\tau(\pi^{\delta,\beta}) := \tau(\pi^{\delta,\frac{\beta}{\hbar}}).
    \end{equation*}
\end{definition}

We present now a more practical formula for this unitary operator.

\begin{definition}
    For $\hbar > 0$ and $x_2 \in \R_2$, we denote by $U_{\nu_0}^{\hbar}(x_2,\cdot)$ the following field of operators on $\Hhat$:
    \begin{equation}
        U_{\nu_0}^{\hbar}(x_2,\cdot) = \left\{\pi^{\delta,\nu_0\delta^{1/3}}\left(\frac{x_2}{\hbar}\right) \in \cL(\cH_{\Tilde{\pi}^\delta})\,:\,\Tilde{\pi}^\delta \in \Hhat_{\rm gen}\right\}.
    \end{equation} 
\end{definition}

We can rewrite informally the unitary operator $\bU_{\nu_0}^\hbar$ as follows: for $f \in \cS(\Heis^{1,1})$, we have for $x=(h_x,x_2)\in\Eng=\Heis\rtimes\R_2$:
\begin{multline*}
    \bU_{\nu_0}^{\hbar} f(h_x,x_2) = \hbar^{-6}\int_{\Hhat\times\R}\int_{\Heis\times\R} {\rm Tr}_{\cH_{\Tilde{\pi}^\delta}}\left(\Tilde{\pi}^{\delta}\left(\hbar^{-1}\cdot (h_{y}^{-1}h_x)\right) e^{i\beta(x_2-y_2)} U_{\nu_0}^{\hbar}(x_2,\Tilde{\pi}^\delta)\right) \\
    \times f(h_y,y_2)\,d\mu_{\Hhat}(\Tilde{\pi}^\delta) d\beta dh_y dy_2.
\end{multline*}

Note that the dilation used on the Heisenberg group $\Heis$ is the restriction of the one on the Engel group, and as such will be denoted the same.

We can give a meaning to this informal expression interpreting it as an oscillatory integral. To this end, we introduce $\left(\theta_\eps\right)_{\eps>0}$, an approximation of the identity in $\cS(\Heis)$.

\begin{proposition}
    \label{prop:Uhbarintegralformula}
    The operator $\bU_{\nu_0}^\hbar$ considered as an operator $\cS(\Heis^{1,1})\rightarrow C_{b}(\Eng)$ from the Schwartz space to the space of bounded continuous maps, makes sense and is continuous. 
    Moreover, it admits the following integral representation: for all $(h_x,x_2)\in\Eng$, we have: 
    \begin{equation*}
        \bU_{\nu_0}^{\hbar}f(h_x,x_2) = \lim_{\eps \rightarrow 0} \hbar^{-6}\int_{\Hhat}\int_{\Heis} {\rm Tr}\left(\Tilde{\pi}^{\delta}\left(\hbar^{-1}\cdot (h_{y}^{-1}h_x)\right) U_{\nu_0}^{\hbar}(x_2,\Tilde{\pi}^\delta) \cF_{\Heis}^{\hbar}\theta_\eps (\Tilde{\pi}^{\delta})\right)f(h_y,x_2)\,d\mu_{\Hhat}(\Tilde{\pi}^\delta) dh_y,
    \end{equation*}
    and the convergence is uniform as $\eps$ goes to 0.
\end{proposition}

\begin{proof}
    First, continuity of $\bU_{\nu_0}^{\hbar}f$ follows from the following facts: since $f$ is Schwartz, $\cF_{\Heis^{1,1}_{pt}}f$ is in $L^{1}(\Hhat^{1,1})$ and as $D^{\hbar}$ and $T_{\nu_0}$ preserves the Plancherel measures up to constants,
    $\bU_{\nu_0}^{\hbar}f$ is the inverse Fourier transform of an element of $L^1(\Enghat)$. As such it is continuous on $\Eng$ by dominated convergence. Thus it makes sense to evaluate $\bU_{\nu_0}^{\hbar}f$ at any point.

    Now, let $\chi \in C_{c}^{\infty}(\R)$ a smooth positive compactly supported function satisfying $\chi(0) = 1$.
    We will write for $\eps>0$, $\chi_{\eps}(x_2) = \eps^{-1}\chi\left(\frac{x_2}{\eps}\right)$ and $\Tilde{\theta}_{\eps}(h_x,x_2) = \theta_{\eps}(h_x)\chi_{\eps}(x_2)$ for $(h_x,x_2)\in\Heis\times\R$. 
    Let $f\in\cS(\Heis^{1,1})$: we have $f *_{\Heis^{1,1}} \Tilde{\theta}_{\eps} \in \cS(\Heis^{1,1})$ and 
    for $(h_x,x_2)\in\Heis\times\R$:  
    \begin{equation*}
        \begin{aligned}
        &\bU_{\nu_0}^{\hbar}(f * \Tilde{\theta}_{\eps})(h_x,x_2) = \hbar^{-7} \int_{\Enghat} {\rm Tr}\left(\Tilde{\pi}^{\delta}(\hbar^{-1}\cdot h_x)e^{i\beta \frac{x_2}{\hbar}}U_{\nu_0}^{\hbar}(x,\Tilde{\pi}^\delta) \cF_{\Heis^{1,1}_{pt}}^{\hbar}(\Tilde{\theta}_\eps)(\check{\pi}^{\delta,\hbar^{-1} \beta})\cF_{\Heis^{1,1}_{pt}}^{\hbar}f(\check{\pi}^{\delta,\hbar^{-1}\beta})\right)\,d\mu_{\Enghat}(\pi^{\delta,\beta})\\
        &= \hbar^{-6} \int_{\Hhat \times \R} \int_{\Heis\times\R}{\rm Tr}\left(\Tilde{\pi}^{\delta}\left(\hbar^{-1}\cdot (h_{y}^{-1}h_x)\right)e^{i\beta (x_2-y_2)} U_{\nu_0}^{\hbar}(x,\Tilde{\pi}^\delta)\cF_{\Heis^{1,1}_{pt}}^{\hbar}(\Tilde{\theta}_\eps)(\check{\pi}^{\delta,\beta})\right)f(h_y,y_2) \,d\mu_{\Hhat}(\Tilde{\pi}^\delta) d\beta dh_y dy_2.
        \end{aligned}
    \end{equation*} 
    We almost recognize an inverse Fourier formula in the variables $y_2$ and $\beta$: for this, we write 
    \begin{equation*}
        \cF_{\Heis^{1,1}_{pt}}^{\hbar}(\Tilde{\theta}_\eps)(\check{\pi}^{\delta,\beta}) = \cF_{\R} \chi_{\eps}(\beta) \cF_{\Heis}^{\hbar}\theta_{\eps}(\Tilde{\pi}^{\delta}).
    \end{equation*}
    And as we have the following asymptotic in $\cS(\R_2)$
    \begin{equation*}
        \int_{\R}\int_{\R} e^{i\beta(x_2 -y_2)} \widehat{\chi_\eps}(\beta) f(h_y,y_2) \,d\beta dy_2 \Tend{\eps}{0} f(h_y,x_2),
    \end{equation*} 
    uniformly in $h_y \in \Heis$, we are left with 
    \begin{equation*}
        \bU_{\nu_0}^{\hbar}(f \,*\, \Tilde{\theta}_{\eps})(h_x,x_2) = \hbar^{-6}\int_{\Hhat}\int_{\Heis} {\rm Tr}\left(\Tilde{\pi}^{\delta}\left(\hbar^{-1}\cdot (h_{y}^{-1}h_x)\right) U_{\nu_0}^{\hbar}(x_2,\Tilde{\pi}^\delta) \cF_{\Heis}^{\hbar}\theta_\eps (\Tilde{\pi}^{\delta})\right)f(h_y,x_2)\,d\mu_{\Hhat}(\Tilde{\pi}^\delta) dh_y + \cO_{\cS(\Eng)}(\eps).
    \end{equation*}
    Notice now that we have
    \begin{equation*}
        \begin{aligned}
            \left\|\bU_{\nu_0}^{\hbar}(f - f * \Tilde{\theta}_{\eps})\right\|_{L^{\infty}(\Eng)} 
            &\lesssim \left\|\cF_{\Heis_{pt}^{1,1}}^{\hbar}(f - f * \Tilde{\theta}_{\eps})\right\|_{L^1(\widehat{\Heis}^{1,1})}\\
            &\lesssim \left\|\cF_{\Heis_{pt}^{1,1}}^{\hbar}\left((1+\Delta_{\Heis^{1,1}})^{N}(f - f * \Tilde{\theta}_{\eps})\right)\right\|_{L^\infty(\Hhat^{1,1})}\\
            &\lesssim \left\|(1+\Delta_{\Heis^{1,1}})^{N}(f - f * \Tilde{\theta}_{\eps})\right\|_{L^1(\Heis^{1,1})}.
        \end{aligned}
    \end{equation*}
    where we have introduced the subLaplacian $\Delta_{\Heis^{1,1}} = \sum_{1\leq i \leq 3} Y_{i}^2$ on the quasi-Heisenberg group and we have chosen $N\in\N$ large enough and independently of $\eps$.
    Since $f$ is in $\cS(\Heis^{1,1})$, the last term goes to 0 when $\eps$ tends to 0. We conclude that the integral formula holds true and the convergence is uniform on $\Eng$.
\end{proof}

As a technical tool, we define a similar operator as $\bU_{\nu_0}^\hbar$ where we introduce a cut-off function away from the neighborhood $\{\delta=0\}$.

\begin{definition}
    \label{def:Uhbardelta0}
    Let $g\in C^{\infty}(\R)$ a smooth function equal to 0 on $(-1/2,1/2)$ and to 1 outside $(-1,1)$. For $\delta_0>0$, we introduce the operators $\bU_{\nu_0,\delta_0}^{\hbar}: L^2(\Heis^{1,1})\rightarrow L^2(\Eng)$ defined by 
    \begin{equation*}
        \bU_{\nu_0,\delta_0}^{\hbar}  = \left(\cF_{\Eng}^{\hbar}\right)^{-1} \circ D_{\delta_0}^{\hbar}\circ T_{\nu_0}^* \circ \cF_{\Heis_{pt}^{1,1}}^{\hbar},
    \end{equation*}
    where the operator $D_{\delta_0}^{\hbar}: L^2(\Enghat)\rightarrow L^2(\Enghat)$ is defined by 
    \begin{equation*}
        D_{\delta_0}^{\hbar}\tau(\pi^{\delta,\beta}) := g\left(\frac{\delta}{\delta_0}\right)\tau(\pi^{\delta,\frac{\beta}{\hbar}}).
    \end{equation*}
    The operators $\bU_{\nu_0,\delta_0}^{\hbar}$ also admit an integral representation: for $f\in\cS(\Heis^{1,1})$ and $(h_x,x_2)\in\Eng$,
    \begin{multline*}
        \bU_{\nu_0,\delta_0}^{\hbar}f(h_x,x_2) = \lim_{\eps \rightarrow 0} \hbar^{-6}\int_{\Hhat\times\Heis} {\rm Tr}\left(\Tilde{\pi}^{\delta}\left(\hbar^{-1}\cdot (h_{y}^{-1}h_x)\right) U_{\nu_0}^{\hbar}(x_2,\Tilde{\pi}^\delta) g\left(\frac{\delta}{\delta_0}\right)\cF_{\Heis}^{\hbar}\theta_\eps (\Tilde{\pi}^{\delta})\right) \\
        \times f(h_y,x_2)\,d\mu_{\Hhat}(\Tilde{\pi}^\delta)dh_y,
    \end{multline*}
    with uniform convergence as $\eps$ goes to 0.
\end{definition}

\subsubsection{Intertwining of the two quantizations}
The next Theorem is the fundamental result of this section as it relates the second-microlocal semiclassical analysis of the Engel group $\Eng$ to the partial semiclassical analysis of $\Heis^{1,1}$ (that will be studied in more detail later)
in a straightforward way.

\begin{lemma}
    \label{lem:correspondencesmoothsym}
    Let $\sigma \in \mathcal{A}_{0}^{2}(\Eng)$ a second-microlocal symbol supported in $\Eng\times\cO_{\delta_0}(\Eng)\times[-R,R]$ for some $\delta_0>0$ and $R >0$.
    Then the symbol $\tau$ defined on $\Heis^{1,1}\times\widehat{\Heis}^{1,1}_{\rm gen}$ by 
    \begin{equation*}
        \forall x \in \Eng\sim\Heis^{1,1},\forall (\delta,\beta)\in\fg_{3}^*\setminus\{0\} \times\fr_{2}^*,\ \tau(x,\check{\pi}^{\delta,\beta}) = \sigma(x,\pi^{\delta,\nu_0\delta^{1/3}},\beta),
    \end{equation*}
    is a smoothing symbol, i.e $\tau$ belongs to $\cA_{0}(\Heis^{1,1})$.
\end{lemma}

This lemma allow us to make sense of our main result of the section.

\begin{theorem}
    \label{thm:sym2microclosetoCnu}
    Let $\sigma \in \mathcal{A}_{0}^{2}(\Eng)$ with compact support on the last variable $\eta$.
    Then if we define the symbol $\tau \in \cA_{0}(\Heis^{1,1})$ as in Lemma \ref{lem:correspondencesmoothsym}, we have
    \begin{equation*}
        \Op_{\hbar,\Eng}^{C_{\nu_0}}(\sigma) = \bU_{\nu_0}^{\hbar} \circ \Op_{\hbar,\Heis_{pt}^{1,1}}(\tau) \circ \bU_{\nu_0}^{\hbar,*} + \cO_{L^2 \rightarrow L^2}(\hbar).
    \end{equation*}
\end{theorem}

The rest of the Section is devoted to the proof of Lemma \ref{lem:correspondencesmoothsym} and Theorem \ref{thm:sym2microclosetoCnu}.

\begin{proof}[Proof of Lemma \ref{lem:correspondencesmoothsym}]
    Let $\sigma \in \cA_{0}^2(\Eng)$ as above. By definition \ref{def:sym2micro} and by compacity of $\sigma$ in the variable $\eta$, the associated kernel $\kappa$ is in
    $C_{c}^{\infty}(\R_\eta,C_{c}^\infty(\Eng_x,\cS(\Eng_y)))$. For $x\in\Heis^{1,1}$, the measurable symbol $\tau(x,\cdot)$ is in $L^\infty(\widehat{\Heis}^{1,1})$ and thus it is the Fourier transform 
    of a distribution $\kappa^{\tau}_x(\cdot)\in\cS'(\Heis^{1,1})$: 
    \begin{equation*}
        \tau(x,\cdot) = \cF_{\Heis^{1,1}}\kappa^{\tau}_x(\cdot).
    \end{equation*}
    As smoothness in the variable $x$ is easily checked, we focus on proving smoothness and decay with respect to the variable $y$ of the distribution $\kappa^{\tau}_x$.

    First, we prove boundedness and continuity in $y$. To this end, we show that $\tau(x,\cdot)$ is in $L^1(\widehat{\Heis}^{1,1})$: indeed we can write, for $N\in\N_{>0}$,
    \begin{equation*}
        \begin{aligned}
            \sigma(x,\pi^{\delta,\nu_0\delta^{1/3}},\beta) &= \cF_{\Eng}(\kappa_{x}(\cdot,\beta))(\pi^{\delta,\nu_0\delta^{1/3}})\\
            &= H(\pi^{\delta,\nu_0 \delta^{1/3}})^{-N}H(\pi^{\delta,\nu_0 \delta^{1/3}})^N\cF_{\Eng}(\kappa_{x}(\cdot,\beta))(\pi^{\delta,\nu_0\delta^{1/3}})\\
            &= H(\pi^{\delta,\nu_0 \delta^{1/3}})^{-N}\cF_{\Eng}(\Delta_{\Eng}^N \kappa_{x}(\cdot,\beta))(\pi^{\delta,\nu_0\delta^{1/3}}).
        \end{aligned}
    \end{equation*}
    As $H(\pi^{\delta,\nu_0 \delta^{1/3}})$ is unitarily equivalent to the operator 
    \begin{equation*}
        \delta^{2/3}\left(-\partial_{\xi}^2 + \left(\nu_0+\frac{1}{2}\xi^2\right)^2\right),
    \end{equation*}
    for $N$ large enough the operator $\delta^{2N/3}H(\pi^{\delta,\nu_0 \delta^{1/3}})^{-N}$ is trace-class, with a uniformly bounded trace norm (see \cite{BBGL}). As the support 
    of $\tau(x,\cdot)$ is included in $\cO_{\delta_0}(\widehat{\Heis}^{1,1})$, up to choosing $N$ large enough, ${\bold 1}_{\delta>\delta_0}\delta^{-2N/3}$ is integrable for the $|\delta|d\delta$ part of 
    the Plancherel measure on $\Hhat^{1,1}$ and as $\kappa$ is compactly supported with respect to the variable $\eta$, we obtain the integrability of ${\rm Tr}(|\tau(x,\cdot)|)$ for $d\mu_{\Hhat^{1,1}}$ as desired.

    We are then allowed to use the Fourier inversion formula and write
    \begin{equation}
        \label{eq:prooftau}
            \kappa^{\tau}_x(y) = \int_{\Hhat^{1,1}}{\rm Tr}\left(\check{\pi}(y) \tau(x,\check{\pi})\right)d\mu_{\Hhat^{1,1}}(\check{\pi}).
    \end{equation}
    We see by this equation that $\kappa^{\tau}_x(\cdot)$ is continuous and uniformly bounded.
    
    Now, for a left-invariant vector field $Y$ on $\Heis^{1,1}$, formula \eqref{eq:prooftau} would give 
    \begin{equation*}
        Y_{y}\kappa^{\tau}_x(y) = \int_{\Hhat^{1,1}}{\rm Tr}\left(\check{\pi}(y) \check{\pi}(Y)\tau(x,\check{\pi})\right)d\mu_{\Hhat^{1,1}}(\check{\pi}),
    \end{equation*}
    if we had $\check{\pi}\in\Hhat^{1,1}\mapsto \check{\pi}(Y)\tau(x,\check{\pi})$ also in $L^1(\Hhat^{1,1})$. As we have
    \begin{equation}
        \label{eq:quasiengelpi}
        \check{\pi}^{\delta,\beta}(Y_2) = i\beta,\quad \check{\pi}^{\delta,\beta}(Y_i) = \pi^{\delta,\nu_0\delta^{1/3}}(X_i),\  i=1,3,4,
    \end{equation}
    we can write for $i=1,3,4$,
    \begin{equation*}
        \check{\pi}^{\delta,\beta}(Y_i)\tau(x,\check{\pi}^{\delta,\beta}) = \pi^{\delta,\nu_0\delta^{1/3}}(X_i)\sigma(x,\pi^{\delta,\nu_0\delta^{1/3}},\beta) = \cF_{\Eng}(X_i \kappa_x(\cdot,\beta))(\pi^{\delta,\nu_0\delta^{1/3}}),
    \end{equation*}
    and 
    \begin{equation*}
        \check{\pi}^{\delta,\beta}(Y_2)\tau(x,\check{\pi}^{\delta,\beta}) = \cF_{\Eng}(i\beta \kappa_x(\cdot,\beta))(\pi^{\delta,\nu_0\delta^{1/3}}).
    \end{equation*}
    Recursively and applying the same previous reasoning, we obtain smoothness in $y$.

    For decay, notice that for $j\in\{1,2,3,4\}$, Equation \eqref{eq:prooftau} gives 
    \begin{equation*}
        -y_j\kappa^{\tau}_x(y) = \int_{\R\setminus\{0\}\times\R}{\rm Tr}\left(-\Delta_{j}^{\Heis^{1,1}} \check{\pi}^{\delta,\beta}(y) \tau(x,\check{\pi}^{\delta,\beta})\right)|\delta|d\delta d\beta,
    \end{equation*}
    where the difference operator $\Delta_{j}^{\Heis^{1,1}}$ is applied to the field $\{\check{\pi}^{\delta,\beta}(y)\,:\,\check{\pi}^{\delta,\beta}\in\widehat{\Heis}^{1,1}\}$. By Proposition \ref{prop:formuleLeibniz}, we have 
    \begin{equation*}
        -y_j\kappa^{\tau}_x(y) = \int_{\R\setminus\{0\}\times\R}{\rm Tr}\left(\check{\pi}^{\delta,\beta}(y) \Delta_{j}^{\Heis^{1,1}}\tau(x,\check{\pi}^{\delta,\beta})\right)|\delta|d\delta d\beta,
    \end{equation*}
    if $\check{\pi}\in\Hhat^{1,1}\mapsto \Delta_{j}^{\Heis^{1,1}}\tau(x,\check{\pi})$ makes sense and is also in $L^1(\Hhat^{1,1})$.
    As we can explicit the difference operators on $\Heis^{1,1}$ as follows
    \begin{equation*}
        \begin{aligned}
        &\Delta_{1}^{\Heis^{1,1}} = \frac{i}{\delta} \left[\check{\pi}^{\delta,\beta}(Y_3), \cdot \right],\ 
        \Delta_{2}^{\Heis^{1,1}} = \frac{1}{i}\partial_{\beta},\ 
        \Delta_{3}^{\Heis^{1,1}} = - \frac{i}{\delta}\left[\check{\pi}^{\delta,\beta}(Y_1),\cdot\right],\\
        &\Delta_{4}^{\Heis^{1,1}} = \frac{1}{i}\partial_{\delta} + \frac{1}{2}\Delta_{1}^{\Heis^{1,1}} \Delta_{3}^{\Heis^{1,1}}  + \frac{i}{\delta} \check{\pi}^{\delta,\beta}(Y_3) \Delta_{3}^{\Heis^{1,1}},
        \end{aligned}
    \end{equation*}
    using the relations \eqref{eq:quasiengelpi}, we see that the action of a difference operator for the quasi-Heisenberg group easily translates into 
    a difference operator on the Engel group acting on the symbol $\sigma$ or into derivatives along the variable $\eta$. We conclude this way by recurrence that 
    $\kappa_{x}^\tau$ is in $\cS(\Heis^{1,1})$.
\end{proof}

The following lemma is a necessary tool in the proof of Theorem \ref{thm:sym2microclosetoCnu} as it allows us to microlocalize away from a neighborhood of $\{\delta = 0\}$.

\begin{lemma}
    \label{lem:calculsymboliqueproof}
    Let $\tau \in \cA_{0}(\Heis^{1,1})$ be a symbol with support in $\Heis^{1,1}\times\cO_{\delta_0}(\widehat{\Heis}^{1,1})$ for a certain $\delta_0 > 0$. Then we have 
    \begin{equation*}
        \bU_{\nu_0}^{\hbar}\circ \Op_{\hbar,\Heis_{pt}^{1,1}}(\tau) = \bU_{\nu_0,\delta_0}^{\hbar}\circ \Op_{\hbar,\Heis_{pt}^{1,1}}(\tau) + \cO_{L^2\rightarrow L^2}(\hbar).
    \end{equation*}
\end{lemma}

\begin{proof}
    Let $g\in C^{\infty}(\R)$ as in \ref{def:Uhbardelta0}, i.e equal to 0 on $(-1/2,1/2)$ and 1 outside $(-1,1)$. 
    Then the symbol $\sigma_{g_{\delta_0}} = \{g_{\delta_0}(\delta) = g\left(\frac{\delta}{\delta_0}\right){\rm Id}\,:\,\check{\pi}^{\delta,\beta}\in\widehat{\Heis}^{1,1}\}$ is in $S_{0,0}^{0}(\Heis^{1,1})$ and satisfies by definition
    \begin{equation*}
        \bU_{\nu_0,\delta_0}^{\hbar} = \bU_{\nu_0}^{\hbar}\circ \Op_{\hbar,\Heis_{pt}^{1,1}}(\sigma_{g_{\delta_0}}).
    \end{equation*}
    Thus it is enough to show that $\Op_{\hbar,\Heis_{pt}^{1,1}}(\sigma_{g_{\delta_0}})\circ \Op_{\hbar,\Heis_{pt}^{1,1}}(\tau)=\Op_{\hbar,\Heis_{pt}^{1,1}}(\tau)+\cO(\hbar)$.
    Note that the operators $\Op_{\hbar,\Heis_{pt}^{1,1}}(\sigma_{g_{\delta_0}})$ acts by convolution with the following kernel:
    \begin{equation*}
        \kappa^{g_{\delta_0},\hbar}(y) = \hbar^{-3}\delta_{(y_1,y_2,y_3)= 0}\otimes \cF_{\R}^{-1}g_{\delta_0}(\hbar^{-3}y_4),\quad y\in\Heis^{1,1},
    \end{equation*}
    while, if we denoteby $\kappa^{\tau}$ the convolution kernel of $\tau$, the operator $\Op_{\hbar,\Heis_{pt}^{1,1}}(\tau)$ acts by 
    \begin{equation*}
        \Op_{\hbar,\Heis_{pt}^{1,1}}(\tau)f(x) = f * \kappa_{x}^{\tau,\hbar}(x),\quad f\in\cS(\Heis^{1,1}),\ x\in\Heis^{1,1}.
    \end{equation*}
    with $\kappa_{x}^{\tau,\hbar}(y) = \hbar^{-6}\kappa_{x}^{\tau}(\hbar^{-1}\cdot_{pt}y)$.
    Thus the composition of the two operators has a convolution kernel $\kappa_{x}^\hbar(z) = \hbar^{-6}\kappa_{x}(\hbar^{-1}\cdot_{pt} z)$ given by
    \begin{equation*}
        \begin{aligned}
        \kappa_{x}^{\hbar}(y^{-1}x) &= \hbar^{-9}\int_{\R} \cF_{\R}^{-1}g_{\delta_0}(\hbar^{-3}(x_4-v_4))\kappa_{(x_1,x_2,x_3,v_4)}^{\tau}(\hbar^{-1}\cdot_{pt}(y^{-1}(x_1,x_2,x_3,v_4)))\,dv_4\\
        &= \hbar^{-6}\int_{\R} \cF_{\R}^{-1}g_{\delta_0}(u_4)\kappa_{x (\hbar\cdot_{pt}u_{4}^{-1})}^{\tau}((\hbar^{-1}\cdot_{pt}y^{-1}x)u_{4}^{-1})\,du_4,
        \end{aligned}
    \end{equation*}
    and with the change of variable $z = \hbar^{-1}\cdot_{pt}(y^{-1}x)$,
    \begin{equation*}
        \kappa_{x}(z) = \int_{\R} \cF_{\R}^{-1}g_{\delta_0}(u_4)\kappa_{x (\hbar\cdot_{pt}u_{4}^{-1})}^{\tau}(zu_{4}^{-1})\,du_4.
    \end{equation*}
    We can perform a Taylor expansion of the ma
    $u_4 \in\R \mapsto \kappa_{x (\hbar\cdot_{pt}u_{4}^{-1})}^{\tau}(z)$:
    \begin{equation*}
        \kappa_{x (\hbar\cdot_{pt}u_{4}^{-1})}^{\tau}(z) = \kappa_{x}^{\tau}(z) - \hbar^{3}u_4\int_{0}^{1} X_{4,x} \kappa_{x(\hbar\cdot_{pt}(tu_4)^{-1})}(z)\,dt,
    \end{equation*}
    and we deduce 
    \begin{equation*}
        \kappa_{x}(z) = \int_{\R} \cF_{\R}^{-1}g_{\delta_0}(u_4)\kappa_{x}^{\tau}(zu_{4}^{-1})\,du_4 - \hbar^{3} \int_{\R} u_4\cF_{\R}^{-1}g_{\delta_0}(u_4) X_{4,x}\kappa_{x (\hbar\cdot_{pt}u_{4}^{-1})}^{\tau}(zu_{4}^{-1})\,du_4.
    \end{equation*}
    The first term in the right-hand side is smooth and compactly supported in $x$ and Schwartz in $z$, and defines, by taking its Fourier transform, a symbol $\sigma$ in $\cA_{0}(\Heis^{1,1})$ equal to 
    \begin{equation*}
        \sigma = \sigma_{g_{\delta_0}}\circ \tau = \tau,
    \end{equation*}
    by condition on the support of $\tau$. Finally, we see that the operator of convolution kernel equal to 
    \begin{equation*}
        \int_{\R} u_4\cF_{\R}^{-1}g_{\delta_0}(u_4) X_{4,x}\kappa_{x (\hbar\cdot_{pt}u_{4}^{-1})}^{\tau}(zu_{4}^{-1})\,du_4,
    \end{equation*}
    is bounded on $L^2(\Heis^{1,1})$ by Schur's lemma, uniformly in $\hbar$, and we obtain the desired equality.
\end{proof}

\begin{proof}[Proof of Theorem \ref{thm:sym2microclosetoCnu}]
    As $\sigma$ is a symbol in $\mathcal{A}_{0}^{2}(\Eng)$, there exists $\delta_0>0$ such that $\sigma$ is supported in 
    $\Eng\times\cO_{\delta_0}(\Enghat)\times[-R,R]$ for $R>0$ large enough.  
    By Lemma \ref{lem:correspondencesmoothsym} we know that the measurable symbol $\tau$ defines indeed a smoothing symbol on $\Heis^{1,1}$.
    
    First, we observe that, for $f \in \cS(\Heis^{1,1})$ and $x= (h_x,x_2)\in\Eng$, we can write
    \begin{equation*}
        \begin{aligned}
        &\Op_{\hbar}^{C_{\nu_0}}(\sigma)(\bU_{\nu_0}^{\hbar}f)(x) = \hbar^{-7}\int_{\Enghat} {\rm Tr}\left(\pi^{\delta,\beta}(\hbar^{-1}\cdot x)\sigma_{\hbar}^{C_{\nu_0}}(x,\pi^{\delta,\beta})\cF_{\Eng}^{\hbar}(\bU_{\nu_0}^{\hbar}f)(\pi^{\delta,\beta})\right)d\mu_{\Enghat}(\pi^{\delta,\beta})\\
        &= \hbar^{-7}\int_{\Enghat} {\rm Tr}\left(\pi^{\delta,\beta}(\hbar^{-1}\cdot x)\sigma_{\hbar}^{C_{\nu_0}}(x,\pi^{\delta,\beta})(D^{\hbar}\circ T_{\nu_0}^* \circ \cF_{\Heis^{1,1}_{pt}}^{\hbar})f (\pi^{\delta,\beta})\right)d\mu_{\Enghat}(\pi^{\delta,\beta})\\
        &= \hbar^{-6}\int_{\Enghat} {\rm Tr}\left(\Tilde{\pi}^{\delta}(\hbar^{-1}\cdot h_x)e^{i\beta x_2}U_{\nu_0}^{\hbar}(x_2,\Tilde{\pi}^\delta)\sigma_{\hbar}^{C_{\nu_0}}(x,\pi^{\delta,\nu_0\delta^{1/3}+\hbar\beta})\cF_{\Heis^{1,1}_{pt}}^{\hbar}f(\check{\pi}^{\delta,\beta})\right)d\mu_{\Enghat}(\pi^{\delta,\beta})\\
        &= \hbar^{-6}\int_{\Enghat}\int_{\Heis\times\R} {\rm Tr}\left(\Tilde{\pi}^{\delta}(\hbar^{-1}\cdot (h_{y}^{-1}h_x))e^{i\beta (x_2-y_2)}U_{\nu_0}^{\hbar}(x_2,\Tilde{\pi}^\delta)\sigma_{\hbar}^{C_{\nu_0}}(x,\pi^{\delta,\nu_0\delta^{1/3}+\hbar\beta})\right)f(h_y,y_2) \,dh_y dy_2 d\mu_{\Enghat}(\pi^{\delta,\beta})\\
        &= \hbar^{-6}\int_{\Enghat}\int_{\Heis\times\R} {\rm Tr}\left(\check{\pi}^{\delta,\beta}(\hbar^{-1}\cdot_{pt}(y^{-1}x))U_{\nu_0}^{\hbar}(x_2,\Tilde{\pi}^\delta)\sigma(x,\pi^{\delta,\nu_0\delta^{1/3}+\hbar\beta},\beta)\right)f(h_y,y_2) \,dh_y dy_2 d\mu_{\Enghat}(\pi^{\delta,\beta}).   
        \end{aligned}
    \end{equation*}
    If we denote by $\kappa$ the convolution kernel of $\sigma$ (see Definition \ref{def:sym2micro}), notice that we have 
    \begin{equation*}
        \begin{aligned}
            \sigma(x,\pi^{\delta,\nu_0\delta^{1/3}+\hbar\beta},\beta)-\sigma(x,\pi^{\delta,\nu_0\delta^{1/3}},\beta) 
            &= \cF_{\Eng}\kappa_x(\cdot,\beta)(\pi^{\delta,\nu_0\delta^{1/3}+\hbar\beta})-\cF_{\Eng}\kappa_x(\cdot,\beta)(\pi^{\delta,\nu_0\delta^{1/3}})\\
            &=\int_{\Eng} \kappa_x(y,\beta)\left(\pi^{\delta,\nu_0\delta^{1/3}+\hbar\beta}(y)^* - \pi^{\delta,\nu_0\delta^{1/3}}(y)^*\right)\,dy\\
            &= \int_{\Eng} \kappa_x(y,\beta)\pi^{\delta,\nu_0\delta^{1/3}+\hbar\beta}(y)^* (1-e^{i\hbar\beta y_2})\,dy\\
            &= -i\hbar\int_{0}^{1}\int_{\Eng} \beta y_2\kappa_x(y,\beta)e^{it\hbar\beta y_2}\pi^{\delta,\nu_0\delta^{1/3}+\hbar\beta}(y)^* \,dydt.
        \end{aligned}
    \end{equation*}
    For $t\in [0,1]$, thanks to the support condition on the variable $\eta$, we introduce the kernel $\Tilde{\kappa}^{t,\hbar} \in C_{c}^\infty(\R_\eta,C_{c}^{\infty}(\Eng_x,\cS(\Eng_y)))$ defined by 
    \begin{equation*}
        \Tilde{\kappa}^{t,\hbar}_x(y,\eta) = \eta y_2 e^{it\hbar\eta y_2}\kappa_x(y,\eta).
    \end{equation*}
    If we denote by $\sigma_{t,\hbar}$ the associated second-microlocal symbol, the previous equation writes 
    \begin{equation*}
        \sigma(x,\pi^{\delta,\nu_0\delta^{1/3}+\hbar\beta},\beta)-\sigma(x,\pi^{\delta,\nu_0\delta^{1/3}},\beta) = -i\hbar\int_{0}^{1} \sigma_{t,\hbar}(x,\pi^{\delta,\nu_0\delta^{1/3}+\hbar\beta},\beta)\,dt.
    \end{equation*}
    Thus we have 
    \begin{multline*}
        \Op_{\hbar}^{C_{\nu_0}}(\sigma)(\bU_{\nu_0}^{\hbar}f)(x) = \hbar^{-6}\int_{\Enghat}\int_{\Heis\times\R} {\rm Tr}\left(\check{\pi}^{\delta,\beta}(\hbar^{-1}\cdot_{pt}(y^{-1}x))U_{\nu_0}^{\hbar}(x_2,\Tilde{\pi}^\delta)\sigma(x,\pi^{\delta,\nu_0\delta^{1/3}},\beta)\right)\\f(h_y,y_2) \,dh_y dy_2 d\mu_{\Enghat}(\pi^{\delta,\beta})
        - i\hbar\int_{0}^{1}\Op_{\hbar}^{C_{\nu_0}}(\sigma_{t,\hbar})(\bU_{\nu_0}^{\hbar}f)(x)\,dt.
    \end{multline*}
    By looking at the definition of $\kappa^{t,\hbar}$, we see that the seminorm $N_2(\sigma_{t,\hbar})$ is uniformly bounded in $t\in[0,1]$ and $\hbar\in(0,1)$. We have thus proved by Proposition \ref{prop:2microbounded} that we have:
    \begin{multline*}
        \Op_{\hbar}^{C_{\nu_0}}(\sigma)(\bU_{\nu_0}^{\hbar}f)(x) = \hbar^{-6}\int_{\Heis^{1,1}}\int_{\Heis\times\R} {\rm Tr}\left(\check{\pi}^{\delta,\beta}(\hbar^{-1}\cdot_{pt}(y^{-1}x))U_{\nu_0}^{\hbar}(x_2,\Tilde{\pi}^\delta)\tau(x,\check{\pi}^{\delta,\beta})\right)\\ 
        \times f(y) \,dy d\mu_{\Hhat^{1,1}}(\check{\pi}^{\delta,\beta}) +\cO(\hbar)\|f\|_{L^2}.
    \end{multline*}

    We now compute the composition of $\bU_{\nu_0}^{\hbar}$ with $\Op_{\Heis^{1,1}_{pt}}^{\hbar}(\tau)$. Remark that $\tau$ is a symbol in $\cA_{0}(\Heis^{1,1})$ with support included in $\Eng\times\cO_{\delta_0}(\widehat{\Heis}^{1,1})$. 
    By Lemma \ref{lem:calculsymboliqueproof}, we can instead look at $\bU_{\nu_0,\delta_{0}}^{\hbar}\circ \Op_{\hbar,\Heis^{1,1}}(\tau)$.
    Using Proposition \ref{prop:Uhbarintegralformula} with an approximation of identity defined by 
    \begin{equation*}
        \forall h\in\Heis,\ \theta_\eps(h) = \eps^{-6}\theta(\eps^{-1}\cdot h),\ \eps>0,
    \end{equation*}
    with $\theta \in \cS(\Heis)$ satisfying $\theta(0)=1$, we can write
    \begin{multline*}
        \bU_{\nu_0,\delta_0}^{\hbar}\left(\Op_{\Heis^{1,1}_{pt}}^{\hbar}(\tau)f\right)(x) = \lim_{\eps\rightarrow 0} \hbar^{-6}\int_{\Hhat}\int_{\Heis} {\rm Tr}\left(\Tilde{\pi}^{\delta}(\hbar^{-1}\cdot (h_{w}^{-1}h_x))U_{\nu_0}^{\hbar}(x_2,\Tilde{\pi}^\delta)g_{\delta_0}(\delta)\cF_{\Heis}^{\hbar}\theta_\eps (\Tilde{\pi}^{\delta})\right)\\ \Op_{\Heis^{1,1}_{pt}}^{\hbar}(\tau)f(h_w,x_2)\,d\mu_{\Hhat}(\Tilde{\pi}^\delta) dh_w\\
        = \lim_{\eps\rightarrow 0} \hbar^{-12}\int_{\Hhat}\int_{\Heis}\int_{\Hhat^{1,1}}\int_{\Heis\times\R} {\rm Tr}\left(\Tilde{\pi}^{\delta}(\hbar^{-1}\cdot (h_{w}^{-1}h_x))U_{\nu_0}^{\hbar}(x_2,\Tilde{\pi}^\delta)g_{\delta_0}(\delta)\cF_{\Heis}^{\hbar}\theta_\eps (\Tilde{\pi}^{\delta})\right)\\{\rm Tr} \left(\Tilde{\pi}^{\delta'}(\hbar^{-1}\cdot (h_{y}^{-1}h_w))e^{i\beta(x_2-y_2)}\tau(h_w , x_2, \check{\pi}^{\delta',\beta})\right)f(h_y,y_2)\, d\mu_{\Hhat^{1,1}}(\check{\pi}^{\delta',\beta}) dh_y dy_2 d\mu_{\Hhat}(\Tilde{\pi}^\delta) dh_w\\
        = \lim_{\eps\rightarrow 0} \hbar^{-12}\int_{\Hhat}\int_{\Heis}\int_{\Hhat^{1,1}}\int_{\Heis\times\R} {\rm Tr}\left(\Tilde{\pi}^{\delta}(\hbar^{-1}\cdot h_{z})^{*}U_{\nu_0}^{\hbar}(x_2,\Tilde{\pi}^\delta)g_{\delta_0}(\delta)\cF_{\Heis}^{\hbar}\theta_\eps (\Tilde{\pi}^{\delta})\right)\\{\rm Tr} \left(\Tilde{\pi}^{\delta'}(\hbar^{-1}\cdot h_z)e^{i\beta(x_2-y_2)}\tau(h_x h_z, x_2, \check{\pi}^{\delta',\beta})\Tilde{\pi}^{\delta'}(\hbar^{-1}\cdot (h_{y}^{-1}h_x))\right) \\ \times f(h_y,y_2)\, d\mu_{\Hhat^{1,1}}(\check{\pi}^{\delta',\beta}) dh_y dy_2 d\mu_{\Hhat}(\Tilde{\pi}^\delta) dh_z,
    \end{multline*}
    where we have performed the change of variable $h_z = h_{x}^{-1}h_w$. 
    
    We recognize in the last term the composition of a Fourier transform on $\Heis$ (with the variable $h_z$) with an inverse Fourier transform with the variable $\Tilde{\pi}^{\delta'}$. 
    There is however an obstruction present in the dependence in the symbol $\tau$ of the variable $h_z$. 
    
    Noting $h_z = \Exp_{\Heis}(z_1 X_1 + z_3 X_3 + z_4 X_4)$, the mean value theorem on the Heisenberg group gives 
    \begin{equation*}
        \tau(h_x h_z,x_2,\check{\pi}^{\delta',\beta}) = \tau(h_x,x_2,\check{\pi}^{\delta',\beta}) + \int_{0}^{1}(z_1 X_1 + z_3 X_3 + z_4 X_4)\tau(h_x (t h_z),x_2,\check{\pi}^{\delta',\beta})\,dt.
    \end{equation*}

    In the precedent computation of the composition of $\bU_{\nu_0}^{\hbar}$ with $\Op_{\Heis^{1,1}_{pt}}^{\hbar}(\tau)$, the first term in the mean value theorem gives 
    \begin{multline*}
        \lim_{\eps\rightarrow 0} \hbar^{-6}\int_{\Hhat^{1,1}}\int_{\Heis\times\R} {\rm Tr}\left(e^{i\beta(x_2-y_2)}\tau(h_x, x_2, \check{\pi}^{\delta',\beta})\Tilde{\pi}^{\delta'}(\hbar^{-1}\cdot (h_{y}^{-1}h_x))U_{\nu_0}^{\hbar}(x_2,\Tilde{\pi}^{\delta'})g_{\delta_0}(\delta')\cF_{\Heis}^{\hbar}\theta_\eps (\Tilde{\pi}^{\delta'})\right)\\ \times f(h_y,y_2)\,d\mu_{\Hhat^{1,1}}(\check{\pi}^{\delta',\beta}) dh_y dy_2\\
        = \lim_{\eps\rightarrow 0} \hbar^{-6}\int_{\Hhat^{1,1}}\int_{\Heis\times\R} {\rm Tr}\left(\Tilde{\pi}^{\delta'}(\hbar^{-1}\cdot (h_{y}^{-1}h_x))e^{i\beta(x_2-y_2)}U_{\nu_0}^{\hbar}(x_2,\Tilde{\pi}^{\delta'})g_{\delta_0}(\delta')\cF_{\Heis}^{\hbar}\theta_\eps (\Tilde{\pi}^{\delta'})\tau(h_x, x_2, \check{\pi}^{\delta',\beta})\right)\\ \times f(h_y,y_2)\,d\mu_{\Hhat^{1,1}}(\check{\pi}^{\delta',\beta}) dh_y dy_2,
    \end{multline*}
    which is equal, in the limit $\eps \rightarrow 0$, by the previous computations, to
    \begin{equation*}
        \Op_{\hbar}^{C_{\nu_0}}(\sigma)(\bU_{\nu_0}^{\hbar}f)(h_x,x_2) + \cO(\hbar)\|f\|_{L^2}.
    \end{equation*}
    
    We are left with proving that the remainder of the mean value theorem gives rise to an operator of norm $\cO_{L^2\rightarrow L^2}(\hbar)$. 
    To this end, we fix $t\in[0,1]$ and remark that for any $i=1,3,4$, we have 
    \begin{equation*}
        \begin{aligned}
        &\Tilde{\pi}^{\delta}(\hbar^{-1}\cdot h_{z})^{*}(\Tilde{\pi}^{\delta})\cF^{\hbar}_\Heis\theta_\eps (\Tilde{\pi}^{\delta}){\rm Tr} \left(\Tilde{\pi}^{\delta'}(\hbar^{-1}\cdot h_z)e^{i\beta(x_2-y_2)}(z_i X_i\tau)(h_x (t h_z), x_2, \check{\pi}^{\delta',\beta})\Tilde{\pi}^{\delta'}(\hbar^{-1}\cdot (h_{y}^{-1}h_x))\right)\\
        &= z_i \Tilde{\pi}^{\delta}(\hbar^{-1}\cdot h_{z})^{*}(\Tilde{\pi}^{\delta})\cF^{\hbar}_\Heis\theta_\eps (\Tilde{\pi}^{\delta}){\rm Tr} \left(\Tilde{\pi}^{\delta'}(\hbar^{-1}\cdot h_z)e^{i\beta(x_2-y_2)}X_i\tau(h_x (t h_z), x_2, \check{\pi}^{\delta',\beta})\Tilde{\pi}^{\delta'}(\hbar^{-1}\cdot (h_{y}^{-1}h_x))\right).
        \end{aligned}
    \end{equation*}
    We observe that 
    \begin{equation*}
        z_i \Tilde{\pi}^{\delta}(\hbar^{-1}\cdot h_{z})^{*} = \hbar^{\upsilon_i} \Delta_{i}^{\Heis}(\Tilde{\pi}^{\delta}(\hbar^{-1}\cdot h_{z})^{*}),
    \end{equation*}
    where the right-hand side is the application of the difference operator $\Delta_{i}^{\Heis}$ to the field $\{\Tilde{\pi}^{\delta}(\hbar^{-1}\cdot h_{z})^{*}\,:\,\Tilde{\pi}^{\delta}\in\Hhat\}$. 
    Applying an integration by part with respect to $\Delta_{i}^{\Heis}$ (see Propostion \ref{prop:IPP}), we have 
    \begin{multline*}
        \int_{\Hhat} {\rm Tr}\left(z_i \Tilde{\pi}^{\delta}(\hbar^{-1}\cdot h_{z})^{*}U_{\nu_0}^{\hbar}(x_2,\Tilde{\pi}^{\delta})g_{\delta_0}(\delta)\cF_{\Heis}^{\hbar}\theta_\eps (\Tilde{\pi}^{\delta})\right)d\mu_{\Hhat}(\Tilde{\pi}^\delta)\\ 
        = -\hbar^{\upsilon_i} \int_{\Hhat} {\rm Tr}\left(\Tilde{\pi}^{\delta}(\hbar^{-1}\cdot h_{z})^{*}\Delta_{i}^{\Heis}\left(U_{\nu_0}^{\hbar}(x_2,\cdot)g_{\delta_0}\cF_{\Heis}^{\hbar}\theta_\eps \right)(\Tilde{\pi}^{\delta})\right)d\mu_{\Hhat}(\Tilde{\pi}^\delta).
    \end{multline*}
    We need to compute the action of the difference operators on $U_{\nu_0}^{\hbar}(x_2,\cdot) g_{\delta_0}\cF_{\Heis}^{\hbar}\theta_\eps$. 
    First observe that we have 
    \begin{equation*}
        \Delta_{i}^{\Heis} \cF_{\Heis}^{\hbar}\theta_\eps = \cF_{\Heis}^{\hbar}(z_i\theta_\eps) = \eps^{\upsilon_i}\cF_{\Heis}^{\hbar}\theta_{\eps}^{i},
    \end{equation*}
    with $\theta_{\eps}^{i}(h_z) = \eps^{-6}(z_i \theta)(\eps^{-1}\cdot h_z)$, and thus, using the Leibniz formula \ref{prop:formuleLeibniz}, these terms disappear in the limit $\eps\rightarrow 0$.
    Moreover, the action of the difference operators on the field $g_{\delta_0}$ gives 0 for all except $\Delta_{4}^{\Heis}$, and this will lead to a term of order $\cO_{L^2\rightarrow L^2}(\hbar^{\upsilon_4})$.
    We can then focus on the action of the difference operators on $U_{\nu_0}^{\hbar}(x_2,\cdot)$. 

    The difference operators on the Heisenberg group for our particular choice of parametrisation of its dual $\Hhat$ are given by the following formulae:
    \begin{equation*}
        \Delta_{1}^{\Heis} = \frac{i}{\delta} \left[\Tilde{\pi}^{\delta}(X_3), \cdot \right],\ 
        \Delta_{3}^{\Heis} = - \frac{i}{\delta}\left[\Tilde{\pi}^{\delta}(X_1),\cdot\right],\ 
        \Delta_{4}^{\Heis} = \frac{1}{i}\partial_{\delta} + \frac{1}{2}\Delta_{1}^{\Heis} \Delta_{3}^{\Heis}  + \frac{i}{\delta} \Tilde{\pi}^{\delta}(X_3) \Delta_{3}^{\Heis},
    \end{equation*}
    We then easily find
    \begin{equation*}
        \begin{aligned}
        &\Delta_1 U_{\nu_0}^{\hbar}(x_2,\cdot) = 0,\ \Delta_3 U_{\nu_0}^{\hbar}(x_2,\cdot) = \hbar^{-1} x_2 \frac{1}{i\delta}U_{\nu_0}^{\hbar}(x_2,\cdot)\pi^{\delta,\beta}(X_3),\\ 
        &\Delta_4 U_{\nu_0}^{\hbar}(x_2,\cdot) = \hbar^{-1} x_2U_{\nu_0}^{\hbar}(x_2,\cdot)\left(\frac{1}{3}\nu_0\delta^{-2/3}+\frac{1}{2}\delta^{-2}\pi^{\delta,\beta}(X_{3}^2)\right).
        \end{aligned}
    \end{equation*}
    The restriction enforced by the cutoff $g$ gives that $\delta$ is bounded by below and thus negative powers of $\delta$ are dominated by powers of $\delta_0$. 
    Moreover, since $\upsilon_3 = 2$, the term obtained by integration by part with respect to $\Delta_3$ is of order $\cO_{L^2 \rightarrow L^2}(\hbar)$, while the one 
    obtained via $\Delta_4$ is of order $\cO_{L^2 \rightarrow L^2}(\hbar^2)$.
\end{proof}

\subsection{Partial semiclassical analysis on the group $\Heis^{1,1}$}
\label{subsect:partialsemianalysis}
As the Fourier integral operators introduced in the previous section allow us to work on the quasi-Heisenberg group,
we present here how the partial semiclassical analysis on $\Heis^{1,1}$ can be reformulated as a particular instance of the semiclassical calculus introduced 
in Appendix \ref{sect:Hilbertvaluedsemiclass}.

The main observation to make is that we have the following isometric identification:
\begin{equation*}
    L^2(\Heis^{1,1}) = L^{2}(\Heis,L^2(\R_2)).
\end{equation*}
Letting $\cG$ denote the Hilbert space $L^2(\R_2)$, we find ourselves in the setting of Appendix \ref{sect:Hilbertvaluedsemiclass}. 
Indeed, for $(h,\Tilde{\pi}^\delta)\in\Heis\times\Hhat$, we can associate to any symbol $\sigma$ in $\cA_{0}(\Heis^{1,1})$ an operator on $\cH_{\Tilde{\pi}^\delta}\otimes\cG = \cH_{\Tilde{\pi}^\delta}\otimes L^2(\R_2)$ in the following way: 
for $\Phi_{1}, \Phi_{2}\in \cH_{\Tilde{\pi}^\delta}$, we consider the scalar-valued smoothing symbol $\sigma_{\Phi_{1},\Phi_{2}}: (x_2,\beta)\in\R_2 \times\R\mapsto \langle \sigma(h,x_2,\check{\pi}^{\delta,\beta})\Phi_{1},\Phi_{2}\rangle_{\cH_{\Tilde{\pi}^\delta}}$. 
We will write $\sigma(h,x_2,\Tilde{\pi}^{\delta},D_{x_2})$ the operator in $\cL(\cH_{\Tilde{\pi}^\delta})\otimes\cK(\cG)$ defined by
\begin{equation*}
    \forall \Phi_{1}, \Phi_{2}\in \cH_{\Tilde{\pi}^\delta},\ \langle \sigma(h,x_2,\Tilde{\pi}^{\delta},D_{x_2})\Phi_{1}, \Phi_{2}\rangle_{\cH_{\Tilde{\pi}^\delta}} = 
    \Op_{1}(\sigma_{\Phi_{1},\Phi_{2}}) \in \cK(\cG),
\end{equation*}
where $\Op_{1}$ denotes here the usual Kohn-Nirenberg quantization on the real line.

\begin{proposition}
    \label{prop:correspondencesymHxR}
    For $\sigma \in \cA_{0}(\Heis^{1,1})$, $\sigma(\cdot,x_2,\cdot,D_{x_2})$ is a symbol in $\cA_{0}^{\cG}(G)$.
\end{proposition}

\begin{proof}
    Let $(h,\Tilde{\pi}^{\delta})\in\Heis\times\Hhat$. We fix $g_1,g_2$ two elements of $\cG$ as well as $\Phi_1,\Phi_2$, two elements of $\cH_{\Tilde{\pi}^{\delta}}$.
    Since $\sigma$ is a smoothing operator on $\Heis^{1,1}$, there exists $\kappa\in C_{c}^{\infty}(\Heis\times\R,\cS(\Heis\times\R))$ such that $\sigma(x,\cdot) = \cF_{\Heis^{1,1}}\kappa_x(\cdot)$, for $x\in\Heis^{1,1}$.
    To show that $\sigma(\cdot,x_2,\cdot,D_{x_2})$ defined as above is in $\cA_{0}^{\cG}(G)$, we need to find the associated convolution kernel $\theta$ and prove that it belongs to $C_{c}^{\infty}(\Heis,\cS(\Heis,\cK(\cG)))$.

    Let consider for that the map defined as follows: for $h_x,h_z \in\Heis$, we write $\cF_{z_2}\kappa(h_x,x_2,h_z,\beta)$ for the Fourier transform of $(h_x,x_2,h_z,z_2)\mapsto\kappa(h_x,x_2,h_z,z_2)$ in the last variable $z_2\in\R$, evaluated at $\beta\in\R$,
    and we write $\theta(h_x,h_z)$ for the operator in $\cK(\cG)$ obtained as the Kohn-Nirenberg quantization of $\cF_{z_2}\kappa(h_x,\cdot,h_z,\cdot)$.

    Following Equation \eqref{eq:BSVfourier}, we compute the Fourier transform on $\Heis$ of $\langle \theta(h_x,h_z)g_1,g_2\rangle_{\cG}$:
    \begin{equation*}
        \begin{aligned}
            \left\langle\cF_{\Heis}\left(\langle\theta(h_x,\cdot)g_1,g_2\rangle_{\cG}\right)(\Tilde{\pi}^{\delta})\,\Phi_1,\Phi_2\right\rangle_{\cH_{\Tilde{\pi}^{\delta}}} 
            &= \int_{\Heis}\langle \theta(h_x,h_z)g_1,g_2\rangle_{\cG}\langle\Tilde{\pi}^{\delta}(h_z)^{*}\Phi_1,\Phi_2\rangle_{\cH_{\Tilde{\pi}^{\delta}}}\,dh_z\\
            &= \left\langle\int_{\Heis}\theta(h_x,h_z)\langle\Tilde{\pi}^{\delta}(h_z)^{*}\Phi_1,\Phi_2\rangle_{\cH_{\Tilde{\pi}^{\delta}}}\,dh_z \,g_1,g_2\right\rangle_{\cG}.
        \end{aligned}
    \end{equation*}
    With the above notations, we are left to show that we have
    \begin{equation*}
        \Op_{1}(\sigma_{\Phi_1,\Phi_2}) = \int_{\Heis}\theta(h_x,h_z)\langle\Tilde{\pi}^{\delta}(h_z)^{*}\Phi_1,\Phi_2\rangle_{\cH_{\Tilde{\pi}^{\delta}}}\,dh_z.
    \end{equation*}
    Indeed, we can write
    \begin{equation*}
        \begin{aligned}
            \int_{\Heis}\theta(h_x,h_z)\langle\Tilde{\pi}^{\delta}(h_z)^{*}\Phi_1,\Phi_2\rangle_{\cH_{\Tilde{\pi}^{\delta}}}\,dh_z &= \int_{\Heis}\Op_{1}(\cF_{z_2}\kappa(h_x,\cdot,h_z,\cdot))\langle\Tilde{\pi}^{\delta}(h_z)^{*}\Phi_1,\Phi_2\rangle_{\cH_{\Tilde{\pi}^{\delta}}}\,dh_z\\
            &= \Op_{1}\left(\int_{\Heis}\cF_{z_2}\kappa(h_x,\cdot,h_z,\cdot)\langle\Tilde{\pi}^{\delta}(h_z)^{*}\Phi_1,\Phi_2\rangle_{\cH_{\Tilde{\pi}^{\delta}}}\,dh_z\right),
        \end{aligned}
    \end{equation*}
    as the Kohn-Nirenberg quantization $\Op_{1}: S^{-\infty}(\R)\rightarrow \cK(L^2(\R))$ is a continuous linear map. By the definition of $\kappa$, we obtain that 
    \begin{equation*}
        \int_{\Heis}\cF_{z_2}\kappa(h_x,x_2,h_z,\beta)\langle\Tilde{\pi}^{\delta}(h_z)^{*}\Phi_1,\Phi_2\rangle_{\cH_{\Tilde{\pi}^{\delta}}}\,dh_z = \langle\sigma(h_x,x_2,\check{\pi}^{\delta,\beta})\Phi_1,\Phi_2\rangle_{\cH_{\Tilde{\pi}^{\delta}}} = \sigma_{\Phi_1,\Phi_2}(h_x,x_2,\check{\pi}^{\delta,\beta}).
    \end{equation*}
    This means that $\theta$ is the associated kernel of $\sigma(\cdot,x_2,\cdot,D_{x_2})$, and as we clearly have $\theta \in C_{c}^{\infty}(\Heis,\cS(\Heis,\cB))$, we obtain that $\sigma(\cdot,x_2,\cdot,D_{x_2})$ is in $\cA_{0}^{\cG}(G)$.
\end{proof}
 
We can now connect the semiclassical calculus with partial dilations on the quasi-Heisenberg group $\Heis^{1,1}$ and the semiclassical calculus of Appendix \ref{sect:Hilbertvaluedsemiclass}
with nilpotent Lie group $G =\Heis$ and Hilbert space $\cG = L^2(\R_2)$.

\begin{proposition}
    \label{prop:correspondenceopHxR}
    For $\sigma \in \cA_{0}(\Heis^{1,1})$, identifying $L^2(\Heis^{1,1})$ with $L^2(\Heis,\cG)$, we have 
    \begin{equation}
        \Op_{\hbar,\Heis_{pt}^{1,1}}(\sigma) = \Op_{\hbar}^{\cG}\left(\sigma(\cdot,x_2,\cdot,D_{x_2})\right).
    \end{equation} 
\end{proposition}

\begin{proof}
    Let $f\in\cS(\Heis^{1,1})$ and $(h_x,x_2)\in\Heis^{1,1}$. Seeing $f$ as an element of $\cS(\Heis,\cG)$, we have
    \begin{multline*}
            \Op_{\hbar}^{\cG}\left(\sigma(\cdot,x_2,\cdot,D_{x_2})\right)f(h_x)(x_2)
            = \int_{\Heis}\int_{\Hhat} {\rm Tr}_{\cH_{\Tilde{\pi}^{\delta}}}\left(\Tilde{\pi}^{\delta}(\hbar^{-1}\cdot(h_{y}^{-1}h_{x}))\sigma(h_x,x_2,\Tilde{\pi}^{\delta},D_{x_2})f(h_y)(x_2)\right) \,d\mu_{\Hhat}(\Tilde{\pi}^{\delta}) dh_y\\
            = \int_{\Heis}\int_{\Hhat} {\rm Tr}_{\cH_{\Tilde{\pi}^{\delta}}}\left(\Tilde{\pi}^{\delta}(\hbar^{-1}\cdot(h_{y}^{-1}h_{x}))\int_{\R}\int_{\R}e^{i\beta(x_2-y_2)}\sigma(h_x,x_2,\check{\pi}^{\delta,\beta})f(h_y,y_2)\,\frac{1}{\sqrt{2\pi}}d\beta dy_2\right) \,d\mu_{\Hhat}(\Tilde{\pi}^{\delta}) dh_y\\
            = \int_{\Heis^{1,1}}\int_{\Hhat^{1,1}} {\rm Tr}_{\cH_{\check{\pi}^{\delta,\beta}}}\left(\check{\pi}^{\delta,\beta}(\hbar^{-1}\cdot_{pt}(h_{y}^{-1}h_{x},x_2-y_2))\sigma(h_x,x_2,\check{\pi}^{\delta,\beta})f(h_y,y_2)\right) \,d\mu_{\Hhat^{1,1}}(\check{\pi}^{\delta,\beta}) dh_y dy_2,
    \end{multline*}
    which means we have 
    \begin{equation*}
        \Op_{\hbar}^{\cG}\left(\sigma(\cdot,x_2,\cdot,D_{x_2})\right)f(h_x) (x_2) = \Op_{\hbar,\Heis^{1,1}}(\sigma)f(h_x,x_2).
    \end{equation*}
\end{proof}

\subsection{Second-microlocal semiclassical measures}
\label{subsect:secondmicromeas}

We associate to a bounded family $(\psi_{0}^\hbar)_{\hbar>0}$ in $L^2(\Eng)$ the quantities 
\begin{equation*}
    \ell_{\hbar}^2(\sigma) = \left(\Op_{\hbar}^{C_{\nu_0}}(\sigma)\psi_{0}^\hbar,\psi_{0}^\hbar\right)_{L^2(\Eng)},\ \sigma\in\cA_{0}^2(\Eng),
\end{equation*}
the limits of which are characterized by the following theorem.

\begin{theorem}
\label{thm:mesures2micro}
    Let $(\psi_{0}^{\hbar})_{\hbar>0}$ be a bounded family in $L^2(\Eng)$, and let $(\hbar_k)_{k\in\N}$ a sequence going to zero
    for which $\Gamma d\gamma$ is a semiclassical measure of $(\psi_{0}^{\hbar})_{\hbar>0}$. Then up to taking a subsequence of $(\hbar_k)_{k\in\N}$, there exist two
    positive operator-valued measures $\Gamma^{\infty} d\gamma^{\infty} \in \cM_{ov}^{+}\left(C_{\nu_0}\times\bS^0, (\cH_\pi)_{(x,\pi,\omega)\in C_{\nu_0}\times\bS^0}\right)$ and 
    $\Gamma^{2}d\gamma^{2} \in \cM_{ov}^{+}\left(\Heis\times\Hhat, (\cH_{\Tilde{\pi}}\otimes L^2(\R_{2}))_{(x,\Tilde{\pi})\in\Heis\times\Hhat}\right)$,
    such that, for every $\sigma \in \cA_{0}^{2}(\Eng)$, we have 
    \begin{multline*}
        \left(\Op_{\hbar_n}^{C_{\nu_0}}(\sigma)\psi_{0}^{\hbar_k},\psi_{0}^{\hbar_k}\right)\Tend{k}{+\infty}\int_{(\Eng\times\Enghat)\setminus C_{\nu_0}} {\rm Tr}_{\cH_{\pi^{\delta,\beta}}}\left(\sigma_{\infty}\left(x,\pi^{\delta,\beta},\frac{\beta-\nu_0\delta^{1/3}}{|\beta-\nu_0\delta^{1/3}|}\right) \Gamma(x,\pi^{\delta,\beta})\right) \,d\gamma(x,\pi^{\delta,\beta})\\
        + \int_{C_{\nu_0} \times\bS^0} {\rm Tr}_{\cH_{\pi^{\delta,\nu_{0}\delta^{1/3}}}}\left(\sigma_{\infty}(x,\pi^{\delta,\nu_0\delta^{1/3}},\omega)\Gamma^{\infty}(x,\pi^{\delta,\nu_0\delta^{1/3}},\omega)\right)\,d\gamma^{\infty}(x,\pi^{\delta,\nu_0\delta^{1/3}},\omega)\\
        + \int_{\Heis \times\Hhat} {\rm Tr}_{\cH_{\Tilde{\pi}^{\delta}}\otimes L^2(\R_2)} \left(\sigma(h_x,x_2,\pi^{\delta,\nu_0\delta^{1/3}},D_{x_2})\Gamma^2(h_x,\Tilde{\pi}^{\delta})\right) d\gamma^{2}(h_x,\Tilde{\pi}^{\delta}),
    \end{multline*}
    where for $(h_x,\Tilde{\pi}^{\delta})\in \Heis\times\Hhat$, $\sigma(h_x,x_2,\pi^{\delta,\nu_0\delta^{1/3}},D_{x_2})$ denotes the operator acting on $\cH_{\Tilde{\pi}^{\delta}} \otimes L^2(\R_2)$ obtained by quantizing the symbol 
    $(y,\eta)\mapsto \sigma(h_x,y,\pi^{\delta,\nu_0\delta^{1/3}},\eta)$.
    
    The two operator-valued measures $(\Gamma^{\infty}d\gamma^\infty, \Gamma^2 d\gamma^2)$ are called second-microlocal semiclassical measures above $C_{\nu_0}$.
    Given the sequence $(\hbar_k)_{k\in\N}$, the second-microlocal semiclassical measures are the only operator-valued measures for which the above convergence holds.
\end{theorem}

\begin{remark}
    As any symbol $\sigma$ in $\cA_{0}(\Eng)$ can be seen as a symbol in $\cA_{0}^2(\Eng)$ independent of the extra variable $\eta$ and as, in this case, its two quantizations coincide, i.e 
    $\Op_\hbar(\sigma)=\Op_{\hbar}^{C_{\nu_0}}(\sigma)$, we can see Theorem \ref{thm:mesures2micro} as a refinement of Theorem \ref{theo:measures} where the cone $C_{\nu_0}$ has been blown-up.
    In particular, any semiclassical measure of a family $(\psi_{0}^\hbar)_{\hbar>0}$ in $L^2(\Eng)$ can be seen, up to a further extraction of the subsequence defining it, to be above $C_{\nu_0}$ the marginal of 
    a second-microlocal semiclassical measure for second-microlocal symbols independent of $\eta$.
\end{remark}

\subsubsection{Proof of Theorem \ref{thm:mesures2micro}}
We now present the proof of the existence of second-microlocalized semiclassical measures.

\begin{lemma}
    \label{lem:sym2microoutofCnu}
    Let $\sigma \in \mathcal{A}_{0}^{2}(\Eng)$, let $\chi \in C^{\infty}(\R)$ a smooth positive compactly supported function and $\eps>0$, we consider the symbol defined for 
    $(x,\pi^{\delta,\beta},\eta) \in \Eng\times\Enghat\times\Bar{\R}$ by $\tau(x,\pi^{\delta,\beta},\eta) = \sigma(x,\pi^{\delta,\beta},\eta)\left(1-\chi\left(\frac{\beta-\nu_0\delta^{1/3}}{\eps}\right)\right)$:
    there exists $\hbar_0>0$ (depending on $\eps$) such that for $\hbar \in (0,\hbar_0)$ we have 
    \begin{equation*}
        \Op_{\hbar}^{C_{\nu_0}}(\tau) = \Op_{\hbar}\left(\tau_{\infty}\left(x,\pi^{\delta,\beta},\frac{\beta-\nu_0\delta^{1/3}}{|\beta-\nu_0\delta^{1/3}|}\right)\right).
    \end{equation*} 
\end{lemma}
    
\begin{proof}
   Let $R_0 >0$ as defined in Definition \ref{def:sym2micro} associated to $\sigma$. Recall that we have 
   \begin{equation*}
    \Op_{\hbar}^{C_{\nu_0}}(\tau)  =\Op_{\hbar}\left(\tau\left(x,\pi^{\delta,\beta},\frac{\beta-\nu_0 \delta^{1/3}}{\hbar}\right)\right),
   \end{equation*}
   and on the support of $\tau$, by definition we have $|\beta-\nu_0 \delta^{1/3}|> \eps$. Thus, for $\hbar$ small enough we have 
   \begin{equation*}
        \frac{|\beta-\nu_0 \delta^{1/3}|}{\hbar} > \eps \hbar^{-1} > R_0,
   \end{equation*}
   on the support of $\tau$, and we obtain 
   \begin{equation*}
    \Op_{\hbar}\left(\tau\left(x,\pi^{\delta,\beta},\frac{\beta-\nu_0 \delta^{1/3}}{\hbar}\right)\right) = \Op_{\hbar}\left(\tau_{\infty}\left(x,\pi^{\delta,\beta},\frac{\beta-\nu_0\delta^{1/3}}{|\beta-\nu_0\delta^{1/3}|}\right)\right).
   \end{equation*}
\end{proof}

\begin{proof}[Proof of Theorem \ref{thm:mesures2micro}]
    Let $\sigma \in \mathcal{A}_{0}^{2}(\Eng)$ a second-microlocal symbol and let $R_0 > 0$ such that 
    \begin{equation*}
        \forall (x,\pi^{\delta,\beta}) \in \Eng\times\Enghat_{gen},\forall \eta \in\R,\,|\eta|>R_0,\ \sigma(x,\pi^{\delta,\beta},\eta) = \sigma_{\infty}\left(x,\pi^{\delta,\beta},\frac{\eta}{|\eta|}\right).
    \end{equation*} 
    Let $\delta_0>0$ such that $\sigma(x,\cdot,\eta)$ is uniformly supported in $\cO_{\delta_0}(\Enghat)$. We fix $R>0$ and $\chi\in\mathcal{C}_{c}^{\infty}(\R,[0,1])$ such that $\chi=1$ in a neighborhood of $0$ and we set
    \begin{equation*}
        \begin{aligned}
            &\sigma^{R}(x,\pi^{\delta,\beta},\eta) := \sigma(x,\pi^{\delta,\beta},\eta)\left(1-\chi\left(\frac{\eta}{R}\right)\right),\\
            &\sigma_{R}(x,\pi^{\delta,\beta},\eta) := \sigma(x,\pi^{\delta,\beta},\eta)\chi\left(\frac{\eta}{R}\right).
        \end{aligned}
    \end{equation*}
    Of course we have
    \begin{equation*}
        \sigma = \sigma_{R} + \sigma^{R}.
    \end{equation*}
    The idea is to consider the limits $\limsup_{R\rightarrow +\infty}\limsup_{\hbar\rightarrow 0}$ of $\ell_{\hbar}^2(\sigma_R)$ and $\ell_{\hbar}^2(\sigma^R)$ respectively, taken in this order.
    
    First we consider the symbol $\sigma_{R}$. Introducing the bounded family $(\Tilde{\psi}_{0}^{\hbar})_{\hbar>0}$ defined by $\Tilde{\psi}_{0}^{\hbar} = \bU_{\nu_0}^{\hbar}\psi_{0}^{\hbar}$, by Theorem \ref{thm:sym2microclosetoCnu}, we can write
    \begin{equation*}
        \left(\Op_{\hbar}^{C_{\nu_0}}(\sigma_{R})\psi_{0}^{\hbar},\psi_{0}^{\hbar}\right) = \left(\Op_{\hbar,\Heis^{1,1}}(\sigma_{R}(x,\pi^{\delta,\nu_0\delta^{1/3}},\beta))\Tilde{\psi}_{0}^{\hbar},\Tilde{\psi}_{0}^{\hbar}\right) + \cO(\hbar).
    \end{equation*}
    Thanks to Proposition \ref{prop:correspondenceopHxR}, viewing the family $(\Tilde{\psi}_{0}^{\hbar})_{\hbar>0}$ as a bounded family $\Tilde{\Psi}_{0}^{\hbar}: h\in\Heis\mapsto \Tilde{\psi}_{0}^{\hbar}(h,\cdot)$ in $L^2(\Heis,\cG)$ where $\cG=L^2(\R_2)$, we have 
    \begin{equation*}
        \left(\Op_{\hbar}^{C_{\nu_0}}(\sigma_{R})\psi_{0}^{\hbar},\psi_{0}^{\hbar}\right) = \left(\Op_{\hbar}^{\cG}(\sigma_{R}(h,x_2,\pi^{\delta,\nu_0\delta^{1/3}},D_{x_2}))\Tilde{\Psi}_{0}^{\hbar},\Tilde{\Psi}_{0}^{\hbar}\right)_{L^2(\Heis,\cG)} + \cO(\hbar).
    \end{equation*}
    By Theorem \ref{thm:mesuresemiopvalued} in Appendix \ref{sect:Hilbertvaluedsemiclass}, up to a further extraction $(\hbar_k)_{k\in\N}$, there exists a semiclassical measure $\Gamma^{2}d\gamma^{2} \in \cM_{ov}^{+}\left(\Heis\times\Hhat, (\cH_{\Tilde{\pi}}\otimes \cG)_{(x,\Tilde{\pi})\in\Heis\times\Hhat}\right)$ 
    associated to the family $(\Tilde{\Psi}_{0}^{\hbar_k})_{k\in\N}$ such that 
    \begin{equation*}
        \begin{aligned}
        &\left(\Op_{\hbar_k}^{\cG}(\sigma_{R}(h,x_2,\pi^{\delta,\nu_0\delta^{1/3}},D_{x_2}))\Tilde{\Psi}_{0}^{\hbar_k},\Tilde{\Psi}_{0}^{\hbar_k}\right) \Tend{k}{+\infty}\\
        &\int_{\Heis\times\Hhat} {\rm Tr}_{\cH_{\Tilde{\pi}^\delta} \otimes \cG}\left(\sigma_{R}(h,x_2,\pi^{\delta,\nu_0\delta^{1/3}},D_{x_2}) \Gamma^{2}(h,\Tilde{\pi}^\delta)\right)\,d\gamma^{2}(h,\Tilde{\pi}^\delta).
        \end{aligned}
    \end{equation*}
    To conclude with this part of the proof, just observe that when $R$ goes to infinity, as the operator $\sigma_{R}(h,x_2,\pi^{\delta,\nu_0\delta^{1/3}},D_{x_2})$ converges in operator norm towards
    $\sigma(h,x_2,\pi^{\delta,\nu_0\delta^{1/3}},D_{x_2})$, the right-hand side term then converges to 
    \begin{equation*}
        \int_{\Heis\times\Hhat} {\rm Tr}_{\cH_{\Tilde{\pi}^\delta} \otimes \cG}\left(\sigma(h,x_2,\pi^{\delta,\nu_0\delta^{1/3}},D_{x_2}) \Gamma^{2}(h,\Tilde{\pi}^\delta)\right)\,d\gamma^{2}(h,\Tilde{\pi}^\delta).
    \end{equation*}

    We consider now the symbol $\sigma^R$ and we introduce the notation
    \begin{equation*}
        I_{\hbar,R}(\sigma) = \left(\Op_{\hbar}^{C_{\nu_0}}(\sigma^{R})\psi_{0}^{\hbar},\psi_{0}^{\hbar}\right).
    \end{equation*}
    We see that choosing $R > R_0$, we have $I_{\hbar,R}(\sigma) = I_{\hbar,R}(\sigma_{\infty})$: this way we consider $I_{\hbar,R}$ as a linear form on the space 
    $\cA_{0}^{2,\infty}(\Eng)$. But by Proposition \ref{prop:2microbounded}, 
    \begin{equation*}
        |I_{\hbar,R}(\sigma_\infty)| \leq N_2(\sigma^R) \|f^{\hbar}\|_{L^2(\Eng)},
    \end{equation*}
    thus $(I_{\hbar,R}(\sigma_\infty))_{\hbar>0}$ is a bounded sequence, and up to a furhter extraction $(\hbar_k)_{k\in\N}$, there exists a limit that we denote by $I_R(\sigma_\infty)$. 
    Using a countable dense subset of $\cA_{0}^{2,\infty}(\Eng)$ and by a diagonal argument, there exists a common extracted subsequence $(\hbar_k)_{k\in\N}$ for which we have the desired convergence. 
    From the previous inequality and from the computations done in the proof of Proposition \ref{prop:2microbounded}, we also obtain
    \begin{equation*}
        \begin{aligned}
        |I_{R}(\sigma_\infty)| &\leq C\left(1+\frac{1}{R}\right)\sup_{(x,\eta)\in\Eng\times\R}\left\|\sigma(x,\cdot,\eta)\left(1-\chi\left(\frac{\eta}{R}\right)\right)\right\|_{L^{\infty}(\Enghat)} \\
        &\leq C\left(1+\frac{1}{R}\right)\sup_{(x,\omega)\in\Eng\times\bS^0}\|\sigma_{\infty}(x,\cdot,\omega)\|_{L^{\infty}(\Enghat)},
        \end{aligned}
    \end{equation*}
    thus $I_R$ is a continuous linear form on the completion of $\cA_{0}^{2,\infty}(\Eng)$, denoted by $\cA^{2,\infty}(\Eng)$, for the norm $\|\cdot\|_{\cA^{2,\infty}}$ defined as 
    \begin{equation*}
        \forall \sigma_\infty \in \cA_{0}^{2,\infty}(\Eng),\ \|\sigma_\infty\|_{\cA^{2,\infty}} = \sup_{(x,\omega)\in\Eng\times\bS^0}\|\sigma_{\infty}(x,\cdot,\omega)\|_{L^{\infty}(\Enghat)}.
    \end{equation*}
    The linear forms on such a space are known as they are described by operator-valued measures $\Gamma^{R,\infty}d\gamma^{\R,\infty} \in \cM_{ov}\left(\Eng\times\Enghat\times\bS^0, (\cH_\pi)_{(x,\pi,\omega)\in\Eng\times\Enghat\times\bS^0}\right)$:
    \begin{equation*}
        \forall \sigma_\infty \in \cA_{0}^{2,\infty}(\Eng),\ I_R(\sigma_\infty) = \int_{\Eng\times\Enghat\times\bS^0} {\rm Tr}_{\cH_{\pi}}\left(\sigma_{\infty}(x,\pi,\omega)\Gamma^{R,\infty}(x,\pi,\omega)\right) d\gamma^{R,\infty}(x,\pi,\omega).
    \end{equation*}
    We are left with understanding the limit of the sequence $(I_R)_{R\in\R^+}$ as $R\rightarrow +\infty$. As a bounded sequence of linear forms on $\cA^{2,\infty}(\Eng)$, there exists a sequence $(R_m)_{m\in\N}$ and a linear form $I = \Gamma^{\infty}d\gamma^{\infty} \in \cA^{2,\infty}(\Eng)^* = \cM_{ov}\left(\Eng\times\Enghat\times\bS^0, (\cH_\pi)_{(x,\pi,\omega)\in\Eng\times\Enghat\times\bS^0}\right)$
    such that
    \begin{equation*}
        \forall \sigma_\infty \in \cA_{0}^{2,\infty}(\Eng),\ I_{R_m}(\sigma_\infty) \Tend{m}{+\infty} I(\sigma_\infty) = \int_{\Eng\times\Enghat\times\bS^0} {\rm Tr}_{\cH_{\pi}}\left(\sigma_{\infty}(x,\pi,\omega)\Gamma^{\infty}(x,\pi,\omega)\right) d\gamma^{\infty}(x,\pi,\omega).
    \end{equation*}
    We are left to show that $\Gamma^{\infty}d\gamma^{\infty}$ is positive, or equivalently defines a state on $\cA^{2,\infty}(\Eng)$. This property follows from the symbolic calculus: as operator on $L^2(\Eng)$, we have 
    \begin{equation*}
        \Op_{\hbar}^{C_{\nu_0}}(\sigma_{\infty}^{R,*}\sigma_{\infty}^{R}) = \Op_{\hbar}^{C_{\nu_0}}(\sigma_{\infty}^{R})^* \circ \Op_{\hbar}^{C_{\nu_0}}(\sigma_{\infty}^{R}) + \cO\left(\hbar + \frac{1}{R}\right),
    \end{equation*}
    and as $I(\sigma_\infty) = \lim_{m\rightarrow +\infty}\lim_{k\rightarrow +\infty} I_{\hbar_k,R_m}(\sigma_\infty)$, we obtain 
    \begin{equation*}
        \forall \sigma_\infty \in \cA_{0}^{2,\infty}(\Eng),\ I(\sigma_{\infty}^* \sigma_{\infty}) \geq 0.
    \end{equation*}
    
    To conclude our proof, we have to further breakdown the operator-valued measure $\Gamma^{\infty}d\gamma^{\infty}$. We introduce a small parameter $\eps>0$ and the folllowing decomposition of $\sigma^R$:
    \begin{equation*}
        \begin{aligned}
            &\sigma^{R}_{\eps}(x,\pi^{\delta,\beta},\eta) := \sigma^R(x,\pi^{\delta,\beta},\eta)\chi\left(\frac{\beta-\nu\delta^{1/3}}{\eps}\right),\\
            &\sigma^{R,\eps}(x,\pi^{\delta,\beta},\eta) := \sigma^R(x,\pi^{\delta,\beta},\eta)\left(1-\chi\left(\frac{\beta-\nu\delta^{1/3}}{\eps}\right)\right).
        \end{aligned}
    \end{equation*}
    By the dominated convergence theorem, we have
    \begin{equation*}
        I(\sigma_{\infty,\eps}) \Tend{\eps}{0} \int_{C_{\nu_0} \times\bS^0} {\rm Tr}_{\cH_{\pi^{\delta,\nu_0\delta^{1/3}}}}\left(\sigma_{\infty}(x,\pi^{\delta,\nu_0\delta^{1/3}},\omega)\Gamma^{\infty}(x,\pi^{\delta,\nu_0\delta^{1/3}},\omega)\right)\,d\gamma^{\infty}(x,\pi^{\delta,\nu_0\delta^{1/3}},\omega),
    \end{equation*}
    and we obtain the second term in Theorem \ref{thm:mesures2micro}.
    
    For the symbol $\sigma_{\infty}^{\eps}$, observe that by Lemma \ref{lem:sym2microoutofCnu}, we have for $\hbar$ small enough,
    \begin{equation*}
        \left(\Op_{\hbar}^{C_{\nu_0}}(\sigma^{R,\eps})\psi_{0}^{\hbar},\psi_{0}^{\hbar}\right)_{L^2(\Eng)} =
        \left(\Op_{\hbar}\left(\sigma^{\eps}_{\infty}\left(x,\pi^{\delta,\beta},\frac{\beta-\nu_0\delta^{1/3}}{|\beta-\nu_0\delta^{1/3}|}\right)\right)\psi_{0}^{\hbar},\psi_{0}^{\hbar}\right)_{L^2(\Eng)}.
    \end{equation*}
    By definition of the semiclassical measure $\Gamma d\gamma \in \mathcal{M}_{ov}^{+}(\Eng\times\Enghat)$ associated to $(\psi_{0}^{\hbar})_{\hbar>0}$, we can write 
    \begin{equation*}
        \left(\Op_{\hbar_k}^{C_{\nu_0}}(\sigma^{R,\eps})\psi_{0}^{\hbar_k},\psi_{0}^{\hbar_k}\right) \Tend{\hbar_k}{0} 
        \int_{\Eng\times\Enghat} {\rm Tr}_{\cH_{\pi^{\delta,\beta}}}\left(\sigma^{\eps}_{\infty}\left(x,\pi^{\delta,\beta},\frac{\beta-\nu_0\delta^{1/3}}{|\beta-\nu_0\delta^{1/3}|}\right)\Gamma(x,\pi^{\delta,\beta})\right) \,d\gamma(x,\pi^{\delta,\beta}).
    \end{equation*}
    Letting $\eps$ go to $0$, we have decomposed our measure $\Gamma^{\infty}d\gamma^{\infty}$ into two contributions, one supported on the subset $C_{\nu_0} \times \bS^0$ and the other as the semiclassical measure $\Gamma d\gamma$
    restricted to the subset $(\Eng\times\Enghat)\setminus C_{\nu_0}$.
\end{proof}

\subsubsection{Example: computation of a second-microlocal semiclassical measure}
\label{subsubsec:example2micro}

We propose here an example of computation of second-microlocal semiclassical measure. Let fix $\nu_0 \in\R$ and $(h_0,\delta_0,\eta_0)\in\Heis\times\R\setminus\{0\}\times\R$, then for $\hbar>0$ we define the profile
\begin{equation*}
    \forall x= (h,x_2)\in\Eng,\ \alpha_\hbar(x) =  \hbar^{-3/2}a(\hbar^{-1/2}\cdot (h_{0}^{-1}h)) \varphi(x_2),
\end{equation*}
where we have fixed $a \in C_{c}^\infty(\Heis)$, $\varphi\in\cS(\R_2)$. Now, choosing a vector $\Phi$ in the Hilbert space associated to the representation $\pi^{\delta,\nu_0\delta^{1/3}}$, we introduce 
the following bounded family in $L^2(\Eng)$:
\begin{equation}
    \label{eq:example2micro}
    \forall x\in\Eng,\ u^{\hbar}(x) = \alpha_\hbar(x) \left(\pi^{\delta_0,\nu_0 \delta_{0}^{1/3} + \hbar\eta_0}(\hbar^{-1}\cdot (x_{0}^{-1}x)) \Phi,\Phi\right),
\end{equation}
where $x_0 = (h_0,0)\in\Eng$. 
This construction is inspired from the wave packets already introduced in Section \ref{subsect:wpevolution}.

\begin{proposition}
    \label{prop:example2micro}
    The family $(u^\hbar)_{\hbar>0}$ is a bounded family of $L^2(\Eng)$ and has unique semiclassical measure $\Gamma d\gamma$ and second-microlocal semiclassical measures $(\Gamma^{\infty}d\gamma^\infty,\Gamma^2 d\gamma^2)$ satisfying:
    \begin{equation*}
        \bm{1}_{(\Eng\times\Enghat)\setminus C_{\nu_0}} \Gamma d\gamma = 0,\quad \Gamma^\infty d\gamma^\infty = 0,
    \end{equation*}
    and
    \begin{equation*}
        d\gamma^2 = \delta_{\{h = h_0\}} \otimes \delta_{\{\pi^\delta = \pi^\delta_0\}},\quad \Gamma^2 = \ket{\Phi}\bra{\Phi}\otimes \ket{\varphi e^{i\eta_0 x_2}}\bra{\varphi e^{i\eta_0 x_2}} \in \cL^{1}\left(\cH_{\pi^{\delta,\nu_0\delta^{1/3}}}\otimes L^2(\R_2)\right).
    \end{equation*}
    From this, we recover the full semiclassical measure $\Gamma d\gamma$:
    \begin{equation*}
        d\gamma = \delta_{\{h = h_0\}}\otimes |\varphi(x_2)|^2 dx_2 \otimes \delta_{\{\pi^{\delta,\beta} = \pi^{\delta_0,\nu_0\delta^{1/3}}\}},\quad \Gamma = \ket{\Phi}\bra{\Phi} \in\cL^{1}\left(\cH_{\pi^{\delta,\nu_0\delta^{1/3}}}\right).
    \end{equation*}
\end{proposition}

We compute the action of a pseudodifferential operator on the family $(u^\hbar)_{\hbar>0}$. Such a comutation can already been found in \cite[Lemma 4.3]{FL}.

\begin{lemma}
    \label{lem:actionpseudowp}
    Let $\sigma\in S_{0,0}^0(\Eng)$, we have for all $x\in\Eng$
    \begin{equation*}
        \Op_{\hbar}(\sigma)u^{\hbar}(x) = \alpha_{\hbar}(x) \left(\pi^{\delta_0,\nu_0 \delta_{0}^{1/3} + \hbar\eta_0}(\hbar^{-1}\cdot (x_{0}^{-1}x)) \sigma(x,\pi^{\delta_0,\nu_0\delta_{0}^{1/3} + \hbar\eta_0})\Phi,\Phi\right) + \cO_{L^2(\Eng)}(\sqrt{\hbar}).
    \end{equation*}
\end{lemma}

\begin{proof}
    We denote by $\kappa$ the convolution kernel associated to $\sigma$. We have
    \begin{equation*}
        \begin{aligned}
        \Op_{\hbar}(\sigma)u^{\hbar}(x) &= u^\hbar * \kappa^{\hbar}_x (x) = \int_{\Eng} u^\hbar(xy^{-1})\kappa^{\hbar}_x (y)\,dy\\
        &= \hbar^{-7}\int_{\Eng} u^\hbar(xy^{-1})\kappa_x (\hbar^{-1}\cdot y)\,dy\\
        &= \int_{\Eng} \alpha_{\hbar}(x (\hbar\cdot y^{-1})) \left(\pi^{\delta_0,\nu_0 \delta_{0}^{1/3} + \hbar\eta_0}(\hbar^{-1}\cdot (x_{0}^{-1}x))\pi^{\delta_0,\nu_0 \delta_{0}^{1/3} + \hbar\eta_0}(y^{-1})\Phi,\Phi\right)\kappa_x(y)\,dy\\
        &= \int_{\Eng} \alpha_{\hbar}(x (\hbar\cdot y^{-1})) \left(\pi^{\delta_0,\nu_0 \delta_{0}^{1/3} + \hbar\eta_0}(\hbar^{-1}\cdot (x_{0}^{-1}x)) \kappa_x(y)\pi^{\delta_0,\nu_0 \delta_{0}^{1/3} + \hbar\eta_0}(y)^{*}\Phi,\Phi\right)\,dy.
        \end{aligned}
    \end{equation*}
    We would recognize in the last equation the Fourier transform of $\kappa_x$, and thus the symbol $\sigma$, evaluated on the representation $\pi^{\delta_0,\nu_0 \delta_{0}^{1/3} + \hbar\eta_0}$ if the profile $\alpha_\hbar$ had no dependency on the variable $y$.
    To this end, we perform a Taylor expansion, and by the defintion of the profile $\alpha_\hbar$ we obtain a remainder of order $\cO(\sqrt{\hbar})$:
    \begin{equation*}
        \Op_{\hbar}(\sigma)u^{\hbar}(x) = \alpha_{\hbar}(x) \left(\pi^{\delta_0,\nu_0 \delta_{0}^{1/3} + \hbar\eta_0}(\hbar^{-1}\cdot (x_{0}^{-1}x)) \int_{\Eng}\kappa_x(y)\pi^{\delta_0,\nu_0 \delta_{0}^{1/3} + \hbar\eta_0}(y)^{*}\,dy\, \Phi,\Phi\right) + \cO(\sqrt{\hbar}),
    \end{equation*}
    and we conclude.
\end{proof}

\begin{remark}
    As for the wave packets of Section \ref{subsect:wpevolution}, the previous proof can give rise to a full asymptotic expansion of the action of pseudodifferential operators on our bounded family. Moreover, such a result stays true
    on other nilpotent Lie groups, with a similar proof.
\end{remark}

\begin{proof}[Proof of Proposition \ref{prop:example2micro}]
    First we prove that the semiclassical measures of $(u^\hbar)_{\hbar>0}$ are necessarily concentrated on $C_{\nu_0}$. For $\sigma\in\cA_{0}(\Eng)$,
    consider, for $\eps>0$, we define the symbol $\sigma^{\eps}$ as follows:
    \begin{equation*}
        \forall (x,\pi^{\delta,\beta})\in\Eng\times\Enghat,\ \sigma^{\eps}(x,\pi^{\delta,\beta}) = \sigma(x,\pi^{\delta,\beta})\left(1-\chi\left(\frac{\beta-\nu_0 \delta^{1/3}}{\eps}\right)\right).
    \end{equation*}
    Then by Lemma \ref{lem:actionpseudowp}, we have for $x\in\Eng$
    \begin{equation*}
        \Op_{\hbar}(\sigma^\eps)u^{\hbar}(x) = \alpha_{\hbar}(x) \left(\pi^{\delta_0,\nu_0 \delta_{0}^{1/3} + \hbar\eta_0}(\hbar^{-1}\cdot (x_{0}^{-1}x)) \sigma^{\eps}(x,\pi^{\delta_0,\nu_0 \delta_{0}^{1/3} + \hbar\eta_0}) \Phi,\Phi\right) + \cO(\sqrt{\hbar}),
    \end{equation*}
    and easily observe that 
    \begin{equation*}
        \sigma^\eps(x,\pi^{\delta_0,\nu_0 \delta_{0}^{1/3} + \hbar\eta_0}) = \cO(\hbar),
    \end{equation*}
    for $\hbar$ small engouh as $\sigma^{\eps}$ vanishes on a neighborhood of $C_{\nu_0}$.
    Thus, any semiclassical measure $\Gamma d\gamma$ satisfies:
    \begin{equation*}
        \int_{\Eng\times\Enghat} {\rm Tr}\left(\sigma^\eps(x,\pi^{\delta,\beta})\Gamma(x,\pi^{\delta,\beta})\right)d\gamma(x,\pi^{\delta,\beta}) = 0,
    \end{equation*}
    and by dominating convergence, letting $\eps \rightarrow 0$, we get $\bm{1}_{(\Eng\times\Enghat)\setminus C_{\nu_0}} \Gamma d\gamma = 0$.
    
    For the part at infinity $\Gamma^\infty d\gamma^\infty$ of any second-microlocal semiclassical measures, recall that for $\sigma \in\cA_{0}^2(\Eng)$, we have
    \begin{equation*}
        \int_{C_{\nu_0}\times\bS^0} {\rm Tr}\left(\sigma(x,\pi,\omega)\Gamma(x,\pi,\omega)\right)d\gamma(x,\pi^{\delta,\beta},\omega) = \lim_{\eps\rightarrow 0}\lim_{R\rightarrow +\infty}\lim_{\hbar\rightarrow 0} \left(\Op_{\hbar}^{C_{\nu_0}}(\sigma_{\eps}^{R})u^\hbar,u^\hbar\right),
    \end{equation*}
    where for all $(x,\pi^{\delta,\beta},\eta)\in\Eng\times\Enghat\times\R$,
    \begin{equation*}
        \sigma_{\eps}^R(x,\pi^{\delta,\beta},\eta) = \sigma(x,\pi^{\delta,\beta},\eta)\left(1-\chi\left(\frac{\eta}{R}\right)\right)\chi\left(\frac{\beta-\nu_0 \delta^{1/3}}{\eps}\right).
    \end{equation*}
    Using Lemma \ref{lem:actionpseudowp} with symbol $\left(\sigma_{\eps}^R\right)_{\hbar}^{C_{\nu_0}}$, we need to understand the action on the vector $\Phi$ of the following operator
    \begin{equation*}
        \left(\sigma_{\eps}^R\right)_{\hbar}^{C_{\nu_0}}(x,\pi^{\delta_0,\nu_0 \delta_{0}^{1/3}+\hbar\eta_0}) = \sigma(x,\pi^{\delta_0,\nu_0 \delta_{0}^{1/3}+\hbar\eta_0},\eta_0)\left(1-\chi\left(\frac{\eta_0}{R}\right)\right)\chi\left(\frac{\hbar \eta_0}{\eps}\right).
    \end{equation*}
    This way, when $|R|>|\eta_0|$, we get that 
    \begin{equation*}
        \Op_{\hbar}^{C_{\nu_0}}(\sigma_{\eps}^{R})u^{\hbar} = \cO_{L^2(\Eng)}(\sqrt{\hbar}),
    \end{equation*}
    giving at the limit $\Gamma^{\infty}d\gamma^{\infty} = 0$.

    Finally for the remaining part $\Gamma^2 d\gamma^2$, recall that this part of this measure can be computed as the operator-valued semiclassical measure of $\Tilde{u}_\hbar = \bU_{\nu_0}^{\hbar} u^{\hbar}$ in $L^2(\Heis^{1,1})$ seen as $L^2(\Heis,L^2(\R_2))$.
    We claim that we have, for $(h,x_2)\in\Heis^{1,1}$,
    \begin{equation}
        \label{eq:calculUwp}
        \Tilde{u}^\hbar(h,x_2) = \hbar^{-3/2} a(\hbar^{-1/2}\cdot (h_{0}^{-1}h)) \varphi(x_2) e^{i\eta_0 x_2}\left(\Tilde{\pi}^{\delta_0}(\hbar^{-1}\cdot (h_{0}^{-1}h)) \Phi,\Phi\right) + \cO_{L^2(\Heis^{1,1})}(\sqrt{\hbar}).
    \end{equation}
    We denote the right-hand side term (without the remainder $\cO(\sqrt{\hbar})$) by $\check{u}^{\hbar}$. If Equation \eqref{eq:calculUwp} holds true, then the computation of any operator-valued measure for $(\check{u}^\hbar)_{\hbar>0}$ is straightforward as we identify a wave packet on the Heisenberg variables, 
    while on the real variable $x_2$, we identify the profile $\varphi e^{i\eta_0 x_2}$. Indeed, the semiclassical measures of wave packets on $\Heis$ being simply the Dirac measures on the point of concentration $(h_0,\Tilde{\pi}^{\delta_0})$ (for proofs see for example \cite{FF2}), we obtain the announced result.
    
    We are left with proving \eqref{eq:calculUwp}: observe that for $(h_x,x_2)\in\Eng$, we have 
    \begin{equation*}
        \bU_{\nu_0}^{\hbar}\check{u}^\hbar(h_x,x_2) = \lim_{\eps\rightarrow 0} \hbar^{-6}\int_{\Heis\times\Hhat}{\rm Tr}\left(\Tilde{\pi}^{\delta}(\hbar^{-1}\cdot (h_{y}^{-1}h_x)) U_{\nu_0}^{\hbar}(x_2,\Tilde{\pi}^\delta)\cF_{\Heis}^{\hbar}\theta_\eps(\Tilde{\pi}^\delta)\right) \check{u}^\hbar(h_y, x_2)\,d\mu_{\Hhat}(\Tilde{\pi}^\delta) dh_y,
    \end{equation*}
    where $\theta_\eps$ is an approximation of the identity in $\cS(\Heis)$. From the structure of $\check{u}^\hbar$, the previous operator only acts on the Heisenberg variables: focusing on these variables, we have 
    \begin{multline*}
        \hbar^{-6}\int_{\Heis\times\Hhat}{\rm Tr}\left(\Tilde{\pi}^{\delta}(\hbar^{-1}\cdot (h_{y}^{-1}h_x)) U_{\nu_0}^{\hbar}(x_2,\Tilde{\pi}^\delta)\cF_{\Heis}^{\hbar}\theta_\eps(\Tilde{\pi}^\delta)\right) WP_{(h_0,\Tilde{\pi}^{\delta_0})}^{\hbar}(a,\Phi,\Phi) (h_y)\,d\mu_{\Hhat}(\Tilde{\pi}^\delta) dh_y\\
        = \Op_{\hbar,\Heis}\left(U_{\nu_0}^{\hbar}(x_2,\cdot)\cF_{\Heis}^{\hbar}\theta_\eps(\cdot)\right) WP_{(h_0,\Tilde{\pi}^{\delta_0})}^{\hbar}(a,\Phi,\Phi) (h_x),
    \end{multline*}
    where we write $WP_{(h_0,\Tilde{\pi}^{\delta_0})}^{\hbar}(a,\Phi,\Phi)(h_x) = \hbar^{-3/2}a(\hbar^{-1/2}\cdot(h_{0}^{-1}h_x))\left(\Tilde{\pi}^{\delta_0}(\hbar^{-1}\cdot(h_{0}^{-1}h_x))\Phi,\Phi\right)$. By an entirely similar proof as in Lemma \ref{lem:actionpseudowp}, we have 
    \begin{multline*}
        \Op_{\hbar,\Heis}\left(U_{\nu_0}^{\hbar}(x_2,\cdot)\cF_{\Heis}^{\hbar}\theta_\eps(\cdot)\right) WP_{(h_0,\Tilde{\pi}^{\delta_0})}^{\hbar}(a,\Phi,\Phi) (h_x) \\
        = \hbar^{-3/2}a(\hbar^{-1/2}\cdot(h_{0}^{-1}h_x))\left(\Tilde{\pi}^{\delta_0}(\hbar^{-1}\cdot(h_{0}^{-1}h_x))U_{\nu_0}^{\hbar}(x_2,\Tilde{\pi}^{\delta_0})\cF_{\Heis}^{\hbar}\theta_\eps(\Tilde{\pi}^{\delta_0})\Phi,\Phi\right) + \cO(\sqrt{\hbar}),
    \end{multline*}
    and in the right-hand side we recognize $\pi^{\delta_0,\nu_0\delta_{0}^{1/3}}(\hbar^{-1}\cdot(x_{0}^{-1}x))\cF_{\Heis}^{\hbar}\theta_\eps(\Tilde{\pi}^{\delta_0})$ acting on the vector $\Phi$. Remark that the remainder does not depend on $\eps$ and thus we can take the limit $\eps \rightarrow 0$ and our claim is proved.
\end{proof}

\section{Schrödinger equation and second-microlocal semiclassical measures}
\label{sect:sch2micro}

This section is devoted to the semiclassical analysis of the following Schrödinger equation 
\begin{equation*}
    \begin{cases}
        & i\partial_{t}\psi^{\hbar} = -\Delta_{\Eng}\psi^{\hbar},\\
        & \psi_{|t=0}^{\hbar} = \psi_{0}^{\hbar}\in L^2(\Eng),
    \end{cases}
\end{equation*}
where $(\psi_{0}^\hbar)_{\hbar>0}$ is a bounded family in $L^2(\Eng)$. More precisely we wish to apply the second-microlocal analysis developed in the previous section to analyze in further details than in 
Proposition \ref{prop:dispersionsupp} the support of the time-averaged semiclassical measures associated to bounded families of solutions to the previous Schrödinger equation.

\subsection{Time-averaged second-microlocal semiclassical measures and their propagation}

For a bounded family $(\psi^\hbar(t))_{\hbar>0}$ in $L^\infty(\R_t, L^2(\Eng))$, we associate the quantities 
\begin{equation}
    \label{def:leps2}
    \ell_{\hbar}^2(\theta, \sigma)=\int_{\R} \theta(t)  \left(\Op_{\hbar}^{C_{\nu_0}} (\sigma) \psi^\hbar(t),\psi^\hbar(t)\right)_{L^2(\Eng)} dt,\;\;\sigma\in \cA_{0}^2(\Eng),\;\;\theta\in L^1(\R_t),
\end{equation}
the limits of which are characterized in the following proposition, which generalizes Theorem \ref{theo:timemeasures}.

\begin{proposition}
    \label{prop:mesures2microtimeaveraged}
    Let $(\psi^\hbar(t))_{\hbar>0}$ be a bounded family in $L^\infty(\R_t,L^2(\Eng))$ and a sequence $(\hbar_k)_{k\in \N}$ in $(0,+\infty)$ with  $\hbar_k\Tend{k}{+\infty}0$
    for which we have a time-averaged semiclassical measure $t\mapsto \Gamma_t d\gamma_t $ in
    $L^\infty(\R_t, {\mathcal M}_{ov}^+(\Eng\times \Enghat))$. Then up to an extraction of $(\hbar_k)_{k\in \N}$, there exist two maps, called the \textit{time-averaged second-microlocal semiclassical measures},
    \begin{equation*}
        \begin{aligned}
        &t\mapsto \Gamma_{t}^\infty d\gamma_{t}^\infty\in L^\infty\left(\R_t, \cM_{ov}^{+}\left(C_{\nu_0}\times\bS^0, (\cH_\pi)_{(x,\pi,\omega)\in C_{\nu_0}\times\bS^0}\right)\right),\\
        &t\mapsto \Gamma_{t}^2 d\gamma_{t}^2\in L^\infty\left(\R_t, \cM_{ov}^{+}\left(\Heis\times\Hhat, (\cH_{\Tilde{\pi}}\otimes L^2(\R_{2}))_{(x,\Tilde{\pi})\in\Heis\times\Hhat}\right)\right),
        \end{aligned}
    \end{equation*} 
    such that for all $ \theta\in L^1(\R_t)$ and $\sigma\in \cA_{0}^2(\Eng)$, we have
    \begin{equation*}
        \begin{aligned}
        \ell_{\hbar}^2(\theta,\sigma)
        &\Tend {k}{+\infty} 
        \int_{\R\times \Eng\times \Enghat} \theta(t){\rm Tr}\left(\sigma(x,\pi) \Gamma_t(x,\pi)\right)d\gamma_t(x,\pi) dt\\
        &+ \int_{\R\times C_{\nu_0} \times\bS^0} \theta(t){\rm Tr}_{\cH_{\pi^{\delta,\nu_{0}\delta^{1/3}}}}\left(\sigma_{\infty}(x,\pi^{\delta,\nu_0\delta^{1/3}},\omega)\Gamma_{t}^{\infty}(x,\pi^{\delta,\nu_0\delta^{1/3}},\omega)\right)\,d\gamma_{t}^{\infty}(x,\pi^{\delta,\nu_0\delta^{1/3}},\omega)dt\\
        &+ \int_{\R\times\Heis \times\Hhat} \theta(t){\rm Tr}_{\cH_{\Tilde{\pi}^{\delta}}\otimes L^2(\R_2)} \left(\sigma(h_x,x_2,\pi^{\delta,\nu_0\delta^{1/3}},D_{x_2})\Gamma_{t}^2(h_x,\Tilde{\pi}^{\delta})\right) d\gamma_{t}^{2}(h_x,\Tilde{\pi}^{\delta})dt.
        \end{aligned}
    \end{equation*}
    Given the sequence $(\hbar_k)_{k\in\N}$, these measures are the only operator-valued measures for which the above convergence holds.
\end{proposition}

This proposition can be proved with the same arguments as in the proof of Theorem \ref{thm:mesures2micro} and in \cite{FF3}[Section 3.2]. The following results generalize 
Theorem \ref{thm:propag2microintro} presented in Section \ref{subsect:secondmircointro} and show that our second-microlocal analysis allow to uncover the propagation laws underlying the Schrödinger equation.

\begin{theorem}
\label{thm:propagmesures2micro}
Let  $(\psi_0^\hbar)_{\hbar>0}$ be a bounded family in $L^2(\Eng)$ 
and $\psi^\hbar(t) = {\rm e}^{i t \Delta_{\Eng}} \psi^\hbar_0$ be the solution to  \eqref{eq:SchEngel} for $\tau = 2$. Let $\nu_0\in\R$,
then any time-averaged second-microlocal semiclassical measures $t\mapsto (\Gamma_{t}^{\infty}d\gamma_{t}^{\infty}, \Gamma_{t}^2 d\gamma_{t}^2)$ 
as in Proposition \ref{prop:mesures2microtimeaveraged} for the family $(\psi^\hbar(t))_{\hbar>0}$, 
satisfies the following additional properties:
\begin{enumerate}
\item[(i)] For $d\gamma_{t}^2 dt$-almost all $(t,h,\Tilde{\pi}^\delta)\in \R_t\times\Heis\times\Hhat$, the operator $\Gamma_{t}^{2}(h,\Tilde{\pi}^\delta)$ commutes with $H(\pi^{\delta,\nu_0 \delta^{1/3}})$
and admits the following decomposition
\begin{equation}
\Gamma_{t}^2(h,\Tilde{\pi}^\delta)=\sum_{n\in\N } \Gamma_{n,t}^2(h,\Tilde{\pi}^\delta),
\end{equation}
where
\begin{equation*}
    \Gamma_{n,t}(h,\Tilde{\pi}^\delta):= \left(\Pi_n(\pi^{\delta,\nu_0 \delta^{1/3}})\otimes {\rm Id}_{L^2(\R_2)}\right)\Gamma_{t}^2(h,\Tilde{\pi}^\delta) \left(\Pi_n(\pi^{\delta,\nu_0 \delta^{1/3}})\otimes {\rm Id}_{L^2(\R_2)}\right)\in\cL^{1}(L^2(\R_2)).
\end{equation*}
\item[(ii)] Similarly, for $d\gamma_{t}^\infty dt$-almost all $(t,(x,\pi),\omega)\in \R_t\times C_{\nu_0}\times\bS^0$, $\Gamma_{t}^{\infty}(x,\pi,\omega)$ commutes with $H(\pi^{\delta,\nu_0\delta^{1/3}})$ and decomposes as follow
\begin{equation}
    \Gamma_{t}^\infty(x,\pi^{\delta,\nu_0\delta^{1/3}},\omega)=\sum_{n\in\N } \Gamma_{n,t}^\infty(x,\pi^{\delta,\nu_0\delta^{1/3}},\omega),
\end{equation}
with
\begin{equation*}
    \Gamma_{n,t}^\infty(x,\pi^{\delta,\nu_0\delta^{1/3}},\omega):= \Pi_n(\pi^{\delta,\nu_0 \delta^{1/3}})\Gamma_{t}^\infty(x,\pi^{\delta,\nu_0\delta^{1/3}},\omega) \Pi_n(\pi^{\delta,\nu_0 \delta^{1/3}}).
\end{equation*}
\suspend{enumerate}
    Let $m\in\N_{>0}$ and suppose that $\nu_0$ is a non-degenerate critical point of the $m$-th curve of dispersion $\nu\mapsto\Tilde{\mu}(\nu)$ of the Montgomery operators $(\Tilde{H}(\nu))_{\nu\in\R}$.
\resume{enumerate}
\item[(iii)] For almost all time $t\in\R_t$, the measure $d\gamma_{m,t}^{\infty} = \Gamma_{m,t}^\infty d\gamma_{t}^\infty$ is zero: 
\begin{equation*}
    d\gamma_{m,t}^{\infty} = 0.
\end{equation*}
\item[(iv)] The $m$-th mode $d\gamma_{m,t}^2= \Gamma_{m,t}^{2}d\gamma_{t}^2$ satisfies the following Heisenberg equation in the weak sense:
\begin{equation}
    i\partial_t \gamma_{m,t}^2 = \left[-\frac{\Tilde{\mu}_{m}''(\nu_0)}{2}\partial_{x_2}^2,\gamma_{m,t}^2\right],
\end{equation}
where the operator $\partial_{x_2}$ is seen as an operator on $L^2(\R_2)$.
\end{enumerate}
\end{theorem}

\begin{remark}
    The equation satisfied by $d\gamma_{m,t}^2$ must be understood in the following sense: for $\sigma\in\cA_{0}^{L^2(\R_{2})}(\Heis)$, we have
    \begin{multline*}
        i\frac{d}{dt}\int_{\Heis\times\Hhat} {\rm Tr}_{L^2(\R_{2})}\left(\sigma(x,\Tilde{\pi}^\delta)\Gamma_{m,t}^2(x,\Tilde{\pi}^\delta)\right)d\gamma_{t}^{2}(x,\Tilde{\pi}^\delta) = \\
        \int_{\Heis\times\Hhat} {\rm Tr}_{L^2(\R_{2})}\left(\left[\sigma(x,\Tilde{\pi}^\delta),\frac{\Tilde{\mu}_{m}''(\nu_0)}{2}\partial_{x_2}^2\right]\Gamma_{m,t}^2(x,\Tilde{\pi}^\delta)\right)d\gamma_{t}^{2}(x,\Tilde{\pi}^\delta).
    \end{multline*}
\end{remark}

In order to be able to integrate these propagations laws from initial conditions corresponding to second-microlocal semiclassical measures of the family of initial datum $(\psi_{0}^\hbar)_{\hbar>0}$, we need 
a result concerning the regularity of these measures with respect to time.

\begin{proposition}
    \label{prop:continuityGamma2}
    Consider, as in Theorem \ref{thm:propagmesures2micro}, the $m$-th mode $t\mapsto d\gamma_{m,t}^2 = \Gamma_{m,t}^2 d\gamma_{t}^2$ of a time-averaged second-microlocal 
    semiclassical measures corresponding to the family of solutions to the Schrödinger equation \eqref{eq:SchEngel} for $\tau=2$. Then for any symbol $\sigma\in \cA_{0}^{L^2(\R_2)}(\Heis)$ commuting
    with the field of projectors $\Tilde{\pi}^{\delta}\in\Hhat\mapsto \Pi_{m}(\pi^{\delta,\nu_0 \delta^{1/3}})\otimes {\rm Id}_{L^2(\R_2)}$, the map
    \begin{equation*}
        t\in\R \mapsto \int_{\Heis\times \Hhat}{\rm Tr}\left(\sigma(x,\Tilde{\pi})\Gamma_{m,t}^2(x,\Tilde{\pi}) \right)d\gamma_{t}^2(x,\Tilde{\pi}),
    \end{equation*}
    is locally Lipshitz on $\R$.
\end{proposition}

\begin{remark}
    The proof of Proposition \ref{prop:continuityGamma2} implies that, under the assumptions of the proposition, if $(\hbar_k)_{k\in\N}$ is the subsequence realizing
    the time-averaged semiclassical measure $t\mapsto \Gamma_{t}^2 d\gamma_{t}^2$, then we have for all $t\in\R$ that $d\gamma_{m,t}^2$ is the $m$-th mode of the second-microlocal measure 
    of the family $(\psi^{\hbar_k}(t))_{k\in\N}$, meaning we can pass to the limit $t$-by-$t$ and not only when averaged in time.
\end{remark}

\subsection{Proof of Theorem \ref{thm:propagmesures2micro} and Proposition \ref{prop:continuityGamma2}}

Let  $(\psi_0^\hbar)_{\hbar>0}$ be a bounded family in $L^2(\Eng)$ and $\psi^\hbar(t) = {\rm e}^{i t \Delta_{\Eng}} \psi^\hbar_0$ be the solution to \eqref{eq:SchEngel} for $\tau = 2$. Let $\nu_0\in\R$,
and consider a time-averaged second-microlocal semiclassical measures $t\mapsto (\Gamma_{t}^{\infty}d\gamma_{t}^{\infty}, \Gamma_{t}^2 d\gamma_{t}^2)$ for the family $(\psi^\hbar(t))_{\hbar>0}$ for the sequence $(\hbar_k)_{k\in\N}$.

To begin with, it follows from the commutator formula \eqref{eq:commutatorDelta} that we have for all $\theta \in C_{c}^{\infty}(\R_t)$ and $\sigma \in \cA_{0}^2(\Eng)$:
\begin{equation}
    \label{eq:ell2}
    -i\hbar^2 \ell_{\hbar}^{2}(\theta',\sigma) = \ell_{\hbar}^{2}(\theta,\left[\sigma,H\right]) + 2\hbar\ell_{\hbar}^{2}(\theta,(\pi(V)\cdot V)\sigma) + \hbar^2 \ell_{\hbar}^{2}(\theta,\Delta_{\Eng}\sigma).
\end{equation}

\subsubsection{Propagation of $\Gamma_{t}^2 d\gamma_{t}^2$}
We focus on the case where our symbol $\sigma\in \cA_{0}^2(\Eng)$ has compact support on its last variable $\eta\in\R$.
As we have seen in the proof of Theorem \ref{thm:mesures2micro}, such a second-microlocal symbol only see in the limit the operator-valued measures $\Gamma_{t}^2 d\gamma_{t}^2$, in the sense that:
\begin{equation*}
    \ell_{\hbar_k}^2(\theta,\sigma) \Tend{k}{+\infty} \int_{\R\times\Heis\times\Hhat} \theta(t){\rm Tr}_{\cH_{\Tilde{\pi}^\delta} \otimes L^2(\R_2)} \left(\sigma(h,x_2,\pi^{\delta,\nu_0 \delta^{1/3}},D_{x_2}) \Gamma^{2}_t(h,\Tilde{\pi}^\delta)\right) \,d\gamma^{2}_t (h,\Tilde{\pi}^\delta)dt.
\end{equation*}
We denote by $\ell_{\infty}^{2,c}(\theta,\sigma)$ the right-hand side term. Using Equation \eqref{eq:ell2}, we readily obtain that 
\begin{equation*}
    \forall \sigma \in \cA_{0}^2(\Eng),\ \ell_{\hbar}^2 \left(\theta,\left[\sigma(x,\pi^{\delta,\beta},\eta),H(\pi^{\delta,\beta})\right]\right) = \cO(\hbar),
\end{equation*}
thus we obtain that for a.e $t\in \R$ and for all $\sigma \in \cA_{0}^2(\Eng)$,
\begin{equation*}
    \int_{\Heis\times\Hhat}{\rm Tr}_{\cH_{\Tilde{\pi}^\delta} \otimes L^2(\R)}\left(\left[\sigma(h,x_2,\pi^{\delta,\nu_0 \delta^{1/3}},D_{x_2}),H(\pi^{\delta,\nu_0\delta^{1/3}})\right] \Gamma^{2}_t(h,\Tilde{\pi}^{\delta})\right)\,d\gamma^{2}_t (h,\Tilde{\pi}^\delta) =  0.
\end{equation*}
As in Section \ref{subsubsect:thmpropmeasuresi}, we deduce from the previous equation that the operator $\Gamma_{t}^2 (h,\Tilde{\pi}^\delta)$ is $d\gamma_{t}^2$-a.e diagonal with respect to the operator $H(\pi^{\delta,\nu_0 \delta^{1/3}})\otimes {\rm Id}_{L^2(\R_2)}$. Being compact, 
we can decompose $\Gamma_{t}^2$ as the following sum
\begin{equation*}
    \Gamma_{t}^2 = \sum_{n\in\N^*} \Gamma_{n,t}^2,
\end{equation*}
where
\begin{equation*}
    \forall (h,\Tilde{\pi}^\delta)\in\Heis\times\Hhat,\ \Gamma_{n,t}^2 (h,\Tilde{\pi}^\delta)= (\Pi_{n}(\pi^{\delta,\nu_0 \delta^{1/3}})\otimes {\rm Id})\Gamma_{t}^2(h,\Tilde{\pi}^\delta)(\Pi_{n}(\pi^{\delta,\nu_0 \delta^{1/3}})\otimes {\rm Id}). 
\end{equation*}

We focus now on the $m$-th mode $\Gamma_{m,t}^{2}d\gamma^{2}_t$, where we assume that $\nu_0$ is a non-degenerate critical point of $\nu \mapsto \Tilde{\mu}_m(\nu)$.

We take $\sigma$ a second-microlocal symbol commuting with the symbol of the suplaplacian $H$, and in particular satisfying $\Pi_m \sigma \Pi_m = \sigma$. 
Equation \eqref{eq:ell2} simply becomes 
\begin{equation*}
    -i\hbar \,\ell_{\hbar}^2(\theta',\sigma) = 2\ell_{\hbar}^{2}(\theta,(\pi(V)\cdot V)\sigma) + \hbar \ell_{\hbar}^2 (\theta,\Delta_{\Eng}\sigma).
\end{equation*}

As previously, we write the decomposition of $(\pi(V)\cdot V)\sigma$ as its $H$-diagonal part and its $H$-off-diagonal one: by Lemma \ref{lem:first_diag_part} we have
\begin{equation*}
    (\pi(V)\cdot V)\sigma(x,\pi^{\delta,\beta},\eta) = \frac{i}{2}\partial_{\beta}\mu_m(\delta,\beta) X_2 \sigma(x,\pi^{\delta,\beta},\eta) + \frac{1}{2}\left[\sigma_{1}(x,\pi^{\delta,\beta},\eta),H(\pi^{\delta,\beta})\right],
\end{equation*}
with 
\begin{equation*}
    \sigma_{1}(x,\pi^{\delta,\beta},\eta) = -\frac{i}{\delta} \pi^{\delta,\beta}(X_3) X_1\sigma(x,\pi^{\delta,\beta},\eta) + i\partial_\beta \Pi_m(\pi^{\delta,\beta}) X_2\sigma(x,\pi^{\delta,\beta},\eta).
\end{equation*}
We perform now the following Taylor expansion:
\begin{equation*}
    \partial_\beta \mu_{m}(\delta,\beta) = \delta^{1/3} \Tilde{\mu}_m '(\beta\delta^{-1/3}) = \delta^{1/3} \Tilde{\mu}_m ' (\nu_0) + (\beta - \nu_0 \delta^{1/3}) \Tilde{\mu}_m ''(\nu_0) + (\beta-\nu_0\delta^{1/3})^2 \delta^{-2/3}r(\delta,\beta),
\end{equation*}
where we have introduced the map $r$ as remainder. As $\nu_0$ is a critical point for $\Tilde{\mu}_m$, we have 
\begin{multline*}
    \partial_{\beta}\mu_m(\delta,\beta) X_2 \sigma_{\hbar}^{C_{\nu_0}}(x,\pi^{\delta,\beta}) =  \hbar \,\Tilde{\mu}_m ''(\nu_0) \frac{\beta - \nu_0 \delta^{1/3}}{\hbar} X_2 \sigma_{\hbar}^{C_{\nu_0}}(x,\pi^{\delta,\beta})\\ + \hbar^{2} \delta^{-2/3}r(\delta,\beta)\left(\frac{\beta - \nu_0 \delta^{1/3}}{\hbar}\right)^2 X_2 \sigma_{\hbar}^{C_{\nu_0}}(x,\pi^{\delta,\beta})\\
    = \hbar \,\Tilde{\mu}_m ''(\nu_0) (\eta X_2 \sigma)_{\hbar}^{C_{\nu_0}}(x,\pi^{\delta,\beta}) + \hbar^{2} \delta^{-2/3}r(\delta,\beta) X_2 (\eta^2 \sigma)_{\hbar}^{C_{\nu_0}}(x,\pi^{\delta,\beta}).
\end{multline*}
Thus we obtain 
\begin{equation*}
    -i\hbar\,\ell_{\hbar}^{2}(\theta',\sigma) = i\hbar\, \ell_{\hbar}^{2}(\theta,\Tilde{\mu}_m ''(\nu_0) \eta X_2\sigma) + \ell_{\hbar}^{2}(\theta,\left[\sigma_{1},H\right]) + \hbar\, \ell_{\hbar}^{2}(\theta,\Delta_{\Eng}\sigma) + \cO(\hbar^2).
\end{equation*}
Using once again Equation \eqref{eq:ell2} with the symbol $\sigma_{1}$, we finally obtain 
\begin{equation}
    \label{eq:propagGamma2}
    -i\ell_{\hbar}^{2}(\theta',\sigma) =  i\ell_{\hbar}^{2}(\theta,\Tilde{\mu}_m ''(\nu_0) \eta X_2\sigma) -\ell_{\hbar}^{2}(\theta,2(\pi(V)\cdot V)\sigma_{1} - \Delta_{\Eng}\sigma) + \cO(\hbar).
\end{equation}

We are left with computing the $H$-diagonal part of the symbol $2(\pi(V)\cdot V)\sigma_1 - \Delta_{\Eng}\sigma$.

\begin{lemma}
    \label{lem:sec_diag_part}
    There exists $\sigma_2\in\cA_{0}^2(\Eng)$ such that
    \begin{equation*}
        2(\pi(V)\cdot V)\sigma_1 - \Delta_{\Eng}\sigma = \frac{1}{\delta} \partial_\beta \mu_m\Pi_m \pi(X_3)\Pi_m X_1 X_2\sigma - \frac{\partial_{\beta}^{2}\mu_m}{2} X_{2}^{2}\sigma
        + \frac{1}{2}[\sigma_2,H].
    \end{equation*}
\end{lemma}

\begin{proof}
    The diagonal part follows easily from Lemmata \ref{lem:diagpartsigma1_1} and \ref{lem:diagpartsigma1_2}, where one must pay attention to the definition of $\sigma_1$ as there is a sign difference. 
    The existence of the symbol $\sigma_2$ is proved using the same arguments as in Lemma \ref{lem:first_diag_part} or as in Section \ref{subsubsect:wpOhbar} in the proof of Theorem \ref{thm:prop_wave_packet}.
\end{proof}

Hence, as $\partial_\beta \mu_m(\delta,\nu_0 \delta^{1/3})=0$ and $\partial_{\beta}^{2}\mu_m(\delta,\nu_0 \delta^{1/3})= \Tilde{\mu}_{m}''(\nu_0)$ for all $\delta\in\fg_{3}^*\setminus\{0\}$, taking the limit $\hbar_k \rightarrow 0$ we obtain 
\begin{multline*}
    -i\ell_{\infty}^{2,c}(\theta',\sigma(h,x_2,\pi^{\delta,\nu_0 \delta^{1/3}},D_{x_2})) =\\ \ell_{\infty}^{2,c}\left(\theta, \frac{\mu''(\nu_0)}{2}\left(2(i\eta X_2\sigma)(h,x_2,\pi^{\delta,\nu_0 \delta^{1/3}},D_{x_2}) + (X_{2}^{2}\sigma)(h,x_2,\pi^{\delta,\nu_0 \delta^{1/3}},D_{x_2})\right)\right).
\end{multline*}

By simple Kohn-Nirenberg symbolic calculus on $\R$, as the vector field $X_2$ simply acts as $\partial_2$, this equation is equivalent to
\begin{equation*}
    -i\ell_{\infty}^{2,c}(\theta',\sigma(h,x_2,\pi^{\delta,\nu_0 \delta^{1/3}},D_{x_2})) = \ell_{\infty}^{2,c}\left(\theta, \left[\sigma(h,x_2,\pi^{\delta,\nu_0 \delta^{1/3}},D_{x_2}), \frac{\mu''(\nu_0)}{2}D_{x_2}^{2}\right]\right),
\end{equation*}
from which we deduce the weak Heisenberg equation satisfied by $t\mapsto d\gamma_{m,t}^2$.

\subsubsection{Proof of Proposition \ref{prop:continuityGamma2}}

For any $\sigma \in \cA_{0}^2(\Eng)$ with compact support in $\eta$ and satisfying $\Pi_m \sigma \Pi_m = \sigma$, we write
\begin{equation*}
    \ell_{\hbar,t}^{2,c}(\sigma) = \left(\Op_{\hbar}^{C_{\nu_0}}(\sigma)\psi^{\hbar}(t),\psi^{\hbar}(t)\right)_{L^2(\Eng)}.
\end{equation*}
Observe that by Equation \eqref{eq:propagGamma2}, we have 
\begin{equation*}
    \frac{d}{dt} \ell_{\hbar,t}^{2,c}(\sigma) = \cO_{\sigma}(1),
\end{equation*}
in the sense that the map $t\mapsto \frac{d}{dt}\ell_{\hbar,t}(\sigma)$ is uniformly bounded in a given bounded interval of $\R$ and this uniformly for $\hbar\in (0,1)$.
Thus the family of maps $(\ell_{\hbar,\cdot}^{2,c}(\sigma))_{\hbar>0}$ is bounded and equicontinuous in $C_{b}(\R)$, so that by Arzelà-Ascoli theorem we can extract 
a subsequence $(\hbar_{j_k})_{k\in\N}$ such that, as $k\rightarrow +\infty$, $\hbar_{j_k} \rightarrow 0$ and $(\ell_{\hbar_{j_k},\cdot}^{2,c}(\sigma))_{k\in\N}$ converges uniformly 
on any bounded interval to a locally Lipshitz function $t\mapsto \ell_{t}^{2,c}(\sigma)$ for all $\sigma \in \cA_{0}^2(\Eng)$ with compact support in $\eta$ and satisfying $\Pi_m \sigma \Pi_m = \sigma$
(this last assertion follows from considering a dense subset of such symbols and a diagonal extraction procedure).

To conclude to the equality
\begin{equation*}
    \forall t\in\R,\ \ell_{t}^{2,c}(\sigma) = \int_{\Heis\times\Hhat} {\rm Tr}\left(\sigma(h,x_2,\pi^{\delta,\nu_0\delta^{1/3}},D_{x_2})\Gamma_{t}^2(x,\Tilde{\pi}^\delta)\right)d\gamma_{t}^2(x,\Tilde{\pi}^\delta),
\end{equation*}
we can proceed as in the proof of Theorem \ref{thm:mesures2micro}, observing that actually the linear form $\ell_{t}^{2,c}(\sigma)$ depends of the symbol 
$\sigma(\cdot,x_2,\cdot,D_{x_2})\in \cA_{0}^{L^2(\R_2)}(\Heis)$ and thus define a state on $\cA_{0}^{L^2(\R_2)}(\Heis)$. As their characterization has been given 
in Theorem \ref{thm:mesuresemiopvalued} in Appendix \ref{sect:Hilbertvaluedsemiclass}, it is given by an operator-valued measure in $\cM_{ov}^+\left(\Heis\times\Hhat, \left(\cH_{\Tilde{\pi}}\otimes L^2(\R_2)\right)_{(x,\Tilde{\pi})\in \Heis\times\Hhat}\right)$.
Finally, by construction of $\ell_{\cdot}^{2,c}$, these measures necessarily coincide with the map $t \mapsto \Gamma_{m,t}^2 d\gamma_{t}^2$, and 
we obtain the desired result.

\subsubsection{Invariance of the spherical semiclassical measures}
We recall that the spherical part of the semiclassical measure, noted $\Gamma_{t}^\infty d\gamma_{t}^\infty$, is obtained as follows: for $\sigma \in \cA_{0}^2(\Eng)$,
\begin{equation*}
    \int_{\R} \theta(t)\int_{C_{\nu_0}\times\bS^{0}} {\rm Tr}\left(\sigma_\infty(x,\pi,\omega)\Gamma_{t}^{\infty}(x,\pi,\omega)\right) d\gamma_{t}^{\infty}(x,\pi,\omega)dt
    = \lim_{\eps\rightarrow 0}\lim_{R\rightarrow +\infty}\lim_{\hbar_k \rightarrow 0} \ell_{\hbar}^2(\theta,\sigma_{\eps}^{R}),
\end{equation*}
where we have introduced the notation
\begin{equation*}
    \forall (x,\pi^{\delta,\beta},\eta)\in\Eng\times\Enghat\times\R,\ \sigma_{\eps}^R(x,\pi^{\delta,\beta},\eta) = \sigma(x,\pi^{\delta,\beta},\eta)\chi\left(\frac{\beta-\nu_0 \delta^{1/3}}{\eps}\right)\left(1-\chi\left(\frac{\eta}{R}\right)\right).
\end{equation*}
Similarly as previously, Equation \eqref{eq:ell2} gives straightforwardly that 
\begin{equation*}
    \ell_{\hbar}^2(\theta,\left[\sigma_{\eps}^R,H\right]) = \cO(\hbar),
\end{equation*}
and thus, for almost every $t\in\R$,
\begin{equation*}
    \int_{C_{\nu_0}\times\bS^{0}} {\rm Tr}\left(\left[\sigma_\infty(x,\pi,\omega),H(\pi)\right]\Gamma_{t}^{\infty}(x,\pi,\omega)\right) d\gamma_{t}^{\infty}(x,\pi,\omega)dt = 0.
\end{equation*}
And we deduce that similarly as before, the operator $\Gamma_{t}^{\infty}$ decomposes for almost every $t\in\R$ as follows
\begin{equation*}
    \forall (x,\pi)\in C_{\nu_0},\ \Gamma_{t}^{\infty}(x,\pi,\omega) = \sum_{n\in\N^*} \Gamma_{n,t}^{\infty}(x,\pi,\omega),
\end{equation*}
where for $n\in\N_{>0}$,
\begin{equation*}
    \Gamma_{n,t}^{\infty}(x,\pi^{\delta,\nu_0 \delta^{1/3}},\omega) = \Pi_{n}(\pi^{\delta,\nu_0 \delta^{1/3}})\Gamma_{t}^{\infty}(x,\pi^{\delta,\nu_0 \delta^{1/3}},\omega)\Pi_{n}(\pi^{\delta,\nu_0 \delta^{1/3}}).
\end{equation*}

We focus now on the $m$-th mode $\Gamma_{m,t}^{\infty} d\gamma_{m,t}^{\infty}$ of the spherical semiclassical measure.
Let $\sigma\in\cA_{0}^2(\Eng)$ commuting with the symbol of the subLaplacian $H$. 

We introduce the following family of symbols: for $\hbar>0$,
\begin{equation*}
    \tau_{\hbar}^{\eps,R} = \hbar^{-1}\left(\eta^{-1}\sigma_{\eps}^{R}\right)_{\hbar}^{C_{\nu_0}}.
\end{equation*}
By Proposition \ref{prop:2microbounded}, seeing that $\eta^{-1}\sigma_{\eps}^{R}$ is a second-microlocal symbol thanks to the cut-off in the variable $\eta$, we have 
that $\tau_{\hbar}^{\eps,R}$ belongs to the class of symbol $S_{0,0}^{0}$ and we have the following estimate on the norm of its quantization:
\begin{equation}
    \label{eq:estim_norm_tau}
    \Op_{\hbar}(\tau_{\hbar}^{\eps,R}) = \cO_{L^2(\Eng)\rightarrow L^2(\Eng)}(1/(\hbar R)).
\end{equation}

By symbolic calculus, we have the following equality:
\begin{equation*}
    [\Op_{\hbar}(\tau_{\hbar}^{\eps,R}),-\hbar^2 \Delta_\Eng] = 2\hbar\Op_{\hbar}((\pi(V)\cdot V)\tau_{\hbar}^{\eps,R}) + \hbar^2 \Op_{\hbar}(\Delta_\Eng \tau_{\hbar}^{\eps,R}).
\end{equation*}

By Lemma \ref{lem:first_diag_part} and Equation \eqref{eq:antidiagT}, we have the following decomposition 
\begin{equation*}
    (\pi(V)\cdot V)\tau_{\hbar}^{\eps,R}(x,\pi^{\delta,\beta}) = \frac{i}{2} \partial_{\beta} \mu_m(\delta,\beta) X_2 \tau_{\hbar}^{\eps,R}(x,\pi^{\delta,\beta}) + \frac{1}{2}\left[\tau_{1,\hbar}^{\eps,R}(x,\pi^{\delta,\beta}),H(\pi^{\delta,\beta})\right],
\end{equation*}
with 
\begin{equation*}
    \tau_{1,\hbar}^{\eps,R} = \frac{1}{i\delta} \pi^{\delta,\beta}(X_3) X_1 \tau_{\hbar}^{\eps,R} +  i\partial_{\beta}\Pi_m(\pi^{\delta,\beta}) X_2 \tau_{\hbar}^{\eps,R} \in S_{0,0}^0.
\end{equation*}

Once again, the $H$-diagonal part of $(\pi(V)\cdot V)\tau_{\hbar}^{\eps,R}$ is the one we are interesting about. We perform a Taylor expansion on $\mu_m$:
\begin{equation*}
    \frac{1}{2}\partial_{\beta}\mu_{m}(\delta,\beta) X_2 \tau_{\hbar}^{\eps,R}(x,\pi^{\delta,\beta}) = \frac{1}{2} \Tilde{\mu}_{m}''(\nu_0) \frac{\beta-\nu_0\delta^{1/3}}{|\beta-\nu_0\delta^{1/3}|} X_2 \sigma_{\eps}^{R}\left(x,\pi^{\delta,\beta},\frac{\beta-\nu_0\delta^{1/3}}{\hbar}\right) + \cO(\eps),
\end{equation*}
where the remainder is caused by the cut-off at scale $\eps$ close to the cone $C_{\nu_0}$. 

Observe that this equality implies
\begin{equation*}
    \begin{aligned}
    &\lim_{\eps\rightarrow 0}\lim_{R\rightarrow +\infty}\lim_{\hbar_k \rightarrow 0} \int_{\R}\theta(t)\left(\Op_\hbar\left(\frac{1}{2}\partial_{\beta}\mu_{m} X_2\tau_{\hbar}^{\eps,R}\right)\psi^\hbar(t),\psi^\hbar(t)\right)_{L^2(\Eng)}dt =\\
    &\int_{\R\times C_{\nu_0}\times \bS^0} \theta(t){\rm Tr}\left(\frac{1}{2} \Tilde{\mu}_{m}''(\nu_0)X_2 \sigma_\infty(x,\pi,\omega)\Gamma_{t}^\infty(x,\pi,\omega)\right) d\gamma_{t}^\infty(x,\pi,\omega)dt.
    \end{aligned}
\end{equation*}

Thus, writing
\begin{equation*}
    \ell_\hbar(\theta,\tau_{\hbar}^{\eps,R}) = \int_{\R}\theta(t)\left(\Op_\hbar(\tau_{\hbar}^{\eps,R})\psi^\hbar(t),\psi^\hbar(t)\right)_{L^2(\Eng)}dt,
\end{equation*}
in order to show the invariance of the measure $d\gamma_{m,t}^{\infty} = \Gamma_{m,t}^{\infty}d\gamma_{t}^{\infty}$ for the flow 
\begin{equation*}
    (h,x_2,\pi^{\delta,\nu_0 \delta^{1/3}},\omega) \mapsto \left(h,x_2 + \frac{s}{2}\Tilde{\mu}_{m}''(\nu_0)\omega,\pi^{\delta,\nu_0 \delta^{1/3}},\omega\right),
\end{equation*}
we are left to prove that
\begin{equation}
    \label{eq:invariance_infty}
    \lim_{\eps\rightarrow 0}\lim_{R\rightarrow +\infty}\lim_{\hbar_k \rightarrow 0} \ell_{\hbar_k}\left(\theta,\frac{1}{2}\partial_{\beta}\mu_{m} X_2 \tau_{\hbar_k}^{\eps,R}\right) = 0.
\end{equation}

To this end, observe that
\begin{equation*}
    \begin{aligned}
    \Op_{\hbar}(i\partial_{\beta}\mu_{m} X_2 \tau_{\hbar}^{\eps,R}) &= 2\Op_{\hbar}((\pi(V)\cdot V)\tau_{\hbar}^{\eps,R}) - \Op_{\hbar}\left(\left[\tau_{1,\hbar}^{\eps,R},H\right]\right)\\
    &= \frac{1}{\hbar}\left[\Op_{\hbar}(\tau_{\hbar}^{\eps,R}),-\hbar^2 \Delta_\Eng\right] -\Op_{\hbar}\left(\left[\tau_{1,\hbar}^{\eps,R},H\right]\right) -\hbar\Op_{\hbar}(\Delta_\Eng \tau_{\hbar}^{\eps,R})
    \end{aligned}
\end{equation*}
Now, doing the same analysis for the symbol $\tau_{1,\hbar}^{\eps,R}$, one can also write that
\begin{equation*}
    -\ell_\hbar\left(\theta,\left[\tau_{1,\hbar}^{\eps,R},H\right]\right) = \hbar^2\ell_{\hbar}(\theta',\tau_{1,\hbar}^{\eps,R}) + \hbar \ell_\hbar(\theta,2(\pi(V)\cdot V)\tau_{1,\hbar}^{\eps,R}) + \hbar^2  \ell_\hbar(\theta,\Delta_\Eng \tau_{1,\hbar}^{\eps,R}),
\end{equation*}
and thus using the operator norm estimate \eqref{eq:estim_norm_tau},
\begin{equation*}
    -\ell_\hbar\left(\theta,\left[\tau_{1,\hbar}^{\eps,R},H\right]\right) - \hbar \ell_\hbar(\theta,\Delta_\Eng \tau_{\hbar}^{\eps,R}) = \hbar \ell_\hbar(\theta,2(\pi(V)\cdot V)\tau_{1,\hbar}^{\eps,R}-\Delta_\Eng\tau_{\hbar}^{\eps,R}) + \cO(\hbar)
\end{equation*}
Using Lemma \ref{lem:sec_diag_part}, we decompose the symbol $2(\pi(V)\cdot V)\tau_{1,\hbar}^{\eps,R}-\Delta_\Eng \tau_{\hbar}^{\eps,R}$ as its $H$-diagonal and $H$-off-diagonal part:
\begin{equation*}
    2(\pi(V)\cdot V)\tau_{1,\hbar}^{\eps,R} - \Delta_{\Eng}\tau_{\hbar}^{\eps,R} = \frac{1}{\delta} \partial_\beta \mu_m\Pi_m \pi^{\delta,\beta}(X_3)\Pi_m\tau_{\hbar}^{\eps,R} + \frac{\partial_{\beta}^{2}\mu_m}{2} X_{2}^{2}\tau_{\hbar}^{\eps,R}
        + \frac{1}{2}[\tau_{2,\hbar}^{\eps,R},H],
\end{equation*}
for some computable symbol $\tau_{2,\hbar}^{\eps,R}$ in $S_{0,0}^0$.
Using once again the symbolic calculus, the off-diagonal part contributes to a $\cO(\hbar)$, while for the diagonal part, we make use of similar estimation as \eqref{eq:estim_norm_tau}
and we obtain a contribution of order $\cO(1/R)$. 

At the end, we obtain
\begin{equation*}
    \Op_{\hbar}(i\partial_{\beta}\mu_{m} X_2 \tau_{\hbar}^{\eps,R}) 
    = \frac{1}{\hbar}\left[\Op_{\hbar}(\tau_{\hbar}^{\eps,R}),-\hbar^2 \Delta_\Eng\right] + \cO(1/R) + \cO(\hbar).
\end{equation*}
Therefore,
\begin{equation*}
    \begin{aligned}
        \ell_{\hbar}\left(\theta,\partial_{\beta}\mu_{m} X_2 \tau_{\hbar}^{\eps,R}\right) 
        &= \int_{\R}\theta(t)\left(\Op_{\hbar}(\partial_{\beta}\mu_{m} X_2\tau_{\hbar}^{\eps,R})\psi^{\hbar}(t),\psi^{\hbar}(t)\right)dt\\
        &= -\frac{i}{\hbar} \int_{\R}\theta(t)\left(\left[\Op_{\hbar}(\tau_{\hbar}^{\eps,R}),-\hbar^2 \Delta_\Eng\right]\psi^{\hbar}(t),\psi^{\hbar}(t)\right)dt + \cO(1/R) + \cO(\hbar)\\
        &= \hbar\int_{\R}\theta(t)\frac{d}{dt}\left(\Op_{\hbar}(\tau_{\hbar}^{\eps,R})\psi^{\hbar}(t),\psi^{\hbar}(t)\right)dt+ \cO(1/R) + \cO(\hbar)\\
        &= -\hbar \int_{\R}\theta'(t)\left(\Op_{\hbar}(\tau_{\hbar}^{\eps,R})\psi^{\hbar}(t),\psi^{\hbar}(t)\right)dt+ \cO(1/R) + \cO(\hbar)\\
        &= \cO(1/R) + \cO(\hbar),
    \end{aligned}
\end{equation*}
which gives \eqref{eq:invariance_infty}. We conclude knowing that $d\gamma_{m,t}^\infty$ has finite mass and the flow for which we have proved invariance is non-zero as $\nu_0$ is a non-degenerate 
critical point for $\Tilde{\mu}_m$.

\subsection{Application: obstruction to Strichartz estimates and local smoothing effect}
\label{subsect:obsdisproofs}

We apply here the results of the previous section to the problem of the dispersion of the subLaplcian $\Delta_\Eng$.

\subsubsection{Proof of Theorem \ref{thm:obstructionsmoothing}}
    We argue by contradiction and suppose that the estimate \eqref{eq:smoothingEng} holds for some choice of $\eps>0$, $s>0$, $C>0$ and for some bounded open set $\Omega$. Let consider the 
    initial datum $(u_{0}^\hbar)_{\hbar>0}$ given by Equation \eqref{eq:example2micro} in Section \ref{subsubsec:example2micro} where we choose $h_0 = 0$, $\eta_0 = 0$ and $\nu_0\in\R$ to be the unique critical point of 
    the first eigenvalue $\nu\mapsto\Tilde{\mu}_{1}(\nu)$. We also choose $\Phi$ to be in the eigenspace of $H(\pi^{\delta_0,\nu_0\delta_{0}^{1/3}})$ corresponding to the eigenvalue $\mu_1(\delta_0,\nu_0\delta_{0}^{1/3})$.
    As already seen, it is a bounded family in $L^2(\Eng)$ converging weakly to zero as $\hbar$ goes to zero. Indeed, we check that for $\psi\in C_{c}(\Eng)$, we have 
    \begin{equation*}
        \begin{aligned}
            \int_{\Eng}u_{0}^{\hbar}(x)\psi(x)\,dx &= \hbar^{-3/2} \int_{\Eng}a(\hbar^{-1/2}\cdot h_x)\varphi(x_2)\left(\pi^{\delta_0,\nu_0\delta_{0}^{1/3}}(\hbar^{-1}\cdot x)\Phi,\Phi\right)\psi(x)\,dx \\
            &= \hbar^{3/2}\int_{\Eng}a(h_x)\varphi(x_2)\left(\pi^{\delta_0,\nu_0\delta_{0}^{1/3}}(\hbar^{-1/2}\cdot h_x, \hbar^{-1}\cdot x_2)\Phi,\Phi\right)\psi(\hbar^{1/2}\cdot h_x, x_2)\,dx,
        \end{aligned}
    \end{equation*}
    and, as $\hbar \rightarrow 0$, the last term converges to zero.
    
    Estimate \eqref{eq:smoothingEng} then implies that the solution to the Schrödinger equation $(u^\hbar(\cdot))_{\hbar >0}$ is a bounded family in $L^2((0,\eps),\Tilde{H}^s(\Omega))$, where the Hilbert spaces $\left(\Tilde{H}^s(\Omega)\right)_{s\in\R}$ are
    the Sobolev spaces adapted to the Engel group $\Eng$, see \cite{RS} and \cite{FR}.
    Moreover, since $u^\hbar$ is a solution of \eqref{eq:SchEngel} with $\tau=2$, i.e we have 
    \begin{equation*}
        i\partial_t u^\hbar = -\Delta_{\Eng}u^{\hbar},
    \end{equation*}
    thus we can write that $\partial_t u^{\hbar} \in L^2((0,\eps),\Tilde{H}^{-2}(\Omega))$. Since $\Omega$ is a bounded domain, by Rellich's theorem, the injection of $\Tilde{H}^{s}(\Omega)$ in $L^2(\Omega)$ is compact
    and by the Aubin-Lions lemma, there exists a subsequence $(u^{\hbar_k})_{k\in\N}$, with $\hbar_k \Tend{k}{+\infty} 0$,
    converging strongly in $L^2((0,\eps)\times\Omega)$: the limit is necessarily zero as it is already converging weakly to zero. 
    Indeed, for $\theta \in \C_{c}(\R)$ and $\psi\in C_{c}(\Eng)$ we have 
    \begin{equation*}
        \begin{aligned}
        \int_{0}^{\eps}\theta(t)\left(u^{\hbar}(t),\psi\right)_{L^2(\Omega)}dt &= \int_{0}^{\eps}\theta(t)\left(u_{0}^{\hbar},e^{-it\Delta_{\Eng}}({\bold 1}_{\Omega}\psi)\right)_{L^2(\Eng)}dt\\
        &= \left(u_{0}^{\hbar},\int_{0}^{\eps}\theta(t)e^{-it\Delta_{\Eng}}({\bold 1}_{\Omega}\psi)\,dt\right)_{L^2(\Eng)},
        \end{aligned}
    \end{equation*}
    and we conclude by weak convergence of $(u_{0}^\hbar)_{\hbar>0}$ to zero.

    As a consequence, we obtain that any time-averaged semiclassical measure $t\mapsto \Gamma_t d\gamma_t$ must necessarily be equal to zero for almost every $t\in\R$.
    To compute $\Gamma_t d\gamma_t$ for all time $t\in\R_t$, we use the existence, up to a further extraction, of a time-averaged second microlocal semiclassical measures 
    $(t\mapsto \Gamma_{t}^2 d\gamma_{t}^2,t\mapsto \Gamma^{\infty}_t d\gamma^{\infty}_t)$. 
    By Theorem \ref{thm:propagmesures2micro}, $\Gamma^{\infty}_t d\gamma^{\infty}_t=0$ for almost all time $t$ and $\Gamma_{t}^2 d\gamma_{t}^2$ satisfies a weak Heisenberg equation. 
    Using Proposition \ref{prop:continuityGamma2} and Proposition \ref{prop:example2micro} for the computation of the second-microlocal measures of $(u_{0}^\hbar)_{\hbar>0}$, the Heisenberg equation can be integrated and we find that, for all $t\in\R$, we have 
    \begin{equation*}
        d\gamma_{t}^2 = \delta_{\{h = 0\}} \otimes \delta_{\{\Tilde{\pi}^\delta = \Tilde{\pi}^{\delta_0}\}},\quad \Gamma_{t}^2 = \ket{\Phi}\bra{\Phi}\otimes \ket{e^{it\frac{\Tilde{\mu}_{1}''(\nu_0)}{2}\partial_{x_2}^2}\varphi}\bra{e^{it\frac{\Tilde{\mu}_{1}''(\nu_0)}{2}\partial_{x_2}^2}\varphi}.
    \end{equation*}
    Thus the full semiclassical measure $\Gamma_t d\gamma_t$ is equal to
    \begin{equation}
        \label{eq:fullsemimeasuhbar}
        d\gamma_t = \delta_{\{h = 0\}}\otimes |e^{it\frac{\Tilde{\mu}_{1}''(\nu_0)}{2}\partial_{x_2}^2}\varphi|^2 dx_2 \otimes \delta_{\{(\delta,\beta) = (\delta_0,\nu_0\delta_{0}^{1/3})\}},\quad \Gamma_t = \ket{\Phi}\bra{\Phi},
    \end{equation}
    which is not uniformly equal to zero and we obtain a contradiction.

\subsubsection{Proof ot Theorem \ref{thm:refinedobstructionsmoothing}}
    We argue once again by contradiction and suppose that the estimate \eqref{eq:smoothingEng} holds for some choice of $\eps>0$, $s>0$, $C>0$ and for some bounded open set $\Omega$, which we suppose to contain 0. Let consider the 
    initial datum $(u_{0}^\hbar)_{\hbar>0}$ as in the proof of Theorem \ref{thm:obstructionsmoothing} and write $u^\hbar(t)$ the associated solution of the Schrödinger equation. Choose $\chi_{\Omega}\in C_{c}^{\infty}(\Eng)$ a smooth cutoff with support in $\Omega$ and satisfying $\chi_\Omega(0)=1$, and 
    a smooth function $\theta\in C_{c}^{\infty}(\R)$ uniformly equal to zero in a neighborhood of 0. We observe that 
    \begin{equation*}
        \begin{aligned}
        \int_{0}^{\eps}\left(\chi_\Omega \theta(i\hbar X_2)u^{\hbar}(t),u^{\hbar}(t)\right)\,dt &= \int_{0}^{\eps}\left(\chi_\Omega \theta(i\hbar X_2)|\hbar X_2|^{-s}|\hbar X_2|^{s}u^{\hbar}(t),u^{\hbar}(t)\right)\,dt\\
        &= \int_{0}^{\eps}\left(\Op_{\hbar}\left(\chi_\Omega \Tilde{\theta}(i\pi(X_2))\right)|\hbar X_2|^{s}u^{\hbar}(t),u^{\hbar}(t)\right)\,dt,
        \end{aligned}
    \end{equation*}
    where $\Tilde{\theta}(\beta) = |\beta|^{-s}\theta(\beta)$, and thus $\Tilde{\theta}\in C_{c}^{\infty}(\R)$. Introducing the symbol $\sigma = \chi_\Omega \Tilde{\theta}(i\pi(X_2)) \in \cA_{0}(\Eng)$, we have 
    \begin{equation*}
        \begin{aligned}
        \left|\int_{0}^{\eps}\left(\chi_\Omega \theta(i\hbar X_2)u^{\hbar}(t),u^{\hbar}(t)\right)\,dt \right|
        &= \hbar^{s}\left|\int_{0}^{\eps}\left(\Op_{\hbar}(\sigma)|X_2|^{s}u^{\hbar}(t),u^{\hbar}(t)\right)\,dt\right|\\
        &\leq C \hbar^{s} \int_{0}^{\eps}\||X_2|^{s}\left(e^{it\Delta_\Eng}u_0\right)\|_{L^2(\Omega)}^2\,dt \leq C\hbar^{s} \|u_{0}^{\hbar}\|_{L^2(\Eng)},
        \end{aligned}
    \end{equation*}
    for some constant $C>0$. As $\hbar \rightarrow 0$, if we denote again by $t\mapsto \Gamma_t d\gamma_t$ the time-averaged semiclassical measures of $(u^\hbar)_{\hbar>0}$, we obtain 
    \begin{equation*}
        \int_{0}^{\eps}\int_{\Eng\times\Enghat}{\rm Tr}\left(\chi_\Omega(x)\theta(i\pi(X_2))\Gamma_{t}(x,\pi)\right)d\gamma_{t}(x,\pi)dt = 0.
    \end{equation*}
    As $t\mapsto \Gamma_t d\gamma_t$ has been computed in the previous proof, we find 
    \begin{equation}
        \label{eq:proofsmoothing}
        \langle\theta(i\pi^{\delta_0,\nu_0\delta_{0}^{1/3}}(X_2))\Phi,\Phi\rangle\int_{0}^{\eps} \int_{\R} \chi_\Omega(0,x_2)\left|e^{it\frac{\Tilde{\mu}_{1}''(\nu_0)}{2}\partial_{x_2}^2}\varphi(x_2)\right|^{2}\,dx_2 dt = 0.
    \end{equation}
    Since we can write 
    \begin{equation*}
        \langle\theta(i\pi^{\delta_0,\nu_0\delta_{0}^{1/3}}(X_2))\Phi,\Phi\rangle = \int_\R \theta\left(-\left(\nu_0\delta_{0}^{1/3} + \frac{\delta_0}{2}\xi^2\right)\right)\left|\Phi(\xi)\right|^2\,d\xi,
    \end{equation*}
    it possible to choose $\theta$ such that this quantity is not equal to zero. Equation \eqref{eq:proofsmoothing} gives then a contradiction. 

\subsubsection{Proof of Proposition \ref{prop:obsstrichartz}}

The proof follows once again by arguing by contradiction and by the use of the particular example $(u_{0}^\hbar)_{\hbar>0}$ and the computation of the time-averaged semiclassical measure of the associated solution 
to the Schrödinger equation. Indeed, suppose that the estimate \ref{eq:strichartz} holds for some couple of exponents $(q,p)$ with $q>2$ and $p>2$, then by a duality argument, any weak limit of time-averaged density 
probabilities $(|u^{\hbar}(t,x)|^2 dx dt)_{\hbar>0}$ is in $L^{q/2}(\R_t,L^{p/2}(\Eng))$ and as such cannot be a measure concentrated on a lower dimensional manifold.

However, since the family $(u^\hbar)_{\hbar>0}$ satisfies at all time $t$ the assumption \eqref{eq:assumptionoscillation} of $\hbar$-oscillation, we deduce by Proposition \ref{prop:hbaroscillsemimeas} that any weak limit
of $(|u^{\hbar}(t,x)|^2 dx dt)_{\hbar>0}$ is a marginal of a semiclassical measure. By Equation \eqref{eq:fullsemimeasuhbar}, we obtain a contradiction as we have exhibited 
a concentration phenomenon on the subgroup $\R_2$.


\appendix

\section{Analysis of square-integrable families valued in a Hilbert space}
\label{sect:Hilbertvaluedsemiclass}

Let $G$ be a graded nilpotent Lie group of dimension $d$ and homogeneous dimension $Q$. Let $\cG$ be a separable Hilbert space. In this appendix, we consider the semiclassical 
pseudodifferential operators acting on $L^2(G,\cG)$, the set of square-integrable families valued in $\cG$, i.e measurable functions from 
$G$ into $\cG$ for which one has
\begin{equation*}
    \Vert f \Vert_{L^{2}(G, \cG)} := \int_{G} \Vert f(x) \Vert_{\cG}^{2} \,dx.
\end{equation*}
We shall also consider $\mathcal{S}(G, \cG)$ the space of smooth functions $f$ defined on $G$ and valued in $\cG$ such that
\begin{equation*}
    \forall \alpha,\beta \in \N^{d},\ \exists C_{\alpha,\beta} > 0,\ \sup_{x\in G} \Vert x^{\alpha} \partial_{x}^{\beta} f(x) \Vert_{\cG} \leq C_{\alpha,\beta}.
\end{equation*}

In a first time, we extend the Fourier transform $\cF_G$ to function valued in the Banach algebra $\cK(\cG)$ of compact operators on $\cG$. In a second time, we develop the notion of symbols as well as present their quantization,
and finally we discuss the associated notion of semiclassical measure.

\subsection{Fourier transform}

\subsubsection{Tensor product}

As to make the presentaion clear, we recall and define the two notions of tensor products we will encounter. In what follows we fix a representation $\pi\in\widehat{G}$ and we
write $\cH_\pi$ for the (separable) Hilbert space associated to $\pi$. 

First, the tensor product $\cH_\pi \otimes \cG$ is the Hilbert space defined as the completion of the algebraic tensor product $\cH_\pi \otimes \cG$ 
for the following hermitian product
\begin{equation*}
       \forall (\Phi_1,g_1),(\Phi_2,g_2) \in \cH_\pi \times \cG,\ \langle \Phi_1 \otimes g_1,\Phi_2 \otimes g_2 \rangle_{\cH_\pi \otimes \cG} = \langle \Phi_1,\Phi_2 \rangle_{\cH_\pi} \langle g_1,g_2\rangle_{\cG}.
\end{equation*}

For a matter of notation, we keep the same one for either the algebraic one and its completion. This definition is straightforward as we are dealing here with Hilbert spaces.  However, when one wish to take the tensor product of Banach spaces, choices need to be made.
As it will be our main concern, we only consider here the case of the Banach spaces $\cL(\cH_\pi)$ of bounded operators on $\cH_\pi$ and $\cK(\cG)$ the space of 
compact operators on $\cG$. 

Every element of the algebraic tensor product $\mathcal{L}(\cH_\pi)\otimes \cK(\cG)$ can be identified with a bounded operator on the Hilbert space $\cH_\pi \otimes \cG$.
We denote by $\iota_\pi: \mathcal{L}(\cH_\pi)\otimes \cK(\cG) \rightarrow \mathcal{L}(\cH_\pi \otimes \cG)$ this injection.

\begin{definition}
    We can define a norm on $\mathcal{L}(\cH_\pi)\otimes \cK(\cG)$ by the following formula:
    \begin{equation*}
        \forall x \in \mathcal{L}(\cH_\pi)\otimes \cK(\cG),\ \Vert x\Vert_{\mathcal{L}(\cH_\pi)\otimes \cK(\cG)} = \Vert \iota_\pi(x)\Vert_{\mathcal{L}(\cH_\pi\otimes \cG)}.
    \end{equation*}
    We still denote by $\mathcal{L}(\cH_\pi)\otimes \cK(\cG)$ the completion of the algebraic tensor product by this norm and by $\iota_\pi$ the extended injection.
\end{definition}

We easily check that $\iota_\pi(\mathcal{L}(\cH_\pi)\otimes \cK(\cG))$ is a closed subspace of $\mathcal{L}(\cH_\pi \otimes \cB)$ stable by composition.

\begin{remark}
    Our choice of norm on the algebraic tensor product $\mathcal{L}(\cH_\pi)\otimes \cK(\cG)$ coincide with the projective cross norm. Other cross norms might be possible, but they are all dominated 
    by the one introduced in the previous definition.
\end{remark}

\subsubsection{Extension of $\cF_G$}
Let $f \in \mathcal{S}(\Heis, \cK(\cG))$. For $g_1, g_2 \in \cG$, the function $x \in G\mapsto f^{g_1,g_2}(x) = \langle f(x)g_1, g_2\rangle_{\cG}$ is in the space of Schwartz functions $\mathcal{S}(G)$.
We can consider its Fourier transform $\cF(f^{g_1,g_2})$ and for $\pi \in \widehat{G}$, we have 
\begin{equation*}
    \Vert \mathcal{F}_G(f^{g_1,g_2})(\pi) \Vert_{\mathcal{L}(\cH_\pi)} \leq \int_{G} |\langle f(x)g_1, g_2\rangle_{\cG}|\,dx.
\end{equation*}
We choose a Hermitian basis $(g_l)_{l\in\N}$ of $\cG$. We define the Fourier tranform of $f$ at $\pi\in\Hhat$ by 
\begin{equation}
    \label{eq:BSVfourier}
    \mathcal{F}_Gf(\pi) := \sum_{i,j\in\N} \mathcal{F}(f^{g_i,g_j})(\pi) \otimes \ket{g_i}\bra{g_j}.
\end{equation}

\begin{proposition}
    \label{prop:BSVfourier}
    For $f \in \mathcal{S}(\Heis, \cK(\cG))$, Equation \eqref{eq:BSVfourier} defines the Fourier transform of $f$ as
    a bounded field of operators $\{\mathcal{F}_Gf(\pi)\in \mathcal{L}(\cH_\pi)\otimes \cK(\cG)\,:\,\pi\in\widehat{G}\}$. Moreover, it does not depend on the choice
    of the Hermitian basis $(g_l)_{l\in\N}$.
\end{proposition}

\begin{proof}
    For $l\in\N$ we write $\bP_l$ for the orthogonal projector on $(g_1,\dots,g_l)$ and $f^{l} = \bP_{l}f \bP_{l}$ the point-wise restriction of $f$ to these spaces. 
    As $f$ is valued in $\cK(\cG)$ we have the pointwise convergence $f^{l}(x) \Tend{l}{\infty} f(x)$ for all $x\in G$: by dominated convergence we easily deduce 
    \begin{equation*}
        \int_{G} \Vert f- f^{l}\Vert_{\cK(\cG)}\,dx \Tend{l}{\infty} 0.
    \end{equation*}
    Observe that for $l,l'\in\N$,
    \begin{equation*}
        \left\lVert \sum_{|i|,|j|\leq l} \mathcal{F}_G(f^{g_i,g_j})(\pi) \otimes \ket{g_i}\bra{g_j}-\sum_{|i|,|j|\leq l'} \mathcal{F}_G(f^{g_i,g_j})(\pi) \otimes \ket{g_i}\bra{g_j} \right\rVert_{\mathcal{L}(\cH_\pi)\otimes \cK(\cG)}
        \leq \int_{\Heis} \Vert f^{l}- f^{l'}\Vert_{\cK(\cG)}\,dx,
    \end{equation*}
    hence, by Cauchy criterion we get the convergence of the serie and the definition of $\cF_G f$ makes sense. Moreover, we get the estimate
    \begin{equation*}
        \sup_{\pi\in\widehat{G}}\lVert \mathcal{F}_Gf(\pi) \rVert_{\mathcal{L}(\cH_\pi)\otimes \cK(\cG)} \leq \lVert f\rVert_{L^1(G,\cK(\cG))}.
    \end{equation*}
    Finally, let $(g'_{k})_{k\in\N}$ a second Hermitian basis of $\cG$. For $i\in\N$, write 
    \begin{equation*}
        g_i = \sum_{j\in\N} a_{i,j} g'_{j},
    \end{equation*}
    with $(a_{i,j})_{j\in\N}$ in $\ell^2(\N)$. We compute that
    \begin{equation*}
        \begin{aligned}
        \sum_{k,l\in\N} \mathcal{F}_G(f^{g_k,g_l})(\pi) \otimes \ket{g_k}\bra{g_l}  
        &= \sum_{k,l\in\N} \left(\sum_{i,j\in\N} a_{k,i}\Bar{a}_{l,j}\mathcal{F}_G(f^{g'_i,g'_j})(\pi)\right) \otimes \left(\sum_{i',j'\in\N} \Bar{a}_{k,i'}a_{l,j'}\ket{g'_{i'}}\bra{g'_{j'}}\right)\\
        &= \sum_{i,j\in\N}\sum_{i',j'\in\N} \left(\sum_{k,l\in\N} a_{k,i}\Bar{a}_{k,i'} a_{l,j'}\Bar{a}_{l,j}\right) \mathcal{F}_G(f^{g'_i,g'_j})(\pi) \otimes \ket{g'_{i'}}\bra{g'_{j'}} \\
        &= \sum_{i,j\in\N} \mathcal{F}_G(f^{g'_i,g'_j})(\pi) \otimes \ket{g'_i}\bra{g'_j}.
        \end{aligned}
    \end{equation*}
\end{proof}

As we do not need a Fourier inversion formula, we will simply mention here that the Fourier transform defined on $\mathcal{S}(\Heis, \cK(\cG))$ is injective.

\subsubsection{Convolution and Fourier transform}

\begin{definition}
    Let $f_1,f_2 \in\mathcal{S}(\Heis, \cK(\cG))$. Define the convolution product of $f_1$ and $f_2$ by 
    \begin{equation*}
        (f_{1} * f_2)(x) = \int_{\Heis} f_{1}(y) \circ f_{2}(y^{-1}x)\,dy.
    \end{equation*}
    Oberve that $f_1 * f_2$ is in $\mathcal{S}(\Heis, \cK(\cG))$.
\end{definition}

\begin{proposition}
    \label{prop:BSVconvol}
    For $f_1,f_2 \in\mathcal{S}(\Heis, \cK(\cG))$, we have for a representation $\pi\in\widehat{G}$:
    \begin{equation*}
        \mathcal{F}_G(f_1 * f_2)(\pi) = \mathcal{F}_G f_2(\pi) \circ \mathcal{F}_G f_1(\pi).
    \end{equation*}
\end{proposition}

\begin{proof}
    Let $(g_l)_{l\in\N}$ be a Hermitian basis of $\cG$. For $i,j\in\N$, we have for $x\in G$,
    \begin{equation*}
        \begin{aligned}
        (f_1 * f_2)^{g_i,g_j}(x) &= \int_{\Heis} \langle f_{1}(y) f_{2}(y^{-1}x) g_i, g_j \rangle_{\cG} \,dy\\
        &= \int_{\Heis} \sum_{l\in\N} \langle f_{2}(y^{-1}x) g_i, g_l\rangle_\cG \langle f_{1}(y) g_l, g_j \rangle_{\cG} \,dy\\
        &= \sum_{l\in\N} \left(f_{1}^{g_l,g_j} * f_{2}^{g_i,g_l}\right)(x).
        \end{aligned}
    \end{equation*}
    Thus, taking the Fourier transform, we have
    \begin{equation*}
        \cF_G (f_1 * f_2)^{g_i,g_j}(\pi) = \sum_{l\in\N} \cF_G (f_{2}^{g_i,g_l})(\pi) \circ \cF_G (f_{1}^{g_l,g_j})(\pi),
    \end{equation*}
    from which we deduce the result of the proposition.
\end{proof}

\subsection{Semiclassical pseudodifferential operators}

\subsubsection{Algebra of smoothing symbols}
We first define the space of symbols we want to quantize. 

\begin{definition}
    \label{def:symbolHilbertvalued}
    We denote by $\cA_{0}^{\cG}(G)$ the space of field of operators $\sigma$ of the form
    \begin{equation*}
       \sigma = \{\sigma(x,\pi) \in \mathcal{L}(\cH_{\pi}) \otimes \cK(\cG)\,:\,(x,\pi) \in G \times \widehat{G} \},
    \end{equation*}
    for which there exists a smooth and compactly supported map $\kappa: x\in G \mapsto \kappa_{x} \in \mathcal{S}(G,\cK(\cG))$ such that
    \begin{equation*}
        \sigma(x,\pi) = \cF_G \kappa_{x}(\pi).
    \end{equation*}
\end{definition}

As the extended Fourier transform is still injective, it yields a one-to-one correspondence between the symbol $\sigma$ and the map $\kappa$. 
As before, we call $\kappa$ the associated convolution kernel to $\sigma$.
Once again, by Proposition \ref{prop:BSVconvol} the set $\cA_{0}^{\cG}(G)$ is an algebra for the composition of symbols.

We can also define the following submultiplicative norm on $\cA_{0}^{\cG}(G)$: for a symbol $\sigma$, we define
\begin{equation*}
    \Vert \sigma\Vert_{\cA^{\cG}(G)} := \sup_{(x,\pi)\in G\times\widehat{G}} \Vert \sigma(x,\pi) \Vert_{\mathcal{L}(\cH_{\pi}) \otimes \cK(\cG)}.
\end{equation*}

\subsubsection{Quantization}
We now define the quantization of these symbols by expliciting their action on Schwartz functions. 
\begin{definition}
    \label{def:quantizationHilbertvalued}
    Let $\sigma \in \cA_{0}^{\cG}(G)$. For $f\in\mathcal{S}(G,\cG)$, we set, for $x\in G$,
\begin{equation*}
    \Op_{\hbar}^{\cG}(\sigma)f(x) = \hbar^{-Q}\int_{G\times\widehat{G}} {\rm Tr}_{\cH_{\pi}}(\pi(\hbar^{-1}\cdot (y^{-1}x))\sigma(x,\pi)f(y)) \,d\mu_{\widehat{G}}(\pi) dy.
\end{equation*}
Then $\Op_{\hbar}^{\cG}(\sigma)f$ is again in $\mathcal{S}(G,\cG)$.
\end{definition}

These operators are called \textit{semiclassical pseudodifferential operators} on $\mathcal{S}(G,\cG)$.

\begin{remark}
    In the precedent quantization formula, the operator $\sigma(x,\pi) \in \mathcal{L}(\cH_{\pi}) \otimes \cK(\cG)$ must be understood as 
    acting on on $0_{\cH_\pi}\otimes f(y)\in \cG$ for $y\in G$. Meanwhile, the trace operator ${\rm Tr}_{\cH_\pi}: \cL^1(\cH_\pi) \rightarrow \C$ 
    is extended to the alegbraic tensor product of $\mathcal{L}^1(\cH_{\pi})$ and $\cK(\cG)$ and into $\cK(\cG)$ in the obvious way, and then extended to 
    the cross-norm completion of $\mathcal{L}^1(\cH_{\pi})\otimes\cK(\cG)$ by continuity.
\end{remark}

Following \cite{FF2},\cite{FF3}, these operators extend their action on the space $L^2(G,\cG)$ as bounded operators:
\begin{equation*}
    \exists C > 0,\ \forall \sigma \in\cA_{0}^{\cG}(G),\ \forall \hbar >0,\ \|\Op_{\hbar}^{\cG}(\sigma)\|_{L^2(G,\cG)} \leq C \int_{G}\sup_{x\in G}\|\kappa_{x}(z)\|_{\cL(\cG)}\,dz.
\end{equation*}

We denote by $\Vert \sigma \Vert_{\cA_{0}^{\cG}(G)}$ the right-hand side of this inequality.

\subsection{Semiclassical measures of a bounded square-integrable family}

As presented in Section \ref{subsect:semiclassicalmeasures}, we describe the structure of the limiting objects of the quadratic quantities 
\begin{equation*}
    \left(\Op_{\hbar}^{\cG}(\sigma)f^{\hbar},f^{\hbar}\right)_{L^2(G,\cG)},
\end{equation*}
where $\sigma \in \cA_{0}^{\cG}(G)$ and $(f^\hbar)_{\hbar>0}$ is a bounded family of $L^2(G,\cG)$.

\begin{theorem}
    \label{thm:mesuresemiopvalued}
    Let $(f^{\hbar})_{\hbar>0}$ be a bounded family in $L^2(G,\cG)$. There exists a sequence $(\hbar_k)_{k\in\N}$ in $(0,+\infty)$ with $\hbar_k\Tend{k}{+\infty}0$
    and a positive operator-valued measure $\Gamma^{\cG} d\gamma^\cG$ in the space 
    \begin{equation*}
        \mathcal M_{ov}^{+} \left(G \times \widehat{G},({\mathcal H}_{\pi} \otimes \cG)_{(x,\pi)\in G \times \widehat{G}}\right),
    \end{equation*}
    such that for all $\sigma \in \cA_{0}^{\cG}(G)$, we have
    \begin{equation*}
        \left(\Op_{\hbar_k}(\sigma)f^{\hbar_k}, f^{\hbar_k}\right)\Tend{k}{+\infty} \int_{\Heis \times \Hhat} {\rm Tr}_{\cH_{\pi} \otimes \cG}(\sigma(x,\pi)\Gamma^{\cG}(x,\pi))\,d\gamma^\cG(x,\pi).
    \end{equation*}
    Given the subsequence $(\hbar_k)_{k\in\N}$, $\Gamma^{\cG} d\gamma^{\cG}$ is the only operator-valued measure for which the above convergence holds.
\end{theorem}

\begin{proof}
    Let $(g_l)_{l\in\N}$ be a Hermitian basis of $\cG$. We write 
    \begin{equation*}
        f^{\hbar}(x) = \sum_{l\in\N} f^{\hbar}_{l}(x) g_{l},\ x\in \Heis,
    \end{equation*}
    where $f^{\hbar}_{l}(x) = \langle f^{\hbar}(x),g_l \rangle_{\cG}$. As we have 
    \begin{equation}
        \label{eq:total_mass}
        \sum_{l\in\N} \Vert f^{\hbar}_l \Vert_{L^2(G)}^2 = \Vert f^{\hbar}\Vert_{L^2(G,\cG)}^2 < +\infty,
    \end{equation}
    each family $(f^{\hbar}_l)_{\hbar>0}$ is bounded in $L^2(G)$. By a diagonal extraction process, we can find a common sequence $(\hbar_k)_{k\in\N}$ and a family 
    of operator-valued measures $(\Gamma_{k,l}d\gamma_{k,l})_{(k,l)\in\N^2}$ in ${\mathcal M}_{ov}(G\times \widehat{G})$ 
    such that for all $\sigma\in\cA_0(G)$ and for all $(k,l)\in\N^2$, we have 
    \begin{equation*}
        \left(\Op_{\hbar}(\sigma)f^{\hbar}_{k},f^{\hbar}_l\right) \Tend{k}{+\infty} \int_{G \times \widehat{G}} {\rm Tr}_{\cH_{\pi}}(\sigma(x,\pi)\Gamma_{k,l}(x,\pi))\,d\gamma_{k,l}(x,\pi),
    \end{equation*}
    with $\Gamma_{l,l}d\gamma_{l,l}$ being positive and satisfying $\Vert\Gamma_{l,l}d\gamma_{l,l}\Vert_{\mathcal{M}} \leq \limsup_{k\rightarrow+\infty} \Vert f^{\hbar_n}_l \Vert_{L^2(G)}$ for $l\in\N$.
    Moreover, by polarization identities one can show that for $k,l\in\N$, $\gamma_{k,l}$ is absolutely continuous with respect to $\gamma_{k,k}+\gamma_{l,l}$.
    
    We construct now our positive measure in $\mathcal M_{ov}^{+} \left(G \times \widehat{G},({\mathcal H}_{\pi} \otimes \cG)_{(x,\pi)\in G \times \widehat{G}}\right)$.
    By Equation \eqref{eq:total_mass}, we know that 
    \begin{equation*}
        \limsup_{k\rightarrow+\infty} \sum_{l\in\N} \Vert f^{\hbar_k}_l\Vert_{L^2(G)}^2 < +\infty,
    \end{equation*}
    and we deduce that we can define a positive Radon measure of finite mass by the formula 
    \begin{equation}
        \gamma^{\cG} = \sum_{l\in\N} \gamma_{l,l}.
    \end{equation}
    By the Radon-Nikodym theorem, there exists for every $(k,l)\in\N^2$ a $\gamma^\cG$-measurable density $a_{k,l}$ such that
    \begin{equation*}
        \gamma_{k,l} = a_{k,l}\gamma^{\cG},
    \end{equation*}
    with $0\leq \|a_{l,l}\| \leq 1$, $d\gamma^{\cG}$-almost everywhere. 
    
    We now define the field of trace class operators $\Gamma^{\cG} = \{\Gamma^{\cG}(x,\pi) \in \mathcal{L}(\cH_{\pi})\otimes\cK(\cG)\,:\, (x,\pi) \in G \times\widehat{G} \}$ by 
    \begin{equation*}
        \Gamma^{\cG}(x,\pi) = \sum_{k,l\in\N} \Gamma_{k,l}(x,\pi) \otimes a_{k,l}\ket{g_k}\bra{g_l}.
    \end{equation*}
    We can check that, for all $(x,\pi)\in G\times\widehat{G}$, $\Gamma^{\cG}(x,\pi)$ is a positive self-adjoint operator
    and satisfies
    \begin{equation*}
        \int_{G\times\widehat{G}} {\rm Tr}_{\cH_\pi \otimes \cG}(\Gamma^{\cG}(x,\pi))\,d\gamma^{\cG}(x,\pi) = 
        \sum_{l\in\N} \int_{G\times\widehat{G}}{\rm Tr}_{\cH_\pi}(\Gamma_{l,l}(x,\pi))a_{l,l}\,d\gamma^{\cG} =
        \sum_{l\in\N} \gamma_{l,l}(G\times\widehat{G}) < +\infty.
    \end{equation*}
    We have thus constructed an element $\Gamma^{\cG}d\gamma^{\cG}$ of $\mathcal M_{ov}^{+} \left(G \times \widehat{G},({\mathcal H}_{\pi} \otimes \cG)_{(x,\pi)\in G \times \widehat{G}}\right)$.
    We need now to prove that this operator-valued has the desired property. 
    
    Let $\sigma \in \mathcal{A}_{0}^{\cG}(G)$. For $l\in\N$, we denote by $\bP_l$ the projector 
    on the finite dimensional subspace of $\cG$ generated by $(g_{k})_{0\leq k \leq l}$. We define the symbol $\sigma_{l} \in \mathcal{A}_{0}^{\cG}(G)$ by 
    \begin{equation*}
        \forall (x,\pi)\in G\times\widehat{G},\ \sigma_l(x,\pi) = \bP_{l}\sigma(x,\pi)\bP_{l}.
    \end{equation*}
    Indeed, if we write $\sigma = \mathcal{F}_G\kappa_{x}$ with $\kappa\in\mathcal{C}^{\infty}_{c}(G,\mathcal{S}(G,\cK(\cG)))$, we simply have $\sigma_l = \mathcal{F}_G\kappa_{x}^{l}$ with 
    \begin{equation*}
        \forall (x,y)\in G^2,\ \kappa_{x}^{l}(y) = \bP_l \kappa_{x}(y) \bP_l.
    \end{equation*}
    Since $\kappa_{x}(y)$ is a compact operator for all $(x,y)\in G^2$ and since $\kappa$ is compact in its first variable and Schwartz in its second one, we easily get that 
    \begin{equation}
        \label{eq:approx_compact}
        \Vert \sigma - \sigma_{l}\Vert_{\mathcal{A}_{0}^{\cG}(G)} \Tend{l}{+\infty} 0.
    \end{equation}
    We observe that 
    \begin{equation*}
        \begin{aligned}
            \left(\Op_{\hbar}(\sigma_l)f^{\hbar},f^{\hbar}\right)_{L^2(G,\cG)} &= \int_{G\times G\times\widehat{G}} {\rm Tr}_{\cH_\pi}\left(\pi(\hbar^{-1}\cdot (y^{-1}x))\langle \sigma_{l}(x,\pi)f^{\hbar}(y),f^{\hbar}(x) \rangle_{\cG}\right) \,d\mu_{\widehat{G}}(\pi) dydx \\
            &= \sum_{k,k' \leq l} \int_{G\times G\times\widehat{G}} {\rm Tr}_{\cH_\pi}\left(\Tilde{\pi}^{\lambda}(\hbar^{-1}\cdot (y^{-1}x))\langle \sigma_{l}(x,\pi)g_k,g_k' \rangle_{\cG}\right)f^{\hbar}_{k}(y)f^{\hbar}_{k'}(x) \,d\mu_{\widehat{G}}(\pi) dydx \\
            &\Tend{\hbar}{0} \sum_{k,k' \leq l} \int_{G \times \widehat{G}} {\rm Tr}_{\cH_{\pi}}\left(\langle \sigma_{l}(x,\pi)g_k,g_k' \rangle_{\cG}\Gamma_{k,k'}(x,\pi)\right)d\gamma_{k,k'}(x,\pi) \\
            &= \sum_{k,k' \leq l} \int_{G \times \widehat{G}} {\rm Tr}_{\cH_{\pi}}\left(a_{k,k'}(x,\pi)\langle \sigma_{l}(x,\pi)g_k,g_k' \rangle_{\cG}\Gamma_{k,k'}(x,\pi)\right)d\gamma^\cG(x,\pi) \\
            &= \int_{G \times \widehat{G}} {\rm Tr}_{\cH_{\pi}\otimes\cG}\left(\sigma_{l}(x,\pi)\Gamma^{\cG}(x,\pi)\right)d\gamma^\cG(x,\pi).
        \end{aligned}
    \end{equation*}
    The conclusion follows by letting $l$ go to $+\infty$: we observe that by the approximation \eqref{eq:approx_compact} and Lebesgue dominated convergence theorem, using that 
    the inequality $\Vert\cdot\Vert_{\mathcal{A}^{\cG}} \leq \Vert\cdot\Vert_{\mathcal{A}^{\cG}_{0}}$, we have
    \begin{equation*}
        \int_{G \times\widehat{G}} {\rm Tr}_{\cH_{\pi}\otimes\cG}\left((\sigma(x,\pi)-\sigma_{l}(x,\pi))\Gamma^{\cG}(x,\pi)\right)d\gamma^\cG(x,\pi) \Tend{l}{+\infty} 0,
    \end{equation*}
    and we deduce the desired convergence.
\end{proof}

\end{document}